\documentclass[A4]{amsart}
\usepackage[square,sort,comma,numbers]{natbib}
\usepackage{amsmath, amssymb, amstext, amsfonts, textcomp, amsxtra, amsbsy, amsgen, amsopn, amscd, mathrsfs, amsthm, latexsym, array}
\usepackage[textwidth=16cm,textheight=22cm,centering]{geometry}

\usepackage{color}
\usepackage{graphicx}

\usepackage[all]{xy}

\newtheorem{theorem}             {Theorem}  [section]
\newtheorem{definition} [theorem] {Definition}
\newtheorem{lemma}      [theorem]{Lemma}
\newtheorem{corollary}  [theorem]{Corollary}
\newtheorem{proposition}[theorem]{Proposition}
\newtheorem{remark} [theorem] {Remark}

\numberwithin{equation}{section} \everymath{\displaystyle}

\newcommand{\sgn}{{\rm sgn}}
\newcommand{\intL}{{\rm L}}

\newcommand{\Nr}{{\rm Nr}}
\newcommand{\Tr}{{\rm Tr}}

\newcommand{\Ram}{{\rm Ram}}

\newcommand{\gp}[1]{\mathbf{#1}}
\newcommand{\GL}{{\rm GL}}
\newcommand{\PGL}{{\rm PGL}}
\newcommand{\SL}{{\rm SL}}
\newcommand{\SO}{{\rm SO}}
\newcommand{\SU}{{\rm SU}}

\newcommand{\ag}[1]{\mathbb{#1}}
\newcommand{\N}{\mathbb{N}}
\newcommand{\Z}{\mathbb{Z}}
\newcommand{\B}{\mathbb{H}}
\newcommand{\Mat}{{\rm M}}

\newcommand{\IntP}[1]{\left\lfloor #1 \right\rfloor}
\newcommand{\IntCP}[1]{\left\lceil #1 \right\rceil}

\newcommand{\Casimir}{\mathcal{C}}

\newcommand{\Q}{\mathbb{Q}}
\newcommand{\R}{\mathbb{R}}
\newcommand{\C}{\mathbb{C}}
\newcommand{\E}{\mathbf{E}}
\newcommand{\F}{\mathbf{F}}
\newcommand{\bL}{\mathbf{L}}
\newcommand{\A}{\mathbb{A}}
\newcommand{\VO}{\mathfrak{O}}
\newcommand{\vO}{\mathcal{O}}
\newcommand{\vo}{\mathfrak{o}}
\newcommand{\vP}{\mathcal{P}}
\newcommand{\vp}{\mathfrak{p}}

\newcommand{\oJ}{\mathfrak{J}}
\newcommand{\oM}{\mathfrak{M}}
\newcommand{\Dis}{{\rm D}}

\newcommand{\Proj}{{\rm P}}
\newcommand{\norm}[1][\cdot]{\lvert #1 \rvert}
\newcommand{\extnorm}[1]{\left\lvert #1 \right\rvert}
\newcommand{\Norm}[1][\cdot]{\lVert #1 \rVert}
\newcommand{\extNorm}[1]{\left\lVert #1 \right\rVert}

\newcommand{\Pairing}[2]{\langle #1, #2 \rangle}
\newcommand{\extPairing}[2]{\left\langle #1, #2 \right\rangle}

\newcommand{\Legendre}[2]{\left( \frac{#1}{#2} \right)}

\newcommand{\Bas}{\mathcal{B}}
\newcommand{\Hom}{{\rm Hom}}
\newcommand{\Res}{{\rm Res}}
\newcommand{\Ind}{{\rm Ind}}
\newcommand{\cInd}{{\rm c-Ind}}

\newcommand{\Intw}{\mathcal{M}}
\newcommand{\IntwR}{\mathcal{R}}

\newcommand{\Cond}{\mathbf{C}}
\newcommand{\cond}{\mathfrak{c}}
\newcommand{\cusp}{{\rm cusp}}
\newcommand{\Eis}{{\rm Eis}}
\newcommand{\fin}{{\rm fin}}
\newcommand{\eis}{{\rm E}}

\newcommand{\JL}{\mathbf{JL}}

\newcommand{\Vol}{{\rm Vol}}
\makeatletter
\newcommand{\rmnum}[1]{\romannumeral #1}
\newcommand{\Rmnum}[1]{\expandafter\@slowromancap\romannumeral #1@}
\makeatother

\newcommand{\fa}{\mathbf{a}}
\newcommand{\fb}{\mathbf{b}}

\title{Subconvex bounds for compact toric integrals}
\author{Han Wu}
\thanks{Research partially supported by DFG-SNF-grant 00021L\_153647}

\begin{document}

\begin{abstract}
	We generalize our method for subconvex bounds for $\GL_2 \times \GL_1$ to the setting of the Waldspurger's formula for compact torical integrals. We address the two major difficulties: one is the lack of split places with small norm, the other is the test vector problem. The final bound is valid with arbitrary high probability and is better than the known bounds for a non-empty interval.
\end{abstract}

	\maketitle
	
	\tableofcontents

\section{Introduction}

	\subsection{Period Formulas for Tori and Subconvexity}
	
	Besides its support for the Lindel\"of hypothesis/Riemann hypothesis, the subconvexity of $L$-functions gives the most effective results for many interesting problems of equidistribution from related domains of the analytic number theory. Such a relation usually shows up when we consider the Weyl sums related to the relevant equidistribution problem. For example, in a series of papers \cite{Zh01,Zh01Bis,Zh04}, among other things, Zhang provided very general formulas relating central values of (a special type of) Rankin-Selberg $L$-functions to twisted Weyl sums corresponding to the equidistribution problem of Heegner points. Consequently, the relevant equidistribution result becomes a consequence of the subconvexity obtained in \cite{MV10}. Another example is provided by \cite{Sar01}, where the author relates the Quantum Unique Ergodicity conjecture to the subconvexity of (a special type of) triple product $L$-functions. In all these examples, one relates the equidistribution problem to the subconvexity problem by computing the relevant Weyl sums using appropriate \emph{period formulas}. Before \cite{Ve10, MV10}, the subconvexity results were mostly established based on approaches using classical relative trace formulas (Petersson-Kuznetsov formulas).
	
	In \cite{Ve10,MV10}, Michel \& Venkatesh initiated an approach to the subconvexity problem based on period formulas, which replaces the relative trace formulas by the Plancherel formulas. The equidistribution of certain subvarieties, such as slopes/cones approaching the low lying horocycles or Hecke points, is one of the main ingredients, which controls the part orthogonal to the one-dimensional contribution of the relevant period bounds. With some improvements on the regroupment of the Plancherel formulas, we made two special cases treated in \cite{MV10}, the cases of $\GL_2 \times \GL_1$ and $\GL_1$ $L$-functions, explicit in \cite{Wu14, Wu2}. The explicitly treated cases in \cite{Wu14,Wu2} use \emph{standard} period formulas, in the sense that they are (adelic generalization of) the standard integral representation of the relevant $L$-functions due to Hecke-Jacquet-Langlands theory, which
\begin{itemize}
	\item[(1)] (conjecturally) represent the most general families of automorphic $L$-functions;
	\item[(2)] yield all the standard analytic properties, such as the analytic continuation and functional equation, of the relevant $L$-functions.
\end{itemize}
Concretely, if $\varphi \in \pi \subset \intL^2(\GL_2,\omega)$ is a vector in a cuspidal representation $\pi$ of $\GL_2$ over a number field $\F$ with central character $\omega$, and if $\chi$ is a (unitary) Hecke character of $\F$. Then the Hecke-Jacquet-Langlands theory is roughly an equality for any $s \in \C$
\begin{equation}
	\int_{\F^{\times} \backslash \A^{\times}} \varphi(\begin{pmatrix} y & \\ & 1 \end{pmatrix}) \chi(y) \norm[y]_{\A}^{s-1/2} d^{\times}y = L(s,\pi \otimes \chi) \cdot \text{Local Factors},
\label{HJLRough}
\end{equation}
	where $\A$ is the ring of adeles of $\F$, $\A^{\times}$ is the group of ideles and $\norm_{\A}$ is the adelic norm. The ``local factors'' is a finite product of terms indexed by the places $v$ of $\F$ at which $\varphi$ or $\chi$ is ramified. In particular, if $s=1/2$, then (\ref{HJLRough}) can be viewed as an integral of $\varphi$ along the split diagonal torus against a character of that torus. If $\gp{T} < \GL_2$ is a torus determined by an $\F$-algebra embedding of a quadratic field extension $\E$ to the matrix algebra $\Mat_2(\F)$, and if $\Omega$ is a (unitary) Hecke character of $\E$ (hence $\gp{T}(\A)$) which coincides with $\omega^{-1}$ along the center $\gp{Z}$ of $\GL_2$, then one has a \emph{non-standard} period formula due to Waldspurger
\begin{equation}
	\extnorm{ \int_{\gp{Z}(\A) \backslash \gp{T}(\A)} \varphi(t) \Omega(t) d^{\times}t }^2 = L(1/2, \pi \times \pi(\Omega)) \cdot \text{Local Factors},
\label{WaFRough}
\end{equation}
where $\pi(\Omega)$ is the automorphic representation of $\GL_2$ over $\F$ related to $\Omega$ via theta correspondence and $L(s,\pi \times \pi(\Omega))$ is the Rankin-Selberg $L$-function. We used (\ref{HJLRough}) to obtain a subconvex bound of $L(1/2,\pi \otimes \chi)$ with respect to the conductor of $\chi$ in \cite{Wu14}.
	
	Since the Waldspurger's formula (\ref{WaFRough}) is formally very similar to the Hecke-Jacquet-Langlands formula (\ref{HJLRough}), many people \footnote{We thank Prof. Farrell Brumley and Prof. Tonghai Yang for generously mentioning such a possibility to us.} are curious about the possibility of generalizing our former result \cite{Wu14} with Waldspurger's formula. In fact, the use of the Waldspurger's formula was already discussed in \cite[\S 7]{Ve10}. But it was still in the direction of applying the subconvexity to the equidistribution results. The possibility of obtaining the subconvexity results directly from the Waldspurger's formula was not pursued in the subsequent \cite{MV10}. In fact, Venkatesh observed in \cite[\S 7]{Ve10} one of the two major difficulties for such a possibility: there is no unconditional results for the abundance of split places in a quadratic field extension with small norm with respect to the discriminant! The other major difficulty is the choice of a test vector in the Waldspurger's formula. In fact, if we take the new vector, we get a trivial equality most of the time. Instead, people look for a translate of the new vector by suitable (adelic) elements of the group of invertible elements in the relevant quaternion algebra which does make the formula vanishing. Much work has been done in this direction, since it (together with the analogous problem of test vector for the Gross-Zagier formula) has an intimate relation with the famous Birch-Swinnerton-Dyer conjecture. However, the best known results in this direction \cite{FMP, CST14} are still partial answer.

	\subsection{Statement of Main Result}
	
	In this paper, we pursue the possibility of obtaining an explicit subconvexity result directly from the Waldspurger's formula, which was left untreated in \cite{MV10} extended from \cite{Ve10}. In particular, we address the two major difficulties introduced in the previous subsection.
	
	Precisely, the family of $L$-functions we consider in this paper is as follows. Let $\pi$ be a fixed cuspidal representation of $\GL_2(\A)$ over a number field $\F$ with central character $\omega$. Let $(\E, \Omega)$ run over pairs consisting of
\begin{itemize}
	\item a quadratic field extension $\E/\F$,
	\item a Hecke character $\Omega$ of the idele class group $\E^{\times} \backslash \A_{\E}^{\times}$ satsifying $\Omega \mid_{\A^{\times}} \cdot \omega = 1$.
\end{itemize}
	We also demand that the following condition on the global epsilon factor is satisfied
	$$ (\varepsilon(1/2, \pi_{\E} \otimes \Omega) \cdot \Omega(-1) =) \varepsilon(1/2, \pi_{\E} \otimes \Omega) = 1, $$
where $\pi_{\E}$ is the base change of $\pi$ to $\GL_2(\A_{\E})$. We call such $(\E,\Omega)$ a $\pi$-\emph{admissible pair}. Write $\Dis(\E)$ for the discriminant of $\E$. We will introduce two conductors $\Cond^{\sharp}(\Omega,\pi)$ and $\Cond^{\flat}(\Omega,\pi)$ whose precise definitions will be given in Definition \ref{CondWaF}.
\begin{definition}
	We say that $(\E, \Omega)$ satisfies the \emph{$(\delta, \Delta)$-hypothesis} if 
	$$ \exists \quad E \asymp \Dis(\E)^{\delta} \Cond^{\flat}(\Omega,\pi)^{\Delta} \quad \text{such that} \quad \extnorm{ \left\{ \vp \text{ prime ideal of } \F \ \middle|\ \Nr(\vp) \leq E, \vp \text{ is split in } \E \right\} } \gg E/\log E. $$
\end{definition}
\begin{theorem}
	Let $(\E,\Omega)$ run over $\pi$-\emph{admissible pairs} in the above sense. Assume that $(\E,\Omega)$ satisfies the $(\delta/2(2+B), (1-2\theta)/4(2+B))$-hypothesis, where $\theta$ is any constant towards the Ramanujan-Petersson conjecture and the constants $\delta, B$ are given in Theorem \ref{AvgBound}. Then we have a subconvex bound
	$$ L(1/2, \pi_{\E} \otimes \Omega) \ll_{\F,\pi,\epsilon} (\Dis(\E)\Cond(\Omega))^{\epsilon} \cdot \Dis(\E)^{\frac{1}{2} - \frac{\delta}{2(2+B)}} \cdot \Cond^{\sharp}(\Omega,\pi)^{\frac{1}{2}} \Cond(\Omega)^{-\frac{1-2\theta}{4(2+B)}} $$
for any $\epsilon > 0$. The implied constants are polynomial in $\Dis(\F)$ and $\Cond(\pi)$.
\label{MainThm}
\end{theorem}
\begin{remark}
\noindent (1) The $(\delta,\Delta)$-hypothesis is a major difficulty we mentioned in the previous subsection. It seems to be out of reach to obtain it unconditionally using the current technology. However, it holds with arbitrary high probability. This will be given in Corollary \ref{SpPAbd}.

\noindent (2) The constants $\delta,B$ are \emph{not} hypothetic. Following our refinement \cite{Wu3} of \cite{Wu14}, we can take $\delta = (1-2\theta)/8$, while $B$ still need to be optimized when the current paper is being written. We feel that better constants can be obtained using the approach of relative trace formulae in the direction of \cite{FW09}.

\noindent (3) If $\pi(\Omega)$ denotes the automorphic representation of $\GL_2(\A)$ associated with $\Omega$ via the theta correspondence, it is easy to verify, using \cite[Theorem 4.7 (iii)]{JL70} for example, that
	$$ \Cond(\pi(\Omega)) \asymp_{\F} \Dis(\E) \Cond(\Omega). $$
	Hence Theorem \ref{MainThm} is a special case of the subconvexity for $\GL_2 \times \GL_2$ treated in \cite{MV10}, which is ``unbalanced'' for the two parts of the conductor. In fact, the dependence on $\Dis(\E)$ is worse than the one provided by \cite{MV10} (we are working on an explicit version of \cite{MV10} which will soon be available), while the dependence on $\Cond(\Omega)$ is far better, which is the main interest of this result.
	
\noindent (4) With the technical partition of places (\ref{PlacesPartition}), we recall
	$$ \Cond^{\sharp}(\Omega,\pi) \asymp_{\pi} \Cond(\Omega) \cdot \sideset{}{_{v \in S_{\infty,i,c}}} \prod \Cond(\Omega_v)^{2\theta}. $$
	The appearance of $\Cond(\Omega_v)^{2\theta}$ for $v \in S_{\infty,i,c}$ is due to the lack of estimation of the maximal absolute value of certain Legendre type functions (see Remark \ref{LackLegEst}), which should be removable. 
\end{remark}

	We explain our innovation to deal with the other major difficulty of test vector problem. Locally, instead of looking among translates of the new vector, we look for translates of some fixed subspace stable by the action of the maximal compact subgroup. This idea came up when we were trying to understand one of the approaches given in \cite{FMP}. It can be viewed as an intermediate step towards a solution of the original test vector problem, which already suffices for our purpose. In fact, we are more concerned with the size of the local factors in the relevant period formula and not interested in the precise location of the test vector. On the one hand, fixed subspaces are more convenient because in dealing with translates of the new vector one always end up with proving the non-vanishing of some Kloosterman-like exponential sums. While such non-vanishing is usually not obvious, the non-vanishing of the sum of the square in modulus of a complete family of such exponential sums is. Geometrically, this means that the projection of the Waldspurger test vector onto some translate of a fixed subspace has size more stable than its projection onto the line of one single vector in that translate of the subspace. On the other hand, finding the precise location in the translate of our fixed subspace might only reduce the polynomial dependence on $\Cond(\pi)$ in the final subconvex bound. However, even this intermediate step manifest a strong arithmetic nature of the problem (see (\ref{ArithNaturProj})). We hope that a more conceptual understanding of our approach in the style of \cite{GP91} can bring in new understanding on the explicit Gross-Zagier formula.
	
	Finally, we compare our result with known results. In view of the identity
	$$ L(s, \pi \times \pi(\Omega)) = L(s, \pi_{\E} \otimes \Omega), $$
where $\pi_{\E}$ is the base change of $\pi$ to $\GL_2(\A_{\E})$, two results in the literature give estimations of the same quantity. One is an explicit version \cite{Wu3} of \cite{Wu14}, which gives
\begin{equation}
	L(1/2, \pi \otimes \chi) \ll_{\epsilon,\pi_{\infty}} (\Cond(\pi)\Cond(\chi))^{\epsilon} \Dis(\F)^{B'}\Cond(\pi_{\fin})^A \Cond(\chi)^{\frac{1}{2}-\frac{1-2\theta}{8}},
\label{ExpGL2GL1}
\end{equation}
	where $\pi$ is an automorphic representation of $\GL_2$ over $\F$ and $\chi$ is a Hecke character. The expected value of $A$ satisfies
	$$ 1/4 \leq A \leq 1/2. $$
	In fact, $A = 1/4$ is the convex bound which seems to be impossible to break using the method of \cite{Wu14} while $A=1/2$ is established for the case $\F=\Q$ in \cite{BH08}. Assuming the non-realistic $B'=0$, it yields
	$$ L(1/2, \pi_{\E} \otimes \Omega) \ll_{\F,\epsilon, \pi} \Dis(\E)^{2A} \Cond(\Omega)^{\frac{1}{2}-\frac{1-2\theta}{8}}. $$
	The other is an explicit version of \cite{MV10} for the Rankin-Selberg $\GL_2 \times \GL_2$. In the original work, no explicit bound is given. But our ongoing explicit bound for $\GL_2$, which follows the method of \cite{MV10}, shows that the bound can not be better than
	$$ L(1/2, \pi_1 \times \pi_2) \ll_{\F,\epsilon, \pi_1} \Cond(\pi_2)^{\frac{1}{2} - \frac{1-2\theta}{24+32A}}, $$
	where $\pi_1,\pi_2$ are cuspidal representations of $\GL_2$ with fixed central characters, $A$ is the same constant in (\ref{ExpGL2GL1}). It yields
	$$ L(1/2, \pi \times \pi(\Omega)) \ll_{\F,\epsilon,\pi} (\Dis(\E) \Cond(\Omega))^{\frac{1}{2}-\frac{1-2\theta}{24+32A}}. $$
	Our result Theorem \ref{MainThm} depends on a pair of constants $(\delta,B)$ given in Theorem \ref{AvgBound}. Inserting (\ref{ExpGL2GL1}) for \emph{quadratic characters} $\chi$ into Theorem \ref{AvgBound}, we deduce that at worst we can take
	$$ (\delta,B) = ((1-2\theta)/8,A). $$
	Ignoring $\theta$, i.e., taking $\theta=0$, we see that our result is better than both \cite{Wu3,Wu14} and \cite{MV10} if
	$$ \Dis(\E)^{\frac{1-2A}{2(4+7A)}} < \Cond(\Omega) < \Dis(\E)^{\frac{1+8(A+2)(4A-1)}{2A}}, $$
which is a non-empty interval. In fact, this is predictable from a coarse examination of the methods. The amplifier in this paper has an efficiency between the ones for \cite{Wu14} and \cite{MV10} with respect to both $\Dis(\E)$ and $\Cond(\Omega)$. It would be interesting to ask if a method combining the advantages of the three methods exists.

	\subsection{Notations and Conventions}
	
	$\N$ is the set of natural numbers including $0$.
	
	By default, $\F$ resp. $\E$ resp. $\B$ stands for a number field resp. a quadratic field extension of $\F$ resp. a quaternion algebra defined over $\F$, with the ring of integers $\vo$ resp. $\vO$ resp. a maximal $\vo$-order $\VO$. The ring of adeles of $\F$ resp. $\E$ is denoted by $\A=\A_{\F}$ resp. $\A_{\E}$, the group of ideles $\A^{\times}$ resp. $\A_{\E}^{\times}$. $v$ resp. $w$ denotes a place of $\F$ resp. $\E$ and $\norm_v$ resp. $\norm_w$ denotes the absolute valuation on $\F_v$ resp. $\E_w$. At a complex place, $\norm$ is the usual absolute value on $\C$. $\vp$ resp. $\vP$ denotes a finite place/prime ideal of $\F$/$\vo$ resp. $\E$/$\vO$. We will sometimes write $q=q_{\vp}$ for $\Nr(\vp)$, and $\varpi_{\vp}$ a uniformizer of $\F_{\vp}$. If $\vp$ is not split in $\E$, we write $\varpi_{\E,\vp}$ for a uniformizer of $\E_{\vp}$. $\eta=\eta_{\E/\F}$ denotes the quadratic Hecke character associated with $\E/\F$ by the class field theory. In several sections/subsections the subscript for localization will be omitted, in which case a sentence of clarification is given in the beginning of these sections/subsections.
	
	$A^{\times}$ denotes the group of invertible elements of the algebra $A$ where $A$ can be $\vo,\vo_{\vp}, \vO, \vO_{\vp}, \vO_{\vP}$.
	
	For $\chi_v$ a character of $\F_v^{\times}$, its (analytic) conductor $\Cond(\chi_v)$ is defined as follows:
	
\noindent (1) If $\F_v = \R$, then there is $\sigma \in \R, m \in \{ 0,1 \}$ such that $\chi_v(t)=|t|^{i\sigma} {\rm sgn} (t)^m, t \in \R^{\times}$. Define 
	$$ \Cond(\chi_v) = 2+ |i\sigma + m|/2. $$
	
\noindent (2) If $\F_v = \C$, then there is $\sigma \in \R, m \in \Z$ such that $\chi_v(\rho e^{i \theta})=\rho^{i\sigma} e^{im\theta}, \rho > 0, \theta \in [0,2\pi)$. Define
	$$ \Cond(\chi_v) = (2+ |i\sigma + |m||/2)^2. $$
	
\noindent (3) If $v=\vp<\infty$, we define the \emph{logarithmic conductor} resp. conductor by
	$$ \cond(\chi_{\vp}) = \min \{ n \in \N : \chi_{\vp} \mid_{(1+\vp^n) \cap \vo_{\vp}^{\times}} = 1 \}, \quad \text{resp.} \quad \Cond(\chi_{\vp})=q_{\vp}^{\cond(\chi_{\vp})}. $$
	
\noindent These definitions also make sense for $\Omega_w$ a character of $\E_w^{\times}$. Define
	$$ \Cond(\Omega_v) = \sideset{}{_{w \mid v}} \prod \Cond(\Omega_v). $$
	
\begin{lemma}
	At $v=\vp < \infty$, if $\Omega_{\vp}$ coincides with $\omega_{\vp}$ on the diagonal embedding $\F_{\vp}^{\times} \hookrightarrow \E_{\vp}^{\times}$, then
	$$ \Cond(\Omega_{\vp}) \asymp_{\omega_{\vp}} \left\{ \begin{matrix} \min_{\vP \mid \vp}(\Cond(\Omega_{\vP}))^2 & \text{if } \vp \text{ is split in } \E \\ \Nr(\vp)^{2 \IntP{\cond(\Omega_{\vP}) / e_{\vp}}} & \text{if } \vP \text{ is the unique prime lying above } \vp \end{matrix} \right. , $$
	where $e_{\vp} = e(\E_{\vp} / \F_{\vp})$ is the ramification index.
\label{CondComp}
\end{lemma}
\begin{proof}
	Only the case when $\vp$ is totally ramified needs to be explained, other cases being easy. One way is to use the classification of orders that we will recall in \S \ref{ClEmb}. In particular, there exists a uniformizer $\varpi_{\E,\vp}$ such that
	$$ \vO_{\vP} = \vo_{\vp} + \varpi_{\E,\vp} \vo_{\vp}. $$
	If the conductor $\cond(\omega_{\vp})=c$ and $\Omega_{\vP}$ is trivial on $1 + \varpi_{\E,\vp}^{2r+1} \vO_{\vP}$ for some $r \geq c$, then from
	$$ 1+\varpi_{\vp}^r \vO_{\vP} = 1+ \varpi_{\vp}^r \vo_{\vp} + \varpi_{\vp}^r \varpi_{\E,\vp} \vo_{\vp} = (1+ \varpi_{\vp}^r \vo_{\vp})(1+\varpi_{\vp}^r \varpi_{\E,\vp} \vo_{\vp}) \subset (1+ \varpi_{\vp}^r \vo_{\vp})(1+\varpi_{\E,\vp}^{2r+1} \vO_{\vP}) $$
	we deduce that $\Omega_{\vP}$ is trivial on $1+\varpi_{\E,\vp}^{2r} \vO_{\vP}$. Thus either $\cond(\Omega_{\vP}) < 2\cond(\omega_{\vp})$ or $2 \mid \cond(\Omega_{\vP})$.
\end{proof}
	
	If $\gp{H}$ is an algebraic group defined over $\F$, then $\gp{H}(\F_v)$ is the group of $\F_v$-points, sometimes written as $\gp{H}_v$ if no confusion occurs. $\gp{G}_{\B}$ denotes the $\F$-group of invertible elements in $\B$ and we usually omit the subscript $\B$. We write $[\gp{G}] := \gp{G}(\F) \gp{Z}(\A) \backslash \gp{G}(\A)$, where $\gp{Z}$ is the center of $\gp{G}$.
	
	$\A, \A_{\E}$ are given the standard measures \`a la Tate with respect to the standard additive character $\psi_0$ of the ring of adeles of the field $\Q$ of rationals and the trace map. $\A^{\times}$ resp. $\A_{\E}^{\times}$ are given the Tamagawa measure with convergence factors $\zeta_v(1)$ resp. $\zeta_w(1)$, where $\zeta_v(s)$ resp. $\zeta_w(s)$ is the local zeta function of $\F$ resp. $\E$. $\gp{G}(\A)$ is given the Tamagawa measure with convergence factors $\zeta_v(1)$. In particular, the total mass of $[\gp{G}]$ is $2$.
	
	In $\GL_2$, for local or global variables $x \in \F_v$ or $\A$, $y \in \F_v^{\times}$ or $\A^{\times}$, we write
	$$ n(x) = \begin{pmatrix} 1 & x \\ & 1 \end{pmatrix}, \quad n_-(x) = \begin{pmatrix} 1 & \\ x & 1 \end{pmatrix}, \quad a(y) = \begin{pmatrix} y & \\ & 1 \end{pmatrix}, \quad a_-(y) = \begin{pmatrix} 1 & \\ & y \end{pmatrix}. $$
	$\gp{B} < \GL_2$ denotes the subgroup of upper triangular matrices. $\gp{K} = \sideset{}{_v} \prod \gp{K}_v$ is the standard maximal compact subgroup of $\GL_2(\A)$, i.e.
	$$ \gp{K}_v = \left\{ \begin{matrix} \SO_2(\R) & \text{if } \F_v = \R \\ \SU_2(\C) & \text{if } \F_v = \C \\ \GL_2(\vo_{\vp}) & \text{if } v=\vp < \infty \end{matrix} \right. . $$
	At $v=\vp < \infty$, we define some subgroups of $\gp{K}_{\vp}$ for $n \in \N$
	$$ \gp{K}_0(\vp^n) := \left\{ \begin{pmatrix} a & b \\ c & d \end{pmatrix} \in \gp{K}_{\vp} \middle| c \in \vp^n \right\}, \quad \gp{K}(\vp^n) := \left\{ \kappa \in \gp{K} \middle| \kappa \equiv \begin{pmatrix} 1 & \\ & 1 \end{pmatrix} \pmod{\vp^n} \right\} $$
	with the convention $\gp{K}_0(\vp^0)= \gp{K}(\vp^0) = \gp{K}_{\vp}$.
	
	$L(\cdot)$ denotes $L$-functions without factors at infinity. $\Lambda(\cdot)$ denotes the complete $L$-functions. In particular, $\zeta_{\F}(s)$ is the Dedekind zeta-function of $\F$ and $\Lambda_{\F}(s)$ is the complete Dedekind zeta-function.
	
	Other notations will be introduced later in the text.

	\subsection{Acknowledgement}
	
	The preparation of this paper scatters during the stay of the author as postdoctoral researcher in FIM at ETHZ, at MPIM in Bonn and in TAN at EPFL. The author would like to thank all three institutes for the hospitality. The author is greatly in debt of his tutor Professor Emmanuel Kowalski and his colleague Professor Paul Nelson at ETHZ for their encouragements, generously sharing ideas and references during the whole preparation of this paper.

\section{Miscellaneous Preliminaries}

	\subsection{Waldspurger Formula: Subspace Version}
	
	Let $\pi = \otimes_v' \pi_v$ be a cuspidal representation of $\GL_2(\ag{A})$ with central character $\omega$. Let $\Omega$ be a Hecke character of $\ag{A}_{\E}^{\times}$, whose restriction to $\ag{A}^{\times}$ coincides with $\omega^{-1}$. Let $\B$ be a quaternion algebra over $\F$ containing $\E$. We write for $\gp{G}=\gp{G}_{\B}$ the $\F$-group of the invertible elements in $\B$.
\begin{definition}
	Let $v$ be a place of $\F$. We say that $\B_v$ belongs to $(\pi_v,\E_v,\Omega_v)$, where $\E_v = \F_v \otimes_{\F} \E$ and $\Omega_v = \sideset{}{_{w \mid v}} \prod \Omega_w$ is the local component of $\Omega$ at $\E_v^{\times}$, if
\begin{itemize}
	\item[(1)] $\E_v$ is $\F_v$-embeddable to $\B_v$, in which case we denote by $\gp{T}_v$ the image of $\E_v^{\times}$ in $\gp{G}_v$ for a(any) such $\F_v$-embedding and regard $\Omega_v$ as a character of $\gp{T}_v$;
	\item[(2)] The Jacquet-Langlands lifting $\pi_v' = \JL(\pi_v,\B_v)$ of $\pi_v$ to $\gp{G}_v$ exists and $\Hom_{\gp{T}_v}(\pi_v', \Omega_v^{-1}) \neq \{ 0 \}$.
\end{itemize}
	We say that $\B$ belongs to $(\pi,\E,\Omega)$ if locally at every place $v$ of $\F$, $\B_v$ does.
\label{QABelong}
\end{definition}

\noindent Let $\pi_{\E} = \otimes_v' \pi_{\E,v}$ be the base change of $\pi$, automorphic representation of $\GL_2(\ag{A}_{\E})$. We know from a theorem of Tunnell \cite{Tu83} and Saito \cite{Sai93} that the local epsilon-factors satisfy:
\begin{itemize}
	\item $ \varepsilon_v(1/2, \pi_{\E,v} \otimes \Omega_v) \cdot \Omega_v(-1) \in \{ \pm 1 \} $.
	\item $ \varepsilon_v(1/2, \pi_{\E,v} \otimes \Omega_v) \cdot \Omega_v(-1) = 1 $ if and only if $\Hom_{\gp{T}_v}(\pi_v, \Omega_v^{-1}) \neq \{ 0 \}$, where $\gp{T}_v$ is the image of $\E_v^{\times}$ for any $\F_v$-embedding of $\E_v$ into $\Mat_2(\F_v)$;  $ \varepsilon_v(1/2, \pi_{\E,v} \otimes \Omega_v) \cdot \Omega_v(-1) = -1 $ if and only if $\Hom_{\gp{T}_v}(\pi_v', \Omega_v^{-1}) \neq \{ 0 \}$, where $\B_v$ the unique division quaternion $\F_v$-algebra and where $\gp{T}_v$ is the image of $\E_v^{\times}$ for any $\F_v$-embedding of $\E_v$ into $\B_v$.
\end{itemize}
Consequently, a $\F$-quaternion algebra $\B$ belonging to $(\pi,\E,\Omega)$ exists if and only if the global epsilon factor satisfies
\begin{equation}
	\varepsilon(1/2, \pi_{\E} \otimes \Omega) \cdot \Omega(-1) = 1,
\label{GlobalAdmCond}
\end{equation}
under which condition such $\B$ is unique and the global Jacquet-Langlands lifting $\JL(\pi)=\JL(\pi; \B)$ of $\pi$ to $\gp{G}(\A)$ exists. Moreover, if we choose an $\F$-embedding $\iota: \E \to \B$, regard the image of $\E^{\times}$ as an $\F$-subgroup $\gp{T}$ of $\gp{G}$ containing the center $\gp{Z}$ and normalizes the Haar measure $dt = \sideset{}{_v} \prod dt_v$ on $\gp{Z} \backslash \gp{T}$ as the quotient of the measures \`a la Tate on $\A_{\E}^{\times}$ and $\A_{\F}^{\times}$, so that for $\eta=\eta_{\E/\F}$ the quadratic Hecke character associated with the quadratic extension $\E/\F$ we have
\begin{equation}
	{\rm Vol}([\gp{T}], dt) = 2\Lambda(1,\eta), \quad [\gp{T}]:=\gp{T}(\F)\gp{Z}(\A)\backslash \gp{T}(\A),
\label{TamNumT}
\end{equation}
Waldspurger \cite[Proposition 7]{Wa85} (extended in \cite[\S 1.4 \& \S 2]{YZZ13} \footnote{The normalization of measures in \cite[\S 1.4.2]{YZZ13} is not convenient for our purpose. We follow the one given in \cite{Wa85}.}) proved the following formula:
\begin{theorem}
	For any smooth function $\varphi = \otimes_v' \varphi_v \in \pi'=\JL(\pi)=\JL(\pi;\B)$ on the automorphic quotient $[\gp{G}] := \gp{G}(\F) \gp{Z}(\ag{A}) \backslash \gp{G}(\ag{A})$, which is (abstractly) decomposable, we have
$$
	\frac{\left| \ell(\varphi; \Omega, \iota)  \right|^2}{\langle \varphi, \varphi \rangle_{[\gp{G}]}} =  \frac{\Lambda(1/2, \pi_{\E} \otimes \Omega)}{2 \Lambda(1, \pi, {\rm Ad})} \cdot \prod_{v \in V_{\F}} \frac{L(1,\eta_v)L(1,\pi_v, {\rm Ad})}{\zeta_v(2)L(1/2, \pi_{\E,v} \otimes \Omega_v)} \alpha_v(\varphi_v; \Omega_v, \iota_v),
$$
	where the notations and conventions are
\begin{itemize}
	\item $ \ell(\varphi; \Omega, \iota) = \int_{[\gp{T}]} \varphi(t) \Omega(t) dt $;
	\item $ \alpha_v(\varphi_v; \Omega_v, \iota_v) = \int_{\gp{Z}(\F_v) \backslash \gp{T}(\F_v)} \frac{\langle \pi_v'(t).\varphi_v, \varphi_v \rangle_v}{\langle \varphi_v, \varphi_v \rangle_v} \Omega_v(t) dt $ is independent of the choice of a local pairing;
	\item the measure on $[\gp{G}]$ defining the norm is the Tamagawa measure which gives the whole space the total mass $2$.
\end{itemize}
\label{WaF}
\end{theorem}
	
	We will use Theorem \ref{WaF} with some subspace of $\JL(\pi)$ instead of individual vectors. If $\sigma_v$ is a finite dimensional subspace of $\pi_v'$ and if $\Bas_v$ is an orthogonal basis of $\sigma_v$, we define
	$$ \alpha_v(\sigma_v; \Omega_v, \iota_v) = \sideset{}{_{e \in \Bas_v}} \sum \alpha_v(e; \Omega_v, \iota_v), $$
	$$ \tilde{\alpha}_{\vp}(\sigma_{\vp}; \Omega_{\vp}, \iota_{\vp}) := \frac{L(1,\eta_{\vp})L(1,\pi_{\vp}, {\rm Ad})}{\zeta_{\vp}(2)L(1/2, \pi_{\E,\vp} \otimes \Omega_{\vp})} \alpha_{\vp}(\sigma_{\vp}; \Omega_{\vp}, \iota_{\vp}), \text{ for } \vp < \infty, $$
where we have of course used the abus of notations
	$$ L(1/2, \pi_{\E,\vp} \otimes \Omega_{\vp}) := \sideset{}{_{\vP \mid \vp}}\prod L(1/2, \pi_{\E,\vP} \otimes \Omega_{\vP}). $$
If $\sigma = \otimes_v' \sigma_v$ is a finite dimensional subspace of $\JL(\pi)$ with an(any) orthogonal basis $\Bas$, we define
	$$ \alpha(\sigma; \Omega, \iota) = \sideset{}{_{e \in \Bas}} \sum \frac{\left| \ell(e; \Omega, \iota)  \right|^2}{\langle e, e\rangle_{[\gp{G}]}}. $$
	They are independent of the choice of the basis. Then we have
\begin{align}
	\alpha(\sigma; \Omega, \iota) &= \frac{L(1/2, \pi_{\E} \otimes \Omega)}{2 L(1, \pi, {\rm Ad})} \cdot \sideset{}{_{v \mid \infty}} \prod \frac{L(1,\eta_v)}{\zeta_v(2)} \label{WaFSV} \\
	&\quad \cdot \sideset{}{_{v \mid \infty}} \prod \alpha_v(\sigma_v; \Omega_v, \iota_v) \cdot \sideset{}{_{\vp < \infty}} \prod \tilde{\alpha}_{\vp}(\sigma_{\vp}; \Omega_{\vp}, \iota_{\vp}). \nonumber
\end{align}
	Note that the component
	$$ \sideset{}{_{v \mid \infty}} \prod \frac{L(1,\eta_v)}{\zeta_v(2)} $$
can be bounded from above and below by constants depending only on $\F$.
\begin{proposition}
	If $\sigma_v$ resp. $\sigma_{\vp}$ is stable by the action of a compact subgroup $\gp{K}_v'$ resp. $\gp{K}_{\vp}'$, then $\alpha_v(\cdot)$ resp. $\tilde{\alpha}_{\vp}(\cdot)$ depends only on the conjugacy class of $\iota_v$ resp. $\iota_{\vp}$ by $\gp{K}_v'$ resp. $\gp{K}_{\vp}'$.
\end{proposition}
\begin{proof}
	If we conjugate $\iota_v$ by $\kappa^{-1} \iota_v \kappa$ for some $\kappa \in \gp{K}_v'$, then we simply replace the basis $\Bas_v$ by $\kappa \Bas_v$, which is another orthogonal basis. The result follows from the independence of choice of basis.
\end{proof}

	As subspaces, we are particularly interested in two types of sub $\gp{K}_v$-representations of $\pi_v$:
\begin{definition}
\noindent (1) $\sigma_v=\sigma_0=\sigma_0(\pi_v)$ is the subspace generated by the new vector $v_0$ and the action of $\gp{K}_v$. We denote the orthogonal projector from $\pi_v$ to $\sigma_0$ by $\Proj_{*,v}$.

\noindent (2) $v = \vp < \infty$, $\sigma_{\vp}=\pi_{\vp}^{\gp{K}(\vp^c)}=[\pi_{\vp};c]$ is the subspace of vectors invariant by $\gp{K}(\vp^c)$. We denote the orthogonal projector from $\pi_{\vp}$ to $[\pi_{\vp};c]$ by $\Proj_{c,\vp}$. If $c=\cond(\pi_{\vp})$, we write $[\pi_{\vp}]=[\pi_{\vp};c]$.
\label{InvSpDef}
\end{definition}

	There is a simple fact which is worth recording.
\begin{lemma}
	For any vector $\varphi_v \in \pi_v'$, $g_v \in \gp{G}_v'$ and any embedding $\iota_v: \E_v \to \B_v$, we have
	$$ \alpha(\pi_v'(g_v).\varphi_v; \Omega_v, g_v\iota_v g_v^{-1}) = \alpha(\varphi_v; \Omega_v, \iota_v). $$
\label{FixEmb}
\end{lemma}

	\subsection{Local Waldspurger Functional}
	
	We omit the subscript $v$ since we work locally in this subsection and assume $\Hom_{\gp{T}}(\pi', \Omega^{-1}) \neq 0$, hence it is of dimension $1$. Recall that $\gp{G}$ is $\B^{\times}$ and $\gp{T}$ is the image of $\E^{\times}$ for an embedding $\iota$.
	
	First consider the case $\E$ is non-split, hence a field. In this case, $\gp{Z} \backslash \gp{T} \simeq \F^{\times} \backslash \E^{\times}$ is compact. Hence $\Res_{\gp{T}}^{\gp{G}} \pi'$ is a direct sum of characters of $\gp{T}$. Any $\ell \in \Hom_{\gp{T}}({\rm Res}_{\gp{T}}^{\gp{G}} \pi', \Omega^{-1})$ can be written as
	$$ \ell(v) = \frac{\langle v, \hat{v} \rangle}{\langle \hat{v}, \hat{v} \rangle} \ell(\hat{v}), \quad \forall v \in \pi'^{\infty} $$	
where $0 \neq \hat{v}=\hat{v}_{\iota,\Omega} \in \pi'$ is the unique up to scalar vector satisfying
\begin{equation}
	\pi'(t).\hat{v} = \Omega^{-1}(t)\hat{v}, \quad \forall t \in \gp{T}.
\label{Omega-1Vec}
\end{equation}
We fix an $\ell \neq 0$, i.e., $\ell(\hat{v}) \neq 0$. The operator
	$$ \Proj: \pi' \to \pi', \Proj(v) = \frac{1}{\Vol(\gp{Z} \backslash \gp{T})} \int_{\gp{Z} \backslash \gp{T}} \Omega(t) \pi'(t).v dt $$
satisfies $\Proj^2 = \Proj, \Proj^*=\Proj$ with image equal to the $\Omega^{-1}$-isotypic subspace of $\Res_{\gp{T}}^{\gp{G}} \pi'$. Hence it is the orthogonal projector onto the $\Omega^{-1}$-isotypical subspace $\C \hat{v}$. We deduce
	$$ \Proj(v) = \frac{\langle v, \hat{v} \rangle}{\langle \hat{v}, \hat{v} \rangle}  \hat{v}, \quad \forall v \in \pi'. $$
Hence we have
\begin{equation}
	\alpha(v; \Omega, \iota) = \int_{\gp{Z} \backslash \gp{T}} \frac{\langle \pi'(t).v, v \rangle}{\langle v, v \rangle} \Omega(t) dt = \Vol(\gp{Z} \backslash \gp{T}) \frac{\langle \Proj(v), v \rangle}{\langle v, v \rangle} = \Vol(\gp{Z} \backslash \gp{T}) \frac{\lVert \hat{v} \rVert^2}{\lVert v \rVert^2} \cdot \frac{|\ell(v)|^2}{|\ell(\hat{v})|^2}.
\label{NSPFRIndividual}
\end{equation}
If $\sigma$ is any finite dimensional subspace of $\pi$ with orthogonal projector $\Proj_{\sigma}$, we deduce that
\begin{equation}
	\alpha(\sigma; \Omega, \iota) = \Vol(\gp{Z} \backslash \gp{T}) \frac{\lVert \Proj_{\sigma}(\hat{v}) \rVert^2}{\lVert \hat{v} \rVert^2}.
\label{NSPFRAverage}
\end{equation}
\begin{definition}
	If $\iota=\iota_r$ resp. $\iota_r'$ (which will be classified in \S \ref{ClassEmb}), we write the corresponding $T=T_r, \hat{v}_r=\hat{v}_{r,\Omega}=\hat{v}_{\iota_r,\Omega}$ resp. $T=T_r', \hat{v}_r'=\hat{v}_{r,\Omega}'=\hat{v}_{\iota_r,\Omega}'$. If we are given a model of $\pi$ as a space of functions, the function corresponding to $\hat{v}_r$ resp. $\hat{v}_r'$ is denoted by $\hat{f}_r$ resp. $\hat{f}_r'$.
\label{conventionHat_r}
\end{definition}

	Next, turn to the case where $\E$ splits. In this case $\E \simeq \F \times \F$, $\B$ must split. There exist $g_0 \in \GL_2(\F)$ and characters $\chi_1,\chi_2$ of $\F^{\times}$ with $\Cond(\chi_1) \leq \Cond(\chi_2)$ such that
\begin{equation}
	\iota(t_1,t_2) = g_0^{-1} \begin{pmatrix} t_1 & \\ & t_2 \end{pmatrix} g_0, \quad \forall t_1,t_2\in F,
\label{SplitEmb}
\end{equation}
and
$$ \Omega(t_1,t_2)=\chi_1(t_1)\chi_2(t_2), \quad \forall t_1,t_2\in \F^{\times}. $$	
A Waldspurger functional $\ell \in \Hom_{\gp{T}}(\pi', \Omega^{-1})$ can be defined via a(any) Whittaker model $\mathcal{W}(\pi', \bar{\psi})$ of $\pi'$ as
\begin{equation}
	\ell(W) = \int_{\F^{\times}} W(a(y)g_0) \chi_1(y) d^{\times}y.
\label{SplitWaldFunct}
\end{equation}
For any $v \in \pi'$, if $W_v$ is the corresponding function in the Whittaker model, then it is easy to show
\begin{equation}
	\alpha(v; \Omega, \iota) = \int_{\gp{Z} \backslash \gp{T}} \frac{\langle \pi'(t).v, v \rangle}{\langle v, v \rangle} \Omega(t) dt = \frac{|\ell(W_v)|^2}{\lVert W_v \rVert^2}
\label{SPFRIndividual}
\end{equation}
where $\lVert W_v \rVert^2 = \int_{\F^{\times}} |W_v(a(y))|^2 d^{\times}y$.

	\subsection{Classes of Local Embeddings $\E_v \to \Mat_2(\F_v)$}
	\label{ClassEmb}
	
	We classify the embeddings of $\E_v \to \Mat_2(\F_v)$ up to conjugation by $\gp{K}_v$, the standard maximal compact subgroup of $\GL_2(\F_v)$. We omit the subscript $v$ since we work locally in this subsection.
	
		\subsubsection{$v \mid \infty$, $\E_v$ split.}
		
	Fix a splitting $s: \E \simeq \F \times \F$, and define
	$$ \iota_0: \E \to \Mat_2(\F), \iota_0(s^{-1}(t_1,t_2)) = \begin{pmatrix} t_1 & 0 \\ 0 & t_2 \end{pmatrix}, $$
	Then for $r \in \F$, we define $\iota_r: \E \to \Mat_2(\F)$ such that
\begin{equation}
	\iota_r(x) = n(-r)\iota_0(x)n(r), \quad \forall x \in \E.
\label{ArchCarEmbS}
\end{equation}
\begin{proposition}
	Every $\F$-embedding $\iota: \E \to \Mat_2(\F)$ is conjugate by $\gp{K}$ to a unique $\iota_r$ for
\begin{itemize}
	\item[(1)] $r \in \R$ if $\F = \R$;
	\item[(2)] $r \geq 0$ if $\F = \C$.
\end{itemize}
\end{proposition}
\begin{proof}
	Every $\iota$ is conjugate to $\iota_0$ by some element in $\GL_2(\F)$. The result follows from the decompositions
	$$ \GL_2(\R) = \bigsqcup_{r \in \R} \iota_0(\E^{\times}) \begin{pmatrix} 1 & r \\ 0 & 1 \end{pmatrix} \SO_2(\R), \quad \GL_2(\C) = \bigsqcup_{r \geq 0} \iota_0(\E^{\times}) \begin{pmatrix} 1 & r \\ 0 & 1 \end{pmatrix} \SU_2(\C). $$
\end{proof}

		\subsubsection{$v \mid \infty$, $\E_v$ non-split.}
		
	Necessarily $\F = \R, \E \simeq \C$. Fix an isomorphism $s: \E \simeq \C$ and define
	$$ \iota_0(s^{-1}(i)) = \begin{pmatrix} 0 & 1 \\ -1 & 0 \end{pmatrix}. $$
	For any $r \neq 0$, we define
\begin{equation}
	\iota_r = \begin{pmatrix} 1 & 0 \\ 0 & r^{-1} \end{pmatrix} \iota_0 \begin{pmatrix} 1 & 0 \\ 0 & r \end{pmatrix} = a(r) \iota_0 a(r^{-1}).
\label{ArchCarEmbNS}
\end{equation}
\begin{proposition}
	Every $\F$-embedding $\iota: \E \to \Mat_2(\F)$ is conjugate by $\gp{K}$ to a unique $\iota_r$ for $r \neq 0$.
\end{proposition}

		\subsubsection{$v=\vp < \infty$, $\E_{\vp}$ split.}
		
	Fix an isomorphism $s=s_{\vp}: \E \simeq \F \times \F$. Any $\vo$-order of $\E$ can be written as $\vO = s^{-1}(\vo(1,1)+\vo \tau)$ for some $\tau$ satisfying a split separable monic polynomial in $\vo[X]$. After replacing $\tau$ by $u+v\tau$ for suitable $u \in \vo, v \in \vo^{\times}$, we can assume $\tau$ satisfies $\tau(\tau-\varpi^r)=0$ for some $r \in \N$, thus $\tau=(\varpi^r,0)=\varpi^r \beta$ with
\begin{equation}
	\beta = (1,0).
\label{BetaSplit}
\end{equation}	
Hence for any $\vo$-order of $\E$ there is a unique $r \in \N$, called the \textbf{(logarithmic) conductor}, such that it can be written as
	$$ \vO_r = s^{-1}(\vo(1,1)+\varpi^r \vo(1,0)). $$
Define $\iota_0$ to be the embedding
	$$ \iota_0(s^{-1}(t_1,t_2)) = \begin{pmatrix} t_1 & \\ & t_2 \end{pmatrix}, \quad \forall (t_1,t_2) \in \F \times \F = \E. $$
The usual Iwasawa decomposition gives
\begin{equation}
	\GL_2(\F) = \bigsqcup_{r=0}^{\infty} \iota_0(\E^{\times}) \begin{pmatrix} 1 & \varpi^{-r} \\ 0 & 1 \end{pmatrix} \GL_2(\vo).
\label{Iwa}
\end{equation}
For any integer $r \geq 1$, we define an embedding $\iota_r$ by
\begin{equation}
	\iota_r(t) = n(\varpi^{-r}) \iota_0(t) n(-\varpi^{-r}), \forall t \in \F \times \F = \E.
\label{CarEmbS}
\end{equation}
Then we have
	$$ \iota_r(\E) \cap \Mat_2(\vo) = \iota_r(\vO_r). $$
\begin{proposition}
	For any embedding $\iota: \E \to \B =\Mat_2(\F)$, there is a unique integer $r=:\cond(\iota) \in \N$, called the \textbf{(logarithmic) conductor} of $\iota$ determined by
	$$ \iota(\E) \cap \Mat_2(\vo) = \iota(\vO_r), $$
such that $\iota$ is conjugate by $\gp{K}$ to $\iota_r$.
\label{LogCondEmbd}
\end{proposition}	
\begin{proof}
	By the Skolem-Noether theorem \cite[THEOREME \Rmnum{2}.2.1]{Vi80}, $\iota$ is conjugate to $\iota_0$ by some element of $\GL_2(\F)$. The assertion then follows readily from (\ref{Iwa}).
\end{proof}
		
		\subsubsection{$v=\vp < \infty$, $\E_{\vp}$ non-split.}
		\label{ClEmb}
		
	$v$ extends to a unique valuation $v_{\E}$ on $\E$ with uniformiser $\varpi_{\E}$. By Definition on \cite[p.44]{Vi80}, there is $\beta \in \vO - \vo$ such that
\begin{equation}
	\vO = \vo + \beta \vo,
\label{Beta}
\end{equation}
and any other order of $\E$ can be written as $\vO_r = \vo + \varpi^r \beta \vo$ for a unique $r \in \N$ called the \textbf{(logarithmic) conductor} of the order. $\beta$ is a root of an irreducible polynomial $X^2-\fb X+\fa$ in $\vo[X]$. For any $x \in \E$, let $\iota(x) \in \Mat_2(\F)$ be the matrix such that
	$$ x(1, \beta) = (1, \beta)\iota(x). $$
Note that the choice of $\beta$ is not unique. All (other) possible choices are $\beta'=u+v\beta$ for $u \in \vo$ and $v \in \vo^{\times}$. Hence
	$$ \iota'(\beta') = \begin{pmatrix} & -\fa' \\ 1 & \fb' \end{pmatrix} \left( = \begin{pmatrix} 1 & u \\ & v \end{pmatrix}^{-1} \iota(\beta') \begin{pmatrix} 1 & u \\ & v \end{pmatrix} \right), \quad \fb'=2u+v\fb, \fa'=u^2+uv\fb+v^2\fa. $$
If $\E/\F$ is unramified, then $X^2-\fb X+\fa \pmod{\vp}$ is irreducible, hence $v(\fa)=0=v_{\E}(\beta)$. If $\E/\F$ is ramified, then $X^2-\fb X+\fa \pmod{\vp}$ is a square. If $v(\fa)>0$ then $v_{\E}(\beta)>0$. But if $v_{\E}(\beta) \geq 2$ then $\beta/\varpi \in \vO$, which contradicts (\ref{Beta}). Hence $v(\fa)=v_{\E}(\beta) \in \{0,1\}$. But if $v(\fa')=0$, then we can find $u_0 \in \vo^{\times}$ such that
\begin{equation}
	X^2 - \fb' X + \fa' \equiv (X-u_0)^2 \pmod{\vp}.
\label{u0}
\end{equation}
	By the above remark on the choice of $\beta$, we can replace $\beta'$ with $\beta = \beta'-u_0$ so that $v(\fa) = 1$, hence $\beta = \varpi_{\E}$ is an uniformizer. Reciprocally, we can pass from $\beta$ with $v(\fa)=0$ to $\beta'$ with $v(\fa') = 1$. In any case, we choose $\beta$ and define
	$$ \iota_0(\beta) = \begin{pmatrix} & -\fa \\ 1 & \fb \end{pmatrix}, \quad v(\fa) = e(\E/\F)-1. $$
	For any $r \in \N$, we define an embedding $\iota_r$ by
\begin{equation}
	\iota_r(t) = a(\varpi^{-r}) \iota_0(t) a(\varpi^{r}), \quad \forall t \in \E.
\label{CarEmbNS}
\end{equation}
	In particular, we have
	$$ \iota_r(\beta) = \begin{pmatrix} 0 & -\fa \varpi^{-r} \\ \varpi^{r} & \fb \end{pmatrix}, $$
and if we define the Eichler orders
	$$ \oM = \Mat_2(\vo), \quad \oJ = \begin{pmatrix} \vo & \vo \\ \vp & \vo \end{pmatrix}, $$
then we have
	$$ \iota_r(\E) \cap \oM = \begin{matrix} \iota_r(\vO_r) \\ \iota_r(\vO_{\max(r-1,0)}) \end{matrix}, \quad \begin{matrix} \text{if } \E \text{ is unramified,} \\ \text{if } \E \text{ is ramified.} \end{matrix}. $$
In other words, $\vO_r$ resp. $\vO_{\max(r-1,0)}$ is optimally embedded in $\oM$ via $\iota_r$ in the sense of Section \Rmnum{2}.3 \cite{Vi80}.
\begin{proposition}
	Proposition \ref{LogCondEmbd} is still valid if we replace the equation by
	$$ \iota(\E) \cap \oM = \begin{matrix} \iota(\vO_r) \\ \iota(\vO_{r-1}) \end{matrix}, \quad \begin{matrix} \text{if } \E \text{ is unramified,} \\ \text{if } \E \text{ is ramified.} \end{matrix} $$
	In particular, the conductor of an embedding is $\geq e(\E/\F)-1$.
\end{proposition}
\begin{proof}
	We claim that
\begin{equation}
	\GL_2(\F) = \bigsqcup_{r=0}^{\infty} \iota_0(\E^{\times}) \begin{pmatrix} 1 & \\ & \varpi^r \end{pmatrix} \GL_2(\vo) = \bigsqcup_{r=e(\E/\F)-1}^{\infty} \iota_0(\E^{\times}) \begin{pmatrix} \varpi^r & \\ & 1 \end{pmatrix} \GL_2(\vo).
\label{IwaforE}
\end{equation}
	In fact, since any $\vo$-lattice in $\E$ is principal for some order $\vO_r$, the first equality follows. For any integer $r \geq 0$, we have
	$$ \begin{pmatrix} 1 & \\ & \varpi^{r+v(\fa)} \end{pmatrix} \iota_0(1+\varpi^{-r-v(\fa)}\beta) \begin{pmatrix} 1 & \\ & \varpi^r \end{pmatrix} = \begin{pmatrix} 1 & -\fa \varpi^{-v(\fa)} \\ 1 & \varpi^{2r+v(\fa)}+ \varpi^r \fb \end{pmatrix} \in \GL_2(\vo) $$
	$$ \Rightarrow \begin{pmatrix} 1 & \\ & \varpi^{r} \end{pmatrix} \in \iota_0(\E^{\times}) \begin{pmatrix} \varpi^{r+v(\fa)} & \\ & 1 \end{pmatrix} \GL_2(\vo), $$
implying the second equality. Now that every $\F$-embedding $\iota$ is conjugate to $\iota_0$ by some element in $\GL_2(\F)$ (Skolem-Noether), the assertion follows from (\ref{IwaforE}).
\end{proof}
\begin{remark}
	It would also be useful to keep the (other) choice (in the case $\E/\F$ is ramified), namely
	$$ \iota_0'(\beta') = \begin{pmatrix} \fb' & 1 \\ -\fa' & \end{pmatrix}, \quad v(\fa') = 0; \quad \iota_r'(\beta') := a(\varpi^r) \iota_0'(\beta') a(\varpi^{-r}) = \begin{pmatrix} \fb' & \varpi^r \\ -\fa' \varpi^{-r} & \end{pmatrix}. $$
	(\ref{IwaforE}) is still valid for $\iota_0'$ if we replace $e(\E/\F)-1$ ($=1$) with $0$. The same proof works. $\iota_r'$ is an optimal embedding of $\vO_r$ into $\oM$.
\label{AltEmb}
\end{remark}

	\subsection{Some Technical Computation}
	
	We omit the subscript $v=\vp$ in this subsection. In the previous subsection, we have defined an order $\vO_r$ of $\E$ for each $r \in \Z_{\geq 0}$. For our purpose, the following sets are interesting:
\begin{equation}
	\begin{matrix} 
		\vO_r^{(n)} = \left\{ t \in \vO_r - \varpi \vO_r : \Nr_{\F}^{\E}(t) \in \varpi^n \vo^{\times} \right\}, \quad \vO_r^{(\leq n)} = \bigsqcup_{m=0}^n \vO_r^{(m)}, \quad n \in \Z_{\geq 0}; \\
		\vO_{(r)}^{(n)} = \left\{ t \in \vO_r - \vO_{r+1} : \Nr_{\F}^{\E}(t) \in \varpi^n \vo^{\times} \right\}.
	\end{matrix}
\label{ArithQuadSet}
\end{equation}
\begin{remark}
	If $\E/\F$ is non split, then since $\varpi \vO_r \subset \vO_{r+1}$ and
	$$ \vO_{r+1} - \varpi \vO_r = \vo^{\times} + \vp^{r+1} \beta = \vO_{r+1}^{\times}, $$
	$\vO_r^{(n)}$ and $\vO_{(r)}^{(n)}$ differ only at $n=0$ and we have
	$$ \vO_r^{(0)} = \vO_{(r)}^{(0)} \sqcup \vO_{r+1}^{\times} = \vO_r^{\times}. $$
\end{remark}
\begin{lemma}
	The sets $\vO_r^{(n)}$ can be characterized as follows:
	\begin{itemize}
		\item[(1)] If $\E/\F$ is split, then for $n \in \Z_{\geq 0}$
		$$ \vO_r^{(n)} = \left\{ \begin{matrix} \vo^{\times}(1+\vp^r,1) & \text{if } n=0 \\ \emptyset & \text{if } 2 \nmid n, n < 2r \\ \varpi^{n/2}\vo^{\times}(1+\varpi^{r-n/2}\vo^{\times},1) & \text{if } 2 \mid n, 0 < n \leq 2r \\ (\varpi^r\vo^{\times},\varpi^{n-r}\vo^{\times}) \cup (\varpi^{n-r}\vo^{\times},\varpi^r\vo^{\times}) & \text{if }n > 2r \end{matrix} \right. . $$
		\item[(2)] If $\E/\F$ is an unramified field extension, then for $n \in \Z_{\geq 0}$
		$$ \vO_{(r)}^{(n)} = \left\{ \begin{matrix} \emptyset & \text{if } 2 \nmid n \text{ or } n > 2r \\ \varpi^{n/2}(\vO_{r-n/2}^{\times} - \vO_{r-n/2+1}^{\times}) & \text{if }2 \mid n, 0 \leq n \leq 2r \end{matrix} \right. . $$
		\item[(3)] If $\E/\F$ is a ramified field extension, then for $n \in \Z_{\geq 0}$
		$$ \vO_{(r)}^{(n)} = \left\{ \begin{matrix} \emptyset & \text{if } n > 2r+1 \text{ or } 0 \leq n \leq 2r,2 \nmid n \\ \varpi^{n/2}(\vO_{r-n/2}^{\times} - \vO_{r-n/2+1}^{\times}) & \text{if }2 \mid n, 0 \leq n \leq 2r \\ \varpi_{\E}^{2r+1} \vO^{\times} & \text{if } n=2r+1. \end{matrix} \right. $$
	\end{itemize}
\label{LIDecomp}
\end{lemma}
\begin{proof}
	This is a simple exercise in field theory. We omit the proof.
\end{proof}
\begin{corollary}
	$\vO_r^{\times}$ acts on $\vO_r^{(n)}$ resp. $\vO_{(r-1)}^{(n)}$ (if $r \geq 1$) by multiplication with finite orbits. The cardinality of $\vO_r^{\times} \backslash \vO_r^{(n)}$ resp. $\vO_r^{\times} \backslash \vO_{(r-1)}^{(n)}$ is:
	\begin{itemize}
		\item[(1)] If $\E/\F$ is an unramified field extension, then for $n \in \Z_{\geq 0}$
		$$ |\vO_r^{\times} \backslash \vO_r^{(n)}| = \left\{ \begin{matrix} 0 & \text{if } 2 \nmid n \text{ or } n > 2r \\ 1 & \text{if } n=0 \\ q^r & \text{if }n=2r \\ q^{n/2}(1-q^{-1}) & \text{if }2 \mid n, 0 < n < 2r \end{matrix} \right. ; $$
		$$ \extnorm{ \vO_r^{\times} \backslash \vO_{(r-1)}^{(n)} } = \left\{ \begin{matrix} 0 & \text{if } 2 \nmid n \text{ or } n > 2(r-1) \\ q^r & \text{if } n = 2(r-1) \\ q^{n/2}(q-1) & \text{if } 2 \mid n, 0 \leq n < 2(r-1) \end{matrix} \right. . $$
		\item[(2)] If $\E/\F$ is a ramified field extension, then for $n \in \Z_{\geq 0}$
		$$ |\vO_r^{\times} \backslash \vO_r^{(n)}| = \left\{ \begin{matrix} 0 & \text{if } n > 2r+1 \text{ or } 0 \leq n \leq 2r,2 \nmid n \\ 1 & \text{if } n=0 \\ q^{n/2}(1-q^{-1}) & \text{if }2 \mid n, 0 < n \leq 2r \\ q^r & \text{if } n=2r+1 \end{matrix} \right. ; $$
		$$ \extnorm{ \vO_r^{\times} \backslash \vO_{(r-1)}^{(n)} } = \left\{ \begin{matrix} 0 & \text{if } n \geq 2r  \text{ or } 0 \leq n \leq 2(r-1), 2 \nmid n \\ q^r & \text{if } n = 2r-1 \\ q^{n/2}(q-1) & \text{if } 2 \mid n, 0 \leq n \leq 2(r-1) \end{matrix} \right. . $$
	\end{itemize}
\label{LIQuotient}
\end{corollary}
\begin{proof}
	We only need to calculate $|\vO_{r+n}^{\times} \backslash \vO_r^{\times}|$ for $r\in \Z_{\geq 0}, n \geq 1$ and insert it to the lemma. Let
	$$ U_{\F}^{(r)} = 1+\varpi^r\vo, r>0; \quad U_{\F}^{(0)}=\vo^{\times}. $$
	$$ U_{\E}^{(r)} = 1+\varpi^r\vO, r>0; \quad U_{\E}^{(0)}=\vO^{\times}. $$
Since $\vO_r^{\times}=\vo^{\times}U_{\E}^{(r)}, \vo^{\times} \cap U_{\E}^{(r)}=U_{\F}^{(r)}$, we get
	$$ |U_{\E}^{(r)} \backslash \vO_r^{\times}| = |U_{\F}^{(r)} \backslash \vo^{\times}| = \left\{ \begin{matrix} q^r(1-q^{-1}) & \text{if }r>0; \\ 1 & \text{if }r=0. \end{matrix} \right. $$
It is easy to see
	$$ |U_{\E}^{(r+n)} \backslash U_{\E}^{(r)}| = \left\{ \begin{matrix} q^{2n} & \text{if }r>0; \\ q^{2n}(1-q^{-2}) & \text{if }r=0, \E/\F \text{ unramified;} \\ q^{2n}(1-q^{-1}) & \text{if }r=0, \E/\F \text{ ramified.} \end{matrix} \right. $$
We deduce that
\begin{align*}
	& |\vO_{r+n}^{\times} \backslash \vO_r^{\times}| = \frac{|U_{\E}^{(r)} \backslash \vO_r^{\times}| \cdot |U_{\E}^{(r+n)} \backslash U_{\E}^{(r)}|}{|U_{\E}^{(r+n)} \backslash \vO_{r+n}^{\times}|} \\
	& = \left\{ \begin{matrix} q^n & \text{if }r>0 \text{ or }\E/\F \text{ ramified;} \\ q^n(1+q^{-1}) & \text{if }r=0, \E/\F \text{ unramified.} \end{matrix} \right.
\end{align*}
\end{proof}

\begin{lemma}
	Let $c \geq 1, r \geq 0$ be integers. Then if $r>0$, we have independently of $c$
\begin{equation}
	\gp{K}_0(\vp^c)\iota_r'(\vO_r^{\times}) = \bigsqcup_{u \in \vo/\vp^c} \gp{K}_0(\vp^c) \begin{pmatrix} 1 & \\ u & 1 \end{pmatrix} = \left\{ \begin{pmatrix} c_1 & c_2 \\ c_3 & c_4 \end{pmatrix} \in \gp{K} \mid c_4 \in \vo^{\times} \right\};
\label{SppS1}
\end{equation}
	while if $r=0$, we have independently of $c$
\begin{equation}
	\gp{K}_0(\vp^c)\iota_0'(\vO^{\times}) = \left\{ \begin{matrix} \gp{K} & \text{if } \E/\F \text{ is unramified;} \\ \gp{K} - \bigsqcup_{u \in u_0 + \vp/\vp^c} \gp{K}_0(\vp^c) \begin{pmatrix} 1 & \\ u & 1 \end{pmatrix} & \text{if } \E/\F \text{ is ramified.} \end{matrix} \right.
\label{DecompK}
\end{equation}
\end{lemma}
\begin{proof}
	Recall the coset decomposition
	$$ \gp{K}= \bigsqcup_{u \in \vo/\vp^c} \gp{K}_0(\vp^c) \begin{pmatrix} 1 & \\ u & 1 \end{pmatrix} \bigsqcup \bigsqcup_{v \in \vp/\vp^c} \gp{K}_0(\vp^c) \begin{pmatrix} 1 & \\ v & 1 \end{pmatrix} w. $$
	If $r > 0$, note that
\begin{equation}
	\vO_r^{\times} = \vo^{\times} (1+\varpi^r \vO) = \vo^{\times}+\varpi^r\vo \beta'.
\label{vOr*}
\end{equation}
	For $x \in \vo^{\times}, y \in \vo$ we have
	$$ \iota_r'(x+\varpi^r \beta' y) \in \gp{K}_0(\vp^c) \begin{pmatrix} 1 & \\ -\fa' y/x & 1 \end{pmatrix}, $$
	which implies (\ref{SppS1}). If $r=0$, we have
\begin{equation}
	\vO^{\times} = \left\{ \begin{matrix} \{ x+y\beta' \mid \min(v(x),v(y)) = 0 \} & \text{if } \E/\F \text{ is unramified;} \\ \{ x+y\beta' \mid \min(v(x),v(y)) = 0, -x/y \neq \beta' \pmod{\vP} \} & \text{if } \E/\F \text{ is ramified.} \end{matrix} \right.
\label{vO0*}
\end{equation}
Since $\beta' \neq 0 \pmod{\vP}$ by our choice and
	$$ \gp{K}_0(\vp^c)\begin{pmatrix} 1 & \\ v & 1 \end{pmatrix} w = \gp{K}_0(\vp^c)\iota_0'(x+y\beta'), \quad \text{ if } x\in \vo, y\in \vo^{\times}, x/(\fa' y) = v \pmod{\vp^c}, $$
	$$ \gp{K}_0(\vp^c) \begin{pmatrix} 1 & \\ u & 1 \end{pmatrix} = \gp{K}_0(\vp^c)\iota_0'(x+y\beta'), \quad \text{ if } x\in \vo^{\times}, y\in \vo, -\fa' y/x = u \pmod{\vp^c}, $$
	we obtain (\ref{DecompK}).
\end{proof}

	\subsection{Inner Product for Special Representations}
	
	We record a relation between the inner product on special representations for $\GL_2$ over a non-archimedean local field $\F$ and those on complementary series representations.
\begin{definition}
	Let $\pi = \pi_{1/2} = \mathrm{St}_{\chi}$ be a special representation with $\pi_{\sigma} = \chi \otimes \pi(\norm^{\sigma}, \norm^{-\sigma}), 0 < \sigma < 1/2$ its continuous deformation into a family of complementary series. We call a section $f(\sigma) \in \pi_{\sigma}$ \emph{admissible} if its $\gp{K}$-isotypic parts does not contain a $\chi \circ \det$ component, i.e., if
	$$ \int_{\gp{K}} f(\sigma, \kappa) \chi(\det \kappa)^{-1} d\kappa = 0. $$
\label{AdmS}
\end{definition}
\begin{lemma}
	Let notations be as in Definition \ref{AdmS}. Then there exists a function $c(\sigma)$ such that for $f_1, f_2 \in \pi$ realized as functions in the induced model and for $f_1(\sigma) \in \pi_{\sigma}$ resp. $f_2(\sigma) \in \pi_{\sigma}$ any continuous admissible section with $f_1(1/2) = f_1$ resp. $f_2(1/2) = f_2$, we have
	$$ \Pairing{f_1}{f_2} = \lim_{\sigma \to 1/2-} c(\sigma) \Pairing{f_1(\sigma)}{f_2(\sigma)}_{\sigma} $$
where the pairing indexed by $\sigma$ is the standard one on $\pi_{\sigma}$ defined by
	$$ \Pairing{f_1(\sigma)}{f_2(\sigma)}_{\sigma} := \int_{\gp{K}} f_1(\sigma, \kappa) \overline{ \IntwR f_2 (\sigma, \kappa)} d\kappa, $$
with the normalized intertwining operator $\IntwR: \pi_{\sigma} \to \pi_{-\sigma}$ stabilizing the function $\gp{K} \to \C, \kappa \to \chi(\det \kappa)$.
\label{InnPSt}
\end{lemma}
\begin{proof}
	If $f_1(\sigma), f_2(\sigma)$ are flat sections, this is just \cite[\S 1.20 (300)]{Go70}. It is even possible to choose $c(\sigma) = (1-q^{\sigma-1/2})^{-1}$ \cite[\S 1.20 (304)]{Go70}. To prove the general case, we first notice that we can assume $\chi=1$ ($\pi=\mathrm{St}$) by replacing $f_j(\sigma)$ with $f_j(\sigma) \cdot (\chi \circ \det)^{-1}$ if necessary. If $f \in \pi(\norm^{1/2}, \norm^{-1/2})$ is a function in the induced model, the condition for $f \in \pi$ \cite[\S 1.20 (298)]{Go70} is equivalent to that the $\gp{K}$-isotypic components of $f$ does not contain the trivial representation, because if the measures are suitably normalized we have
	$$ \int_{\gp{K}} f(\kappa) d\kappa = \int_{\gp{B} \backslash \GL_2} f(g) dg = \int_{\F} f(w^{-1} n(x)) dx. $$
	So is the admissibility of sections. Let $1 \neq \tau \in \widehat{\gp{K}}$ run over non trivial irreducible representations of $\gp{K}$ and $e \in \Bas(\tau)$ run over an orthonormal basis of $\tau$. We can decompose
	$$ f_j(\sigma) = \sum_{\tau \neq 1} \sum_e f_j^{(\tau,e)}(\sigma), \quad j=1,2 $$
	by Peter-Weyl theorem, where the summation is finite and
	$$ f_j^{(\tau,e)}(\sigma, g) = \int_{\gp{K}} f_j(\sigma, g\kappa) d_{\tau} \Pairing{e}{\kappa.e} d\kappa, \quad d_{\tau} := \dim \tau. $$
	But multiplicity one holds for the branching law of $\pi_{\sigma}$ restricted to $\gp{K}$. Hence if we write $f_{j,\sigma}$ for the flat section of $f_j$ then we get
	$$ f_j^{(\tau,e)}(\sigma) = \left\{ \begin{matrix} c_j^{(\tau,e)}(\sigma) \cdot f_{j,\sigma}^{(\tau,e)} & \text{if } f_{j,1/2}^{(\tau,e)} \neq 0 \\ \to 0 \text{ as } \sigma \to 1/2- & \text{ if } f_{j,1/2}^{(\tau,e)} = 0, \end{matrix} \right. $$
	for some continuous functions $c_j^{(\tau,e)}(\sigma)$ with $c_j^{(\tau,e)}(1/2)=1$. Consequently, as $\sigma \to 1/2-$
	$$ c(\sigma) \Pairing{f_1(\sigma)}{f_2(\sigma)}_{\sigma} = \sideset{}{_{\tau \neq 1,e}^{'}} \sum c(\sigma) m(\sigma,\tau) c_1^{(\tau,e)}(\sigma) \overline{c_2^{(\tau,e)}(\sigma)} \int_{\gp{K}} f_{1,\sigma}^{(\tau,e)}(\kappa) \overline{f_{2,-\sigma}^{(\tau,e)}(\kappa)} d\kappa + o(1), $$
where the summation is over those $\tau,e$ such that $f_{j,1/2}^{(\tau,e)} \neq 0, j=1,2$, and where $m(\sigma,\tau)$ is the ``eigenvalue'' of $\IntwR$ on flat sections of elements in $\tau$ determined by
	$$ \IntwR f_{\sigma} = m(\sigma,\tau) f_{-\sigma}, \quad f \text{ is } \tau-\text{isotypic}. $$
	$m(\sigma,\tau)$ are computed in \cite[Lemma 3.18 (4)]{Wu5} or \cite[Corollary 4.12]{Wu2}, which all have a zero at $\sigma=1/2$ of order $1$ since $\tau \neq 1$. Hence $c(\sigma) \Pairing{f_1(\sigma)}{f_2(\sigma)}_{\sigma}$ has the same limit (as $\sigma \to 1/2-$) with
	$$ c(\sigma) \Pairing{f_{1,\sigma}}{f_{2,\sigma}}_{\sigma} = \sideset{}{_{\tau \neq 1,e}^{'}} \sum c(\sigma) m(\sigma,\tau) \int_{\gp{K}} f_{1,\sigma}^{(\tau,e)}(\kappa) \overline{f_{2,-\sigma}^{(\tau,e)}(\kappa)} d\kappa. $$
\end{proof}

	\subsection{Classification of Supercuspidal Representations}
	
	We recall briefly the necessary information on the classification of supercuspidal representations $\pi$ of $\GL_2$ at a finite place $v=\vp < \infty$. For more details, we refer to \cite[Section 5.1-5.4]{FMP} or \cite[Chapter 4]{BuH06}. We shall omit the subscript $\vp$, and fix an additive character $\tilde{\psi}$ trivial on $\vp$ but not on $\vo$, according to the convention of people working on local Langlands. One should not confuse it with $\psi$ that we choose in the main argument.
	
		\subsubsection{Type 0 minimal supercuspidal}
		
	There is a representation $\rho$ of $\gp{J}=\gp{Z}\gp{K}$ inflated from a cuspidal representation $\tilde{\rho}$ of $\GL_2(\vo/\vp)$, i.e., $\rho$ is trivial on $\gp{K}(\vp)$ and factors through $\tilde{\rho}$. We have
	$$ \pi \simeq \cInd_{\gp{J}}^{\GL_2} \rho. $$	
Consequently, we have
	$$ \rho \mid_{\gp{Z}} = \omega_{\pi}, \quad \cond(\omega_{\pi}) \leq 1. $$

	The character table of $\tilde{\rho}$ is given in \cite[Section 6.4.1]{BuH06} or \cite[Proposition 5.1]{FMP}. We do not need the full information about the table but its restriction to $n_-(\vo)$. We have
	$$ \Tr \tilde{\rho}(n_-(x)) = \left\{ \begin{matrix} -1 & \text{if }x \notin \vp, \\ q-1 & \text{if }x \in \vp, \end{matrix} \right. $$
from which we deduce
\begin{equation}
	\rho \mid_{n_-(\vo)} \simeq \bigoplus_{u \in \vo^{\times}/(1+\vp)} \tilde{\psi}_u, \quad \text{ with } \tilde{\psi}_u(x) := \tilde{\psi}(ux).
\label{BLrhoN-}
\end{equation}
In other words, there is an orthonormal basis $\{ e_u \}_{u \in \vo^{\times}/(1+\vp)}$ of $\rho$ such that
	$$ \rho(n_-(x)).e_u = \tilde{\psi}(ux)e_u. $$
	
		\subsubsection{Type 1 minimal supercuspidal}
		\label{T1MS}
		
	There exist $\fa_0 \in \vo^{\times}, \fa_1 \in \vo$ such that
	$$ \alpha = \begin{pmatrix} 0 & 1 \\ \fa_0 & \fa_1 \end{pmatrix} $$
generates an unramified field extension $\bL/\F$ in $\Mat_2(\F)$ with $\bL=\F[\alpha]$. Its ring of integers is equal to $\vO_{\bL} = \vo[\alpha]$. There is an integer $m \geq 0$ and a character $\lambda$ of the group $\gp{J}=\bL^{\times} \gp{K}(\vp^{m+1})$ with
	$$ \lambda \mid_{\gp{K}(\vp^{m+1})}: \gp{K}(\vp^{m+1}) \to \C^{\times}, \quad x \mapsto \tilde{\psi}(\varpi^{-2m-1} \Tr(\alpha(x-1))). $$
Note that $\lambda$ is trivial on $\gp{K}(\vp^{2m+2})$ but not trivial on $\gp{K}(\vp^{2m+1})$. We have
	$$ \pi \simeq \cInd_{\gp{J}}^{\GL_2} \lambda. $$
Consequently, we must have
	$$ \lambda \mid_{\gp{Z}} = \omega_{\pi}, \quad \cond(\omega_{\pi}) \leq 2m+2. $$
The integer $2m+1$ is called the level/depth of $\pi$. The conductor $\cond(\pi)=4m+4$.
	
	Since $\vO_{\bL}$ is optimally embedded in $\Mat_2(\vo)$ and $\bL/\F$ is unramified, we have
	$$ \bL^{\times} = \gp{Z} \vO_{\bL}^{\times}, \quad \gp{J}^0 := \gp{J} \cap \gp{K} = \vO_{\bL}^{\times} \gp{K}(\vp^{m+1}), \quad \gp{J}=\gp{Z} \gp{J}^0. $$
	
		\subsubsection{Type 2 minimal supercuspidal}
		\label{T2MS}
		
	We have $\fa_0, \fa_1, \alpha, \bL, \vO_{\bL}$ the same as in the Type 1 case. There is a character $\chi$ of $\bL^{\times}$ and an integer $m>0$ such that
	$$ \chi \mid_{1+\varpi^{m+1}\vO_{\bL}}: 1+\varpi^{m+1}\vO_{\bL} \to \C^{\times}, \quad  x \mapsto \tilde{\psi}(\varpi^{-2m}\Tr(\alpha(x-1))). $$
It determines a character $\lambda$ of $\gp{H}^1 = (1+\varpi \vO_{\bL}) \gp{K}(\vp^{m+1})$ by
	$$ \lambda: \gp{H}^1 \to \C^{\times}, ux \mapsto \chi(u) \tilde{\psi}(\varpi^{-2m}\Tr(\alpha(x-1))), \quad \forall u \in 1+\varpi\vO_{\bL}, x \in \gp{K}(\vp^{m+1}). $$
Set $\gp{A}^n := a(1+\vp^n), n \in \N$. $\lambda$ can be extended to a character
	$$ \tilde{\lambda}: \gp{A}^m \gp{H}^1 \to \C^{\times}, yx \mapsto \lambda(x), \quad \forall y \in \gp{A}^m, x \in \gp{H}^1. $$
Let $\gp{J}^1 = (1+\varpi \vO_{\bL}) \gp{K}(\vp^{m}), \gp{J}=\bL^{\times} \gp{K}(\vp^{m})$. Define
	$$ \eta = \Ind_{\gp{A}^m \gp{H}^1}^{\gp{J}^1} \tilde{\lambda}, $$
then $\eta$ is an irreducible representation of $\gp{J}^1$. It has the property that (c.f. \cite[Lemma 15.6]{BuH06})
\begin{equation}
	\eta \mid_{\gp{H}^1} \simeq \lambda^{\oplus q}.
\label{BLetaH^1}
\end{equation}
There is an irreducible representation $\rho$ of $\gp{J}$ such that
	$$ \pi \simeq \cInd_{\gp{J}}^{\GL_2} \rho, \quad \rho \mid_{\gp{J}^1} \simeq \eta. $$
Consequently, we must have
	$$ \rho \mid_{\gp{Z}} = \omega_{\pi}, \quad \cond(\omega_{\pi}) \leq 2m+1. $$
The integer $2m$ is called the level/depth of $\pi$. The conductor $\cond(\pi)=4m+2$.

	We still write $\gp{J}^0= \gp{J} \cap \gp{K} = \vO_{\bL}^{\times} \gp{K}(\vp^m)$.
	
	We need some more information about $\eta$. Since we have
	$$ \gp{J}^1 = \gp{A}^m \gp{H}^1 n(\vp^m), \quad \gp{A}^m \gp{H}^1 \cap n(\vp^m) = n(\vp^{m+1}), $$
we see that $\eta \mid_{n(\vp^m)}$ can be identified as
	$$ \eta \mid_{n(\vp^m)} \simeq \Ind_{n(\vp^{m+1})}^{n(\vp^m)} \lambda. $$
For $x\in \vp^{m+1}$, we have
	$$ \lambda(n(x)) = \tilde{\psi}(\varpi^{-2m} \fa_0 x). $$
Hence, identifying $n(\vp^m)$ with $\vp^m$, we see that $\eta$ is $\Ind_{\vp^{m+1}}^{\vp^m} \tilde{\psi}(\varpi^{-2m} \fa_0 \cdot)$. We deduce
\begin{equation}
	\eta \mid_{n(\vp^m)} \simeq \bigoplus_{u \in 1+\vp^m / 1+\vp^{m+1}} \tilde{\psi}(\varpi^{-2m} \fa_0 u\cdot).
\label{BLeta}
\end{equation}
In other words, there is an orthonormal basis $\{ e_u \}_{u \in 1+\vp^m / 1+\vp^{m+1}}$ of $\eta$ such that
	$$ \eta(n(x)).e_u = \tilde{\psi}(\varpi^{-2m} \fa_0 ux)e_u, \quad \forall x \in \vp^m. $$
	
		\subsubsection{Type 3 minimal supercuspidal}
		\label{T3MS}
		
There is $\fa_0 \in \varpi\vo^{\times}, \fa_1 \in \vp$ and
	$$ \alpha = \begin{pmatrix} 0 & 1 \\ \fa_0 & \fa_1 \end{pmatrix} $$
generates a (totally) ramified field extension $\bL / \F$ in $\Mat_2(\F)$ with $\bL = \F[\alpha]$, and the ring of integers $\vO_{\bL}=\vo[\alpha]$. Write the following Eichler order and its prime element as
	$$ \oJ = \begin{pmatrix} \vo & \vo \\ \vp & \vo \end{pmatrix}, \quad \Pi = \begin{pmatrix} 0 & 1 \\ \varpi & 0 \end{pmatrix}. $$
Let $U_{\oJ}^0=U_{\oJ}=\oJ^{\times}=\gp{K}_0(\vp)$ and $U_{\oJ}^n = 1+\Pi^n \oJ, n \geq 1$. There is an integer $m\geq 0$ and a character $\lambda$ of the group $\gp{J}=\bL^{\times}U_{\oJ}^{m+1}$ with
	$$ \lambda \mid_{U_{\oJ}^{m+1}}: U_{\oJ}^{m+1} \to \C^{\times}, \quad x \mapsto \psi(\varpi^{-m-1} \Tr(\alpha(x-1))). $$
Note that $\lambda$ is trivial on $\gp{K}(\vp^{m+2})$ but not trivial on $\gp{K}(\vp^{m+1})$. We have
	$$ \pi \simeq \cInd_J^G \lambda. $$
Consequently, we must have
	$$ \lambda \mid_{\gp{Z}} = \omega_{\pi}, \quad \cond(\omega_{\pi}) \leq m+2. $$
The half integer $(2m+1)/2$ is called the level/depth of $\pi$. The conductor $\cond(\pi)=2m+3$.

	$\vO_{\bL}$ is optimally embedded in $\oJ$ hence in $\Mat_2(\vo)$ (since $\alpha \in \oJ$). $\alpha$ is a prime element of $\vO_{\bL}$. The normalizer of $\oJ$ in $\GL_2$ is
	$$ \gp{K}_{\oJ} = \langle \Pi \rangle \ltimes \gp{K}_0(\vp). $$
We have
	$$ \bL^{\times}=\alpha^{\Z}\vO_{\bL}^{\times}, \quad \gp{J}^0=\gp{J} \cap \gp{K}_0(\vp) = \gp{J} \cap \oJ^{\times} = \vO_{\bL}^{\times}U_{\oJ}^{m+1}, \quad \gp{J}=\alpha^{\Z} \gp{J}^0. $$
	
		\subsubsection{Non minimal supercuspidal}
There is a minimal supercuspidal $\vartheta$ and a character $\chi$ of $\F^{\times}$, such that
	$$ \pi \simeq \vartheta \otimes (\chi \circ \det ), \quad \cond(\pi)=2\cond(\chi) > \cond(\vartheta). $$

	\subsection{Variants of Cartan Decomposition}
	
	Let $\F = \F_{\vp}$ at a finite place $\vp$. The following two orders
	$$ \oM = \Mat_2(\vo), \quad \oJ = \begin{pmatrix} \vo & \vo \\ \vp & \vo \end{pmatrix} $$
of $\Mat_2(\F)$ as well as their normalizer subgroup in $\GL_2(\F)$
	$$ \gp{K}_{\oM} = \gp{Z} \gp{K} = \gp{Z} \GL_2(\vo), \quad \gp{K}_{\oJ} = \Pi^{\Z} \gp{K}_0(\vp) $$
play important roles in the classification of supercuspidal representations. They also give some variants of Cartan decomposition which are useful for our purpose. We need the following mirabolic subgroups
	$$ \gp{B}_1(\vo) = \left\{ \begin{pmatrix} z & z' \\ 0 & 1 \end{pmatrix}: z \in \vo^{\times}, z' \in \vo \right\}; \quad \gp{B}_2(\vo) = \left\{ \begin{pmatrix} 1 & z' \\ 0 & z \end{pmatrix}: z \in \vo^{\times}, z' \in \vo \right\}; $$
	$$ \gp{B}_3(\vo) = \left\{ \begin{pmatrix} z & 0 \\ z' & 1 \end{pmatrix}: z \in \vo^{\times}, z' \in \vo \right\}; \quad \gp{B}_4(\vo) = \left\{ \begin{pmatrix} 1 & 0 \\ z' & z \end{pmatrix}: z \in \vo^{\times}, z' \in \vo \right\}. $$
\begin{proposition}
	Write $\GL_2 = \GL_2(\F)$ for simplicity. Denote by $a_-(y) := w a(y) w^{-1}$ and by $A^T$ the transpose of $A$.
\begin{itemize}
	\item[(1)] We have
\begin{align*}
	\GL_2 &= \sideset{}{_{l \geq 0}} \bigsqcup \gp{K}_{\oM} a_-(\varpi^l) \gp{K}_{\oM} = \sideset{}{_{l \geq 1}} \bigsqcup \gp{K}_{\oJ} a_-(\varpi^l) \gp{K}_{\oJ}^T \\
	&= \sideset{}{_{l \geq 1}} \bigsqcup \gp{K}_{\oJ} a_-(\varpi^l) \gp{K}_{\oM} = \sideset{}{_{l \geq 1}} \bigsqcup \gp{K}_{\oM} a_-(\varpi^l) \gp{K}_{\oJ}^T.
\end{align*}
	Moreover, for any $l \geq 1$, we have
	$$ \gp{K}_{\oM} a_-(\varpi^l) \gp{K}_{\oJ}^T \subset \gp{K}_{\oM} a_-(\varpi^l) \gp{K}_{\oM} \bigsqcup \gp{K}_{\oM} a_-(\varpi^{l-1}) \gp{K}_{\oM}, $$
	$$ \gp{K}_{\oJ} a_-(\varpi^l) \gp{K}_{\oJ}^T \subset \gp{K}_{\oJ} a_-(\varpi^l) \gp{K}_{\oM} \bigsqcup \begin{matrix} \gp{K}_{\oJ} a_-(\varpi^{l-1}) \gp{K}_{\oM} & \text{if } l \geq 2 \\ \emptyset & \text{if } l=1. \end{matrix} $$
	\item[(2)] If $\bL = \F[\alpha]$ is an unramifed quadratic extension of $\F$ contained in $\Mat_2(\F)$, defined as in \S \ref{T1MS} \& \ref{T2MS}, then $\bL^{\times} < \gp{K}_{\oM}$ and for any $l \geq 0$ we have
	$$ \gp{K}_{\oM} a_-(\varpi^l) \gp{K}_{\oM} = \bL^{\times} a_-(\varpi^l) \gp{K}_{\oM}, \quad \gp{K}_{\oM} a_-(\varpi^l) \gp{K}_{\oJ}^T = \bL^{\times} a_-(\varpi^l) \gp{K}_{\oJ}^T. $$
	\item[(3)] If $\bL = \F[\alpha]$ is a ramifed quadratic extension of $\F$ contained in $\Mat_2(\F)$, defined as in \S \ref{T3MS}, then $\bL^{\times} < \gp{K}_{\oJ}$ and for any $l \geq 1$ we have
	$$ \gp{K}_{\oJ} a_-(\varpi^l) \gp{K}_{\oM} = \bL^{\times} a_-(\varpi^l) \gp{K}_{\oM}, \quad \gp{K}_{\oJ} a_-(\varpi^l) \gp{K}_{\oJ}^T = \bL^{\times} a_-(\varpi^l) \gp{K}_{\oJ}^T. $$
\end{itemize}
\label{VarCartan}
\end{proposition}
\begin{proof}
\noindent (1) The double coset decomposition with respect to $\gp{K}_{\oM}$ follows directly from the usual Cartan decomposition. From the Iwahori decomposition
	$$ \gp{K} = \gp{K}_0(\vp) \bigsqcup \gp{K}_0(\vp) w \gp{B}_1(\vo) $$
	and the obvious inclusion
\begin{equation}
	a_-(\varpi^{-l}) \gp{B}_1(\vo) a_-(\varpi^l) \subset \gp{K}_0(\vp)^T \subset \gp{K}, \quad l \geq 1
\label{B1ConjInc}
\end{equation}
	we deduce the following decomposition valid for any integer $l \geq 1$
	$$ \gp{K} a_-(\varpi^l) \gp{K} = \gp{K}_0(\vp) a_-(\varpi^l) \gp{K} \bigsqcup \gp{K}_0(\vp) w a_-(\varpi^l) \gp{K}. $$
	Together with the identities valid for any integer $l \geq 1$ (since $\Pi = a_-(\varpi)w$)
	$$ \Pi^{-1} \gp{K}_0(\vp) a_-(\varpi^l) \gp{K} = \gp{K}_0(\vp) \Pi^{-1} a_-(\varpi^l) \gp{K} = \gp{K}_0(\vp) w a_-(\varpi^{l-1}) \gp{K}, $$
	$$ \Pi \gp{K}_0(\vp) w a_-(\varpi^{l-1}) \gp{K} = \gp{K}_0(\vp) \Pi w a_-(\varpi^{l-1}) \gp{K} = \gp{K}_0(\vp) a_-(\varpi^l) \gp{K}, $$
	we deduce the third decomposition with respect to $\gp{K}_{\oJ}$ and $\gp{K}_{\oM}$. The fourth with respect to $\gp{K}_{\oM}$ and $\gp{K}_{\oJ}$ follows by ``transposing'' the third one. It also implies the first relation in the ``moreover'' part. Note that the proof of the third decomposition also implies for any integer $l \geq 1$
	$$ \gp{K} a_-(\varpi^l) \gp{K}_{\oJ}^T = \gp{K}_0(\vp) a_-(\varpi^l) \gp{K}_{\oJ}^T \bigsqcup \gp{K}_0(\vp) w a_-(\varpi^l) \gp{K}_{\oJ}^T, $$
	$$ \Pi^{-1} \gp{K}_0(\vp) a_-(\varpi^l) \gp{K}_{\oJ}^T = \gp{K}_0(\vp) w a_-(\varpi^{l-1}) \gp{K}_{\oJ}^T, \quad \Pi \gp{K}_0(\vp) w a_-(\varpi^l) \gp{K}_{\oJ}^T = \gp{K}_0(\vp) a_-(\varpi^{l+1}) \gp{K}_{\oJ}^T. $$
	Together with the identity
	$$ \Pi^{-1} \gp{K}_0(\vp) a_-(\varpi) \gp{K}_{\oJ}^T = \gp{K}_0(\vp) w \gp{K}_0(\vp)^T = \gp{K}_0(\vp) a_-(\varpi) (\Pi^T)^{-1} \gp{K}_{\oJ}^T = \gp{K}_0(\vp) a_-(\varpi) \gp{K}_{\oJ}^T $$
	we deduce the second decomposition from the fourth one. Note that we can also ``transpose'' the above argument to get the second decomposition from the third one. In particular, we obtain the second relation in the ``moreover'' part.
	
\noindent (2) $\vO_{\bL}$ is optimally embedded in $\Mat_2(\vo)$, implying $\bL^{\times} = \gp{Z} \vO_{\bL}^{\times} < \gp{K}_{\oM}$. By (\ref{B1ConjInc}), both equalities follow easily from
\begin{equation}
	\gp{K}=\vO_{\bL}^{\times} \gp{B}_i(\vo), \quad \vO_{\bL}^{\times} \cap \gp{B}_i(\vo) = \{ 1 \}, \quad i = 1,2,3,4.
\label{KdecompLB1}
\end{equation}
	For any $\begin{pmatrix} a & b \\ c & d \end{pmatrix} \in \gp{K}$, we have $\min(v(a),v(c))=0$, hence $\fa_0 a + c \alpha \in \vO_{\bL}^{\times}$. Thus
\begin{align*}
	& (\fa_0 a + c \alpha)^{-1}\begin{pmatrix} a & b \\ c & d \end{pmatrix} = \begin{pmatrix} \fa_0 a & c \\ \fa_0 c & \fa_0 a + \fa_1 c \end{pmatrix}^{-1} \begin{pmatrix} a & b \\ c & d \end{pmatrix} \in \vo^{\times} \begin{pmatrix} \fa_0 a + \fa_1 c  & -c \\ -\fa_0 c & \fa_0 a \end{pmatrix} \begin{pmatrix} a & b \\ c & d \end{pmatrix} \\
	&\subseteq \vo^{\times} \begin{pmatrix} \frac{\Nr(a+\fa_0^{-1}c\alpha)}{ad-bc} & \frac{ab+\fa_0^{-1}\fa_1 bc - \fa_0^{-1} cd}{ad-bc} \\ 0 & 1 \end{pmatrix} \text{ or } \vo^{\times} \begin{pmatrix} 1 & \frac{ab+\fa_0^{-1}\fa_1 bc - \fa_0^{-1} cd}{\Nr(a+\fa_0^{-1}c\alpha)} \\ 0 & \frac{ad-bc}{\Nr(a+\fa_0^{-1}c\alpha)} \end{pmatrix},
\end{align*}
where $\Nr$ is the norm map for $\bL/\F$. The other cases follow similar argument by noting $\min(v(b),v(d))=0$ and replacing $\fa_0 a + c \alpha$ with $d-b\fa_1 + b \alpha$.

\noindent (3) It is easy to see
	$$ \Pi^{-1} \alpha = \begin{pmatrix} \varpi^{-1} \fa_0 & \varpi^{-1} \fa_1 \\ 0 & 1 \end{pmatrix} \in \gp{K}_0(\vp) \quad \Rightarrow \quad \gp{K}_{\oJ} = \alpha^{\Z} \gp{K}_0(\vp). $$
	$\vO_{\bL}$ is optimally embedded in $\oJ$, implying $\bL^{\times} = \alpha^{\Z} \vO_{\bL}^{\times} \subset \alpha^{\Z} \gp{K}_0(\vp) = \gp{K}_{\oJ}$. By (\ref{B1ConjInc}), both equalities follow easily from
\begin{equation}
	\gp{K}_0(\vp)=\vO_{\bL}^{\times} \gp{B}_i(\vo), \quad \vO_{\bL}^{\times} \cap \gp{B}_i(\vo) = \{ 1 \}, \quad i = 1,2.
\label{KdecompLB3}
\end{equation}
	The proof of (\ref{KdecompLB1}) can be moved here, replacing $\fa_0 a + c \alpha$ with $\varpi^{-1}(\fa_0 a + c \alpha) \in \vO_{\bL}^{\times}$, since $v(a)=0, v(c) \geq 1 = v(\fa_0)$ in this case.
\end{proof}
\begin{remark}
	The following descriptions of some double cosets will be useful.
	$$ \gp{K} a_-(\varpi^l) \gp{K} = \left\{ A \in \oM - \varpi \oM \ |\ \det A \in \varpi^l \vo^{\times} \right\}, \quad l \geq 0; $$
	$$ \gp{K}_0(\vp) a_-(\varpi^l) \gp{K} = \gp{K} a_-(\varpi^l) \gp{K} \cap \Pi \gp{K} a_-(\varpi^{l-1}) \gp{K}, \quad l \geq 1. $$
	The first one is classical and follows from the usual Cartan decomposition. The second one follows easily from the above proof of Proposition \ref{VarCartan} (1).
\label{DCDes}
\end{remark}

	Recall our classification in \S \ref{ClEmb} of embeddings of $\E \hookrightarrow \Mat_2(\F)$, in particular the fact that $\iota_0(\E^{\times}) < \gp{K}_{\oM}$ if $\E$ is unramifed and $\iota_0(\E^{\times}) < \gp{K}_{\oJ}$ if $\E$ is ramifed (proved in the same way for $\bL^{\times}$). We propose to refine the decomposition in Propositions \ref{VarCartan} taking into account the multiplication by $\iota_0(\E^{\times})$.
\begin{lemma}
	Let $l \geq e(\E/\F) \in \{ 1, 2 \}$ be an integer. Let $\bL$ be as in the classification of supercuspidal representations and $\gp{J}, \gp{J}^0$ associated with it. We have
	$$ \begin{matrix} \gp{K}_{\oM} \\ \gp{K}_{\oJ} \end{matrix} a_-(\varpi^l) \begin{matrix} \gp{K}_{\oM} \\ \gp{K}_{\oJ}^T \end{matrix} = \sideset{}{_{u \in \vo^{\times}}} \bigsqcup \bL^{\times} a_-(\varpi^l u) \iota_0(\E^{\times}) = \sideset{}{_{z}} \bigsqcup \gp{J} a_-(\varpi^l z) \iota_0(\E^{\times}), $$
	where we choose the left resp. right group on the LHS as $\gp{K}_{\oM}$ or $\gp{K}_{\oJ}$ resp. $\gp{K}_{\oJ}^T$ to contain $\bL^{\times}$ resp. $\iota_0(\E^{\times})$, and $z$ runs over a system of representatives of $\vo^{\times}/(1+\vp^k)$ with $k$ determined by:
\begin{itemize}
	\item[(1)] $k=m \geq 1$ in the Type 1 resp. Type 2 minimal supercuspidal case of level $2m-1$ resp. $2m$;
	\item[(2)] $k = \IntP{m/2} + 1$ in the Type 3 minimal supercuspidal case of level $m+1/2$.
\end{itemize}
\label{KdecompJ^0}
\end{lemma}
\begin{proof}
	We only treat the decomposition with respect to $z$, the case with respect to $u$ being simpler. The argument for Proposition \ref{VarCartan} (2) \& (3) can be ``transposed'' to yield (generalizing (\ref{IwaforE}))
	$$ \begin{matrix} \gp{K}_{\oM} \\ \gp{K}_{\oJ} \end{matrix} a_-(\varpi^l) \begin{matrix} \gp{K}_{\oM} \\ \gp{K}_{\oJ}^T \end{matrix} = \begin{matrix} \gp{K}_{\oM} \\ \gp{K}_{\oJ} \end{matrix} a_-(\varpi^l) \iota_0(\E^{\times}). $$
	Since we have an identity of double coset decompositions
	$$ \gp{J} \backslash \begin{matrix} \gp{K}_{\oM} \\ \gp{K}_{\oJ} \end{matrix} a_-(\varpi^l) \iota_0(\E^{\times}) / \iota_0(\E^{\times}) \simeq \gp{J}^0 \backslash \begin{matrix} \gp{K} \\ \gp{K}_0(\vp) \end{matrix} / \begin{matrix} \iota_l(\vO_l^{\times}) \\ \iota_l(\vO_{l-1}^{\times}) \end{matrix}, $$
	where $\iota_l(\vO_{l-1}^{\times})$ replaces $\iota_l(\vO_l^{\times})$ if $\E$ is ramified, it suffices to identify a system of representatives for the RHS as $a_-(z)$.
	
\noindent First suppose $\bL$ is unramified. The proofs for Type 1 and Type 2 are the same, hence we may assume Type 1. By (\ref{KdecompLB1}), we have a projection map
	$$ \tau: \gp{K} \simeq \vO_{\bL}^{\times} \times \gp{B}_1(\vo) \to \gp{B}_1(\vo). $$
	Choosing a system of representatives $[\gp{B}_1(\vo)]$ of $\gp{B}_1(\vo) \pmod{\vp^m}$, we get a decomposition
\begin{equation}
	\gp{K} = \sideset{}{_{b \in [\gp{B}_1(\vo)]}} \bigsqcup \gp{J}^0b
\label{KtoJ}
\end{equation}
as well as a map
	$$ \tilde{\tau}: \gp{K} \to [\gp{B}_1(\vo)] $$
determined by $\tilde{\tau}(\kappa) = \tau(\kappa) \pmod{\vp^{m}}$. Let $s=l+1-e(\E/\F)$. If $z_1,z_2 \in \vo^{\times}, t_1, t_2 \in \vO_s^{\times}$ are such that $\tilde{\tau}(a_-(z_1)\iota_l(t_1)) = \tilde{\tau}(a_-(z_2)\iota_l(t_2))$, then $a_-(z_1)\iota_l(t_1t_2^{-1})a_-(z_2)^{-1} \in \gp{J}^0$, i.e., $\tilde{\tau}(a_-(z_1)\iota_l(t_1t_2^{-1})a_-(z_2)^{-1})=1 $, or 
	$$ \tau(a_-(z_1)\iota_l(t_1t_2^{-1})a_-(z_2)^{-1})=1 \pmod{\vp^m}. $$
Writing $t=x+\varpi^s y \beta$ with $x\in \vo^{\times}, y \in \vo$ and using the formula proving (\ref{KdecompLB1}), we identify the $(1,2)$-term of the image of
	$$ a_-(z_1)\iota_l(t)a_-(z_2)^{-1} = \begin{pmatrix} x & \fa \varpi^{s-l} y z_2^{-1} \\ - \varpi^{l+s} y z_1 & (x+\fb \varpi^s y) z_1 z_2^{-1} \end{pmatrix} =: \begin{pmatrix} a & b \\ c & d \end{pmatrix} $$
under $\tau$ as the product of
	$$ ab+\fa_0^{-1}\fa_1 bc - \fa_0^{-1} cd \in b \vo^{\times}, $$
since $\fa_0^{-1} c/b = -\varpi^{2l} \fa_0^{-1} \fa^{-1} z_1 z_2 \in \vp, c \in \vp, a \& d \in \vo^{\times}$, and the inverse of
	$$ ad-bc = \Nr_{\E}(x+\varpi^s y \beta) z_1 z_2^{-1} \in \vo^{\times}. $$
	Its lying in $\vp^m$ implies $b = \fa \varpi^{s-l} y z_2^{-1} \in \vp^m$, which is equivalent to $y \in \vp^m$ since $v(\fa) = e(\E/\F)-1$. We can also identify the (1,1)-term as the product of $z_2 z_1^{-1}$ and
	$$ \Nr_{\bL}(x - \fa_0^{-1} \varpi^{l+s} y z_1 \alpha) \Nr_{\E}(x+\varpi^s y \beta)^{-1} \in 1+\vp^m, $$
since $l+s-1 = 2l-e(\E/\F) \geq e(\E/\F) \geq 1, s \geq 1$ and $y \in \vp^m$. Its lying in $1+\vp^m$ implies $z_1 = z_2$. Thus $\tilde{\tau}$ induces an injection
	$$ \vo^{\times}/(1+\vp^{m}) \times \vO_s^{\times}/\vO_{s+m}^{\times} \to [\gp{B}_1(\vo)], \quad (z,t) \mapsto \tilde{\tau}(a_-(z)\iota_l(t)). $$
	Both sides being finite and having the same cardinality, it must be surjective as well. Since $a_-(1+\vp^{m}), \iota_l(\vO_{s+m}^{\times}) \subset \vo^{\times} \gp{K}(\vp^{m})$, we conclude by (\ref{KtoJ}).
	
\noindent Now suppose $\bL$ is ramified. We should replace (\ref{KdecompLB1}) with (\ref{KdecompLB3}) to get
	$$ \tau: \gp{K}_0(\vp) \simeq \vO_{\bL}^{\times} \times \gp{B}_1(\vo) \to \gp{B}_1(\vo). $$
	We should also re-define $[\gp{B}_1(\vo)]$ as a system of representatives for 
	$$ (U_{\oJ}^{m+1} \cap \gp{B}_1(\vo)) \backslash \gp{B}_1(\vo). $$
	We then easily verify that the above argument also works through, yielding an injection
	$$ \vo^{\times}/(1+\vp^{\IntP{m/2}+1}) \times \vO_s^{\times}/\vO_{s+\IntCP{m/2}}^{\times} \to [\gp{B}_1(\vo)], \quad (z,t) \mapsto \tilde{\tau}(a_-(z)\iota_l(t)). $$
	We then conclude the proof the same way.
\end{proof}

	\subsection{Some Comparison of Measures}
	
	Two different normalizations of the Haar measure on $\GL_r$ ($r \geq 1$) are frequently used in the literature: the Tamagawa measure and the \emph{hyperbolic} measure. We recall their definitions in an explicit descriptive way as follows.
	
\noindent The Tamagawa measure \cite[Chapter \Rmnum{2}]{We82}  is defined via $\GL_r$-invariant differential forms of highest order defined over $\F$ together with a choice of (local) convergence factors. It is independent of the choice of the differential form. The choice of convergence factors are usually standard. We identify the additive group of $\Mat_r(\F)$ as $\F^{r^2}$, with the standard translation invariant form $dx$. The local component $dg_v$ on $\GL_r(\F_v)$ of the Tamagawa measure $dg := \sideset{}{_v} \prod dg_v$ is chosen as
	$$ dg_v := \zeta_v(1) \norm[\det x_v]_v^{-r} dx_v. $$
	
\noindent Recall the Iwasawa decomposition
	$$ \mathrm{Iw}: \gp{Z}_r \times \gp{N}_r \times \gp{A}_{r-1} \times \gp{K}_r \to \GL_r, \quad (z,n,a,\kappa) \mapsto zn \begin{pmatrix} a & \\ & 1 \end{pmatrix} \kappa, $$
	where $\gp{Z}_r \simeq \ag{G}_{{\rm m}}$ is the center of $\GL_r$, $\gp{N}_r$ is the upper triangular unipotent subgroup of $\GL_r$, $\gp{A}_{r-1} \simeq \ag{G}_{{\rm m}}^{r-1}$ is the diagonal torus of $\GL_{r-1}$ and $\gp{K}_r$ is the standard maximal compact subgroup which is not an algebraic group and is locally equal to $\SU_r(\ag{C})$ resp. $\SO_r(\R)$ resp. $\GL_r(\vo_{\vp})$ on a complex resp. real resp. finite place. We equip each group on the LHS locally with a measure as follows:
\begin{itemize}
	\item On $\gp{Z}_r \simeq \ag{G}_{{\rm m}}$, we transport the Tamagawa measure on $\ag{G}_{{\rm m}}$, which is $d^{\times}z_v = \zeta_v(1)^{-1} \norm[z_v]_v^{-1} dz_v$ on $\F_v^{\times}$. Similarly on $\gp{A}_{r-1} \simeq \ag{G}_{{\rm m}}^{r-1}$, we transport the $r-1$-power of the Tamagawa measure on $\ag{G}_{{\rm m}}$ and denote it by $da_v$ on $\gp{A}_{r-1}(\F_v)$.
	\item On $\gp{N}_r(\F_v)$ for any place $v$, we have an obvious canonical homeomorphism $\gp{N}_r(\F_v) \simeq \F_v^{r(r-1)/2}$. We transport the additive Haar measure on $\F_v^{r(r-1)/2}$ to $\gp{N}_r(\F_v)$ and denote it by $dn_v$. It is a (left and right) Haar measure of $\gp{N}_r(\F_v)$.
	\item On $\gp{K}_{r,v}$, we take the probability Haar measure $d\kappa_v$.
\end{itemize}
	We define a modulus function $\delta = \sideset{}{_v} \prod \delta_v$ on $\gp{A}_r(\A)$ with
	$$ \delta_v : \gp{A}_r(\F_v) \to \R_+, \quad \mathrm{diag}(a_1,a_2,\cdots, a_{r-1},1) \mapsto \sideset{}{_{k=1}^{r-1}} \prod \norm[a_k]_v^{r-2k+1}. $$
	The hyperbolic measure on $\GL_r(\F_v)$ is defined to be the image measure under the Iwasawa decomposition map $\mathrm{Iw}$ of
	$$ d^{\times}z_v \cdot dn_v \cdot \delta_v(a_v)^{-1} da_v \cdot d\kappa_v. $$
	The hyperbolic measure on $\GL_r(\A)$ is the product of the local ones.
	
\begin{proposition}
	We have the relation
	$$ dg = c_r \cdot d^{\times}z \cdot dn \cdot \delta(a)^{-1} da \cdot d\kappa, \quad c_r = \Dis(\F)^{-\frac{r(r-1)}{4}} \sideset{}{_{j=2}^r} \prod \Lambda_{\F}(j)^{-1}. $$
\label{MeasCompar}
\end{proposition}
\begin{proof}
	We calculate place by place $c_{r,v} > 0$, the quotient of the two local measures determined by
	$$ dg_v = c_{r,v} \cdot d^{\times}z_v \cdot dn_v \cdot \delta(a_v)^{-1} da_v \cdot dk_v. $$

\noindent (1) At $v = \vp <\infty$, we have
	$$ {\rm Vol}(\gp{K}_{r,\vp}, dg) = {\rm Vol}(\Mat_r(\vo_{\vp}), dx) \zeta_v(1) \frac{\norm[ \GL_r(\F_q) ]}{\norm[ \Mat_r(\F_q) ]} = \Cond(\psi_{\vp})^{-\frac{r^2}{2}} \sideset{}{_{k=2}^{r}} \prod \zeta_{\vp}(k)^{-1}. $$
	On the other hand, under the map $\mathrm{Iw}$, the preimage of $\gp{K}_{r,\vp}$ is $\gp{Z}_r(\vo_{\vp}) \times \gp{N}_r(\vo_{\vp}) \times \gp{A}_{r-1}(\vo_{\vp}) \times \gp{K}_{r,\vp}$, hence is of total mass $\Cond(\psi_{\vp})^{-1/2} \Cond(\psi_{\vp})^{-r(r-1)/4} \Cond(\psi_{\vp})^{-(r-1)/2}$. Consequently we have
	$$ c_{r,\vp} = \Cond(\psi_{\vp})^{-\frac{r(r-1)}{4}} \sideset{}{_{k=2}^{r}} \prod \zeta_{\vp}(k)^{-1}. $$

\noindent (2) At $\F_v=\R$, we take a function $f$ defined by
	$$ f: \GL_r(\R) \to \C, \quad X \mapsto \exp(-{\rm Tr}(XX^t)) |\det X|^r. $$
	On the one hand, we have
	$$ \int_{\GL_r(\R)} f(X) \zeta_v(1) \frac{dX}{|\det X|^r} = \int_{\Mat_r(\R)} \exp(-{\rm Tr}(XX^t)) dX = \pi^{\frac{r^2}{2}}. $$
	On the other hand, writing $X=zna \kappa$ with 
	$$ n=(x_{i,j}), x_{i,j}=0 \text{ for } i>j, \quad x_{i,i}=1; \quad a={\rm diag}(a_1, \cdots, a_{r-1}, 1); $$
	we obtain
	$$ {\rm Tr}(XX^t) = z^2 \left( 1+\sideset{}{_{j=1}^{r-1}} \sum a_j^2 \right) + z^2 \sideset{}{_{i=1}^{r-1}} \sum x_{i,r}^2 + z^2 \sideset{}{_{1\leq i<j<r}} \sum a_j^2 x_{i,j}^2; \quad |\det X| = |z|^r \sideset{}{_{j=1}^{r-1}} \prod |a_j|. $$
	Hence, we can calculate the integral in another way as
\begin{align*}
	&\quad c_{r,v} \int_{\R^{\times}} \int_{\left(\R^{\times} \right)^{r-1}} \int_{\R^{\frac{r(r-1)}{2}}} \exp(-{\rm Tr}(XX^t)) |z|^{r^2} \sideset{}{_{j=1}^{r-1}} \prod |a_j|^{2j-1} dx da d^{\times}z \\
	&= c_{r,v} \pi^{\frac{r(r-1)}{4}} \int_{\R^{\times}} \int_{\left(\R^{\times} \right)^{r-1}} |z|^{\frac{r(r+1)}{2}} \sideset{}{_{j=1}^{r-1}} \prod |a_j|^j \exp(-z^2\left( 1+\sideset{}{_{j=1}^{r-1}} \sum a_j^2 \right)) da d^{\times}z \\
	&= c_{r,v} \pi^{\frac{r(r-1)}{4}} \sideset{}{_{j=1}^r} \prod \Gamma(\frac{j}{2}),
\end{align*}
	$$ \Rightarrow \quad c_{r,v}= \prod_{j=2}^{r} \Gamma_{\R}(j)^{-1} = \prod_{j=2}^{r} \zeta_v(j)^{-1}. $$

\noindent (3) At $\F_v=\C$, we do the calculation similar to the real case by taking
	$$ f: \GL_r(\C) \to \C, \quad X \mapsto \exp(-{\rm Tr}(X \bar{X}^t)) |\det X|^{2r}, $$
	and get
	$$ c_{r,v} = \sideset{}{_{j=2}^r} \prod \Gamma_{\C}(j)^{-1} = \sideset{}{_{j=2}^{r}} \prod \zeta_v(j)^{-1}. $$

\noindent We conclude by noting $\sideset{}{_{\vp < \infty}} \prod \Cond(\psi_{\vp})=\Dis(\F)$.
\end{proof}

	If $\E/\F$ is a quadratic field extension, let $\gp{T}$ be an $\F$-torus defined by the image  of $\E^{\times}$ under an(any) $\F$-embedding $\E \hookrightarrow \Mat_2(\F)$. If $\gp{B}_1$ denotes the upper triangular group (mirabolic/Heisenberg group)
	$$ \gp{B}_1(R) = \left\{ \begin{pmatrix} y & x \\ & 1 \end{pmatrix} \middle| y \in R^{\times}, x \in R \right\}, $$
	then we have a decomposition realized via matrix multiplication
\begin{equation}
	\gp{B}_1(\F) \times \gp{T}(\F) \simeq \GL_2(\F), \quad (b,t) \mapsto b \cdot t,
\label{IwaToric}
\end{equation}
	which extends by continuity to any completion $\F_v$ of $\F$ as an injective map
	$$ i_v: \gp{B}_1(\F_v) \times \gp{T}(\F_v) \to \GL_2(\F), $$
	the complement of whose image has measure $0$. Let $db_v$ denote the left Haar measure
	$$ db_v = dx_v \cdot \zeta_v(1)^{-1} \norm[y_v]_v^{-1} d^{\times}y_v. $$
	We transport the local Tamagawa measure of $\E_v^{\times}$ onto $\gp{T}(\F_v) \simeq \E_v^{\times}$ and denote it by $dt_v$, where $\E_v := \E \otimes_{\F} \F_v$. It is expected that $db_v dt_v$ coincides with a Haar measure of $\GL_2(\F_v)$ on the image of $i_v$.
\begin{lemma}
	On the intersection of $\GL_2(\A)$ with the image of the product map of $i_v$, the Tamagawa measure $dg$ is related to the product measure as
	$$ dg = \sideset{}{_v} \prod \zeta_v(1)^{-1} L_v(1,\eta_v)^{-1} db_v dt_v, $$
	where $\eta$ is the quadratic character associated with $\E/\F$.
\label{IwaTorHaar}
\end{lemma}
\begin{proof}
	Assume $\E=\F[\sqrt{d}]$ with $d\in \F^{\times}-(\F^{\times})^2$. Note that the global relation is independent of the choice of the embedding, because any two such embeddings are conjugate to each other by an element of $\GL_2(\F)$ hence of $\gp{B}_1(\F)$, and the global modulus function $\delta$ of $\gp{B}_1$ is trivial on $\gp{B}_1(\F)$. Hence it suffices to prove the result for a specific embedding, such as
	$$ \iota(\sqrt{d}) = \begin{pmatrix} 0 & 1 \\ d & 0 \end{pmatrix}. $$
	Then we get the coordinates $(x,y,s,t)$ on $\GL_2$ defined by
	$$ \begin{pmatrix} y & x \\ 0 & 1 \end{pmatrix} \begin{pmatrix} s & t \\ dt & s \end{pmatrix} = \begin{pmatrix} ys+dxt & yt+xs \\ dt & s \end{pmatrix}. $$
	Hence locally we have by definition of the Tamagawa measure
	$$ dg_v = \zeta_v(1)\frac{|d|_v dxdydsdt}{|y|_v^2 |s^2-dt^2|_v} = |d|_v\zeta_{\E,v}(1)^{-1} db_v dt_v. $$
	But we have $\sideset{}{_v} \prod |d|_v=1$ and $\zeta_{\E,v}(s)=\zeta_v(s)L_v(s,\eta_v)$, from which we conclude the proof.
\end{proof}

	\subsection{Wielonsky Formula}
	
	We need an analogue of Waldspurger's Formula for Eisenstein series, obtained first by Wielonsky \cite{Wi85}. We present a variant, which can be viewed as a generalization of \cite[Lemma 2.8]{Wu14} to non-split torus.
	
\noindent Let $\pi=\pi_{s,\xi} = \Ind_{\gp{B}(\A)}^{\GL_2(\A)} (\xi |\cdot|^s, \xi |\cdot|^{-s})$ where $\xi$ is a Hecke character of $\F^{\times} \backslash \A^{\times}$ and $s \in i\R$. Recall the Eisenstein series defined for a flat section $f_s \in \pi_{s,\xi}$ (with $f=f_0\in \pi_{0,\xi}$) by
	$$ \eis(s,f)(g) = \sum_{\gamma \in \gp{B}(\F) \backslash \GL_2(\F)} f_s(\gamma g) = \sum_{\gamma \in \gp{Z}(\F) \backslash \gp{T}(\F)} f_s(\gamma g), \quad \forall g \in \GL_2(\A), $$
where the last equality is due to (\ref{IwaToric}). Hence we have
	$$ \ell(\eis(s,f);\iota,\Omega) = \int_{\gp{Z}(\A) \gp{T}(\F) \backslash \gp{T}(\A)} \sum_{\gamma \in \gp{Z}(\F) \backslash \gp{T}(\F)} f_s(\gamma t) \Omega(t) dt = \int_{\gp{Z}(\A) \backslash \gp{T}(\A)} f_s(t) \Omega(t) d^{\times}t, $$
	where we have written $d^{\times}t$ in order to distinguish it from $dt$ on $\gp{T}(\A)$. Consequently, we get
\begin{align*}
	|\ell(\eis(s,f);\iota,\Omega)|^2 &= \sideset{}{_v} \prod \int_{(\gp{Z}(\F_v) \backslash \gp{T}(\F_v))^2} f_{v,s}(t_1) \overline{f_{v,s}(t_2)} \Omega_v(t_1t_2^{-1}) d^{\times}t_1 d^{\times}t_2 \\
	&= \sideset{}{_v} \prod \int_{\gp{Z}(\F_v) \backslash \gp{T}(\F_v)} \langle \pi_v(t).f_{v,s}, f_{v,s} \rangle_{v,\gp{T}} \Omega_v(t) d^{\times}t,
\end{align*}
	where $\langle \cdot, \cdot \rangle_{v,\gp{T}}$ is the hermitian form on $\pi_v$ defined by
	$$ \langle f_1, f_2 \rangle_{v,\gp{T}} = \int_{\gp{Z}(\F_v) \backslash \gp{T}(\F_v)} f_1(t) \overline{f_2(t)} d^{\times}t, \quad \forall f_1,f_2 \in \pi_v. $$
	Lemma \ref{IwaTorHaar} together with Proposition \ref{MeasCompar} implies that on $\gp{Z}(\A) \backslash \gp{T}(\A) \subset_{\text{full measure}} \gp{B}(\A) \backslash \GL_2(\A) \simeq (\gp{B}(\A) \cap \gp{K}) \backslash \gp{K}$, we have
	$$ \sideset{}{_v} \prod d^{\times}t_v = \sideset{}{_v} \prod \zeta_v(1)L(1,\eta_v) \frac{dg_v}{dz_v db_v} = c_2 \cdot \sideset{}{_v} \prod \zeta_v(1)L(1,\eta_v) d\kappa_v. $$
	Consequently, we get
	$$ \sideset{}{_v} \prod \langle f_{1,v},f_{2,v} \rangle_{v,\gp{T}} = \Dis(\F)^{-1/2} \Lambda_{\F}(2)^{-1} \sideset{}{_v} \prod \zeta_v(1)L(1,\eta_v) \int_{\gp{K}_v} f_{1,v}(\kappa_v) \overline{f_{2,v}(\kappa_v)} d\kappa_v. $$
	Recall the Eisenstein norm \cite[Lemma 2.8]{Wu14} for $s \in i\R$ defined by
	$$ \lVert \eis(s,f) \rVert_{{\rm Eis}}^2 = \int_{\gp{K}} |f_s(\kappa)|^2 d\kappa = \sideset{}{_v} \prod \int_{\gp{K}_v} |f_{v,s}(\kappa_v)|^2 d\kappa_v. $$
	We have proved
\begin{proposition}
	Let $f=\otimes_v' f_v \in \pi=\pi_{0,\xi}$ be decomposable in the induced model such that $f_s \in \pi_{s,\xi}$ defines a flat section for $s\in \C$. Then for any $s\in i\R$, we have
\begin{align*}
	\frac{|\ell(\eis(s,f); \Omega,\iota)|^2}{\lVert \eis(s,f) \rVert_{{\rm Eis}}^2} &= \frac{L(\frac{1}{2}+s,\xi_{\E} \otimes \Omega) L(\frac{1}{2}-s,\xi_{\E}^{-1} \otimes \Omega)}{\Dis(\F)^{\frac{1}{2}} L(1+2s,\xi^2) L(1-2s,\xi^{-2})} \cdot \sideset{}{_{v \mid \infty}} \prod \frac{\zeta_{\E,v}(1)}{\zeta_v(2)} \cdot \\
	&\quad \sideset{}{_{v \mid \infty}} \prod \alpha_v(f_{s,v}; \Omega_v, \iota_v) \cdot \sideset{}{_{\vp < \infty}} \prod \tilde{\alpha}_{\vp}(f_{s,\vp}; \Omega_{\vp}, \iota_{\vp}),
\end{align*}
	where the notations are
\begin{itemize}
	\item $\xi_{\E}$ resp. $\xi_{\E,v}$ is the base change of $\xi$ resp. $\xi_v$ to $\A_{\E}^{\times}$ resp. $\E_v^{\times}$, i.e., $\xi_{\E} = \xi \circ \Nr_{\F}^{\E}$ resp. $\xi_{\E,v} = \xi_v \circ \Nr_{\F_v}^{\E_v}$.
	\item $ \alpha_v(f_{s,v}; \Omega_v, \iota_v) = \int_{\gp{Z}(\F_v) \backslash \gp{T}(\F_v)} \frac{\langle \pi_{s,\xi_v}(t).f_{s,v}, f_{s,v} \rangle_v}{\langle f_{s,v}, f_{s,v} \rangle_v} \Omega_v(t) dt $ with $\langle \cdot, \cdot \rangle_v$ any $\GL_2(\F_v)$-invariant inner product on $\pi_{s,\xi_v}$ for every place $v$.
	\item $ \tilde{\alpha}_{\vp}(f_{s,\vp}; \Omega_{\vp}, \iota_{\vp}) = \frac{\zeta_{\vp}(1)L_{\vp}(1,\eta_{\vp}) L_{\vp}(1+2s,\xi_{\vp}^2) L_{\vp}(1-2s,\xi_{\vp}^{-2})}{\zeta_{\vp}(2) L_{\vp}(1/2+s,\xi_{\E,\vp} \otimes \Omega_{\vp}) L_{\vp}(1/2-s,\xi_{\E,\vp}^{-1} \otimes \Omega_{\vp})} \alpha_{\vp}(f_{s,\vp}; \Omega_{\vp}, \iota_{\vp}) $.
\end{itemize}
\label{WieF}
\end{proposition}
\begin{remark}
	The formula extends to $s \in \C$ if we repalce on the left hand side $|\ell(\eis(s,f); \Omega,\iota)|^2$ by $\ell(\eis(s,f);\Omega,\iota) \ell(\eis(-s,\bar{f});\Omega^{-1},\iota)$, and replace/extend $\langle \cdot, \cdot \rangle_v$ on the right hand side by a $\GL_2(\F_v)$-invariant pairing on $\pi_{s,\xi_v} \times \pi_{s,\xi_v}^{\vee} \simeq \pi_{s,\xi_v} \times \pi_{-s,\xi_v^{-1}}$.
\end{remark}

\section{Local Estimations}

	We omit the subscript $v$ since we work locally in this section. We assume $\B$ belongs to $(\pi,\E,\Omega)$ in all the following statements and write $\pi' = \JL(\pi;\B)$ for abbreviation. We choose an additive character $\psi$ of $\F$ with (logarithmic) conductor $\cond(\psi)=0$ at every finite place. Note that this is different from the convention made in references on local Langlands such as \cite{BuH06}. We will cite these references but adjust the argument accordingly. The change to the normalization \`a la Tate only results in some polynomial dependence on the discriminant $\Dis(\F)$ in the final bounds, which we do not care in this paper. We may only consider the embeddings $\iota$ classified in \S \ref{ClassEmb} due to Lemma \ref{FixEmb}. Our central interest is the \emph{size of the local factors} $\alpha(\cdot), \tilde{\alpha}(\cdot)$ appearing in and defined above (\ref{WaFSV}).
	
	\subsection{$\B$-nonsplit Place}
	
	In this case $\B$ is a division quaternion algebra over $\F$, $\pi$ is a square-integrable representation of $\GL_2(\F)$ and $\pi'$ is finite dimensional. $\E$ is also non-split. We would like to choose the subspace $\sigma$ to be the whole space $\pi'$. For the final polynomial dependence on $\pi$, we need to make some clarifications by showing that the (analytic) conductor $\Cond(\pi)$ ``controls'' everything. Since the local $L$-functions and the epsilon-factors of $\pi$ and $\pi'$ coincide (see \cite[Theorem 4.3 (\rmnum{5})]{JL70} and the discussion around \cite[Proposition 5.20]{JL70}), we have $\Cond(\pi') = \Cond(\pi)$.
	
\noindent In the case $v \mid \infty$, we must have $\F=\R$ and $\B$ is the Hamiltonian quaternion algebra. The explicit formula of $L(s,\pi')$ given in the proof of \cite[Proposition 5.20]{JL70} together with the definition of (local) analytic conductors \cite[(5.6)]{IK04} immediately implies the following bound.
\begin{proposition}
	For $v \mid \infty$, the dimension $d$ of $\pi'$ satisfies $d \ll \Cond(\pi')^{1/2}$ with absolute implied constant. The conductor of $\Omega$ satisfies $\Cond(\Omega) \leq \Cond(\pi')$.
\label{QuatCondBdInf}
\end{proposition}
	
\noindent In the case $v=\vp<\infty$, recall that $\B$ admits a ring of integers $\VO$ which is a discrete valuation ring with valuation $v_{\B}$ and uniformizer $\varpi_{\B}$ such that $v_{\B}(\varpi)=2$. The character $\psi_{\B} := \psi \circ \mathrm{Tr}_{\B/\F}$ with reduced trace $\mathrm{Tr}_{\B/\F}$ has \emph{level minus one} \cite[Lemma \S 53.4.2]{BuH06}, i.e., trivial on $\varpi_{\B}^{-1}\VO$ but not on $\varpi_{\B}^{-2}\VO$.
\begin{definition}
	For $v=\vp < \infty$, the \emph{level} $\ell(\pi')$ of $\pi'$ is the least integer $n \geq 0$ such that
	$$ U_{\B}^n := (1 + \varpi_{\B}^n \VO) \cap \VO^{\times} \subset \mathrm{Ker}(\pi'). $$
\end{definition}
\begin{proposition}
	For $v=\vp < \infty$, we have $\cond(\pi') = \ell(\pi')+1$. Consequently, the dimension $d$ of $\pi'$ satisfies $d \ll q^{-1} \Cond(\pi')$ with absolute implied constant. The conductor of $\Omega$ satisfies $\Cond(\Omega) \leq q^{-1}\Cond(\pi')$.
\label{QuatCondBdFin}
\end{proposition}
\begin{proof}
	Write $n = \ell(\pi')$. We first deal with the first assertion. If $n=0$, then $\pi'$ is one dimensional and $\pi$ must be (twisted) Steinberg. Hence the assertion is verified. We assume $n \geq 1$ in the sequel. It suffices to determine the form of $\varepsilon(s, \pi', \psi)$, which is equal to ``$-\beta(1)$'' in the proof of \cite[Lemma 4.2.5]{JL70}. It can be inferred from the proof that
	$$ \beta(x) = \norm[x]_{\B} \int_{\B} \psi_{\B}(-xy) \norm[y]_{\B}^{-\frac{s}{2}-\frac{1}{4}} \pi'(y)^{-1} dy, $$
where the integral is interpreted in the regularized / Cauchy principal sense. In particular
\begin{align*}
	\beta(1) &= \int_{\B} \psi_{\B}(-y) \norm[y]_{\B}^{-\frac{s}{2}-\frac{1}{4}} \pi'(y)^{-1} dy \\
	&= \sum_{m \in \Z} \int_{\varpi_{\B}^m \VO^{\times}} \psi_{\B}(-y) \pi'(y)^{-1} dy \cdot q^{m(s+\frac{1}{2})}.
\end{align*}
	If $m < -n-1$, then we have
\begin{align*}
	\int_{\varpi_{\B}^m \VO^{\times}} \psi_{\B}(-y) \pi'(y)^{-1} dy &= \frac{1}{\Vol(U_{\B}^n)} \int_{U_{\B}^n} \int_{\varpi_{\B}^m \VO^{\times}} \psi_{\B}(-y) \pi'(uy)^{-1} dy d^{\times}u \\
	&= \frac{1}{\Vol(U_{\B}^n)} \int_{U_{\B}^n} \int_{\varpi_{\B}^m \VO^{\times}} \psi_{\B}(-uy) \pi'(y)^{-1} dy d^{\times}u.
\end{align*}
	But, up to a zeta factor, we have
	$$ \int_{U_{\B}^n} \psi_{\B}(-uy) d^{\times}u = \psi_{\B}(-y) \int_{\varpi_{\B}^n \VO} \psi_{\B}(-vy) dy = 0, $$
	since $-vy$ runs over $\varpi_{\B}^{n+m} \VO \supset \varpi_{\B}^{-2} \VO$ on which $\psi_{\B}$ is non trivial. Hence the terms for $m < -n-1$ vanish. On the other hand, if $m > -n-1$, then we have $\psi_{\B}(-(u-1)y)=1$ for $u \in U_{\B}^{n-1}, y \in \varpi_{\B}^m \VO^{\times}$,
\begin{align*}
	\int_{\varpi_{\B}^m \VO^{\times}} \psi_{\B}(-y) \pi'(y)^{-1} dy &= \frac{1}{\Vol(U_{\B}^{n-1})} \int_{U_{\B}^{n-1}} \int_{\varpi_{\B}^m \VO^{\times}} \psi_{\B}(-uy) \pi'(y)^{-1} dy d^{\times}u \\
	&= \frac{1}{\Vol(U_{\B}^{n-1})} \int_{U_{\B}^{n-1}} \int_{\varpi_{\B}^m \VO^{\times}} \psi_{\B}(-y) \pi'(y)^{-1} \pi'(u) dy d^{\times}u.
\end{align*}
	We claim that
	$$ \int_{U_{\B}^{n-1}} \pi'(u) d^{\times} u = 0. $$
	Otherwise, the image of the operator defined by the LHS of the above equation is non zero, thus the subspace $V$ of fixed vectors by $U_{\B}^{n-1}$ in $\pi'$ is non zero. As $U_{\B}^{n-1}$ is normal in $\B^{\times}$ and $\pi'$ is irreducible, we deduce $V = \pi'$. Hence $\pi'$ is trivial on $U_{\B}^{n-1}$, contradicting the hypothesis on the level of $\pi'$. Hence the terms for $m > -n-1$ also vanish. We get and conclude the first assertion by
	$$ -\varepsilon(s,\pi',\psi) = \beta(1) = \int_{\varpi_{\B}^{-(n+1)} \VO^{\times}} \psi_{\B}(-y) \pi'(y)^{-1} dy \cdot q^{-(n+1)(s+\frac{1}{2})}. $$
	The assertion on the dimension $d$ of $\pi'$ then follows from
	$$ \norm[\B/\gp{Z}U_{\B}^n] = 2q^{2n}(1-q^{-2}) \geq d^2, $$
since the sum of the squares of the dimensions of all irreducible representations of a finite group is equal to the cardinality of the group.

\noindent The final assertion on the conductor of $\Omega$ follows from $\vO \subset \VO$ and
	$$ U_{\B}^n \supset \left\{ \begin{matrix} 1+\varpi_{\E}^n \vO & \text{if } \E/\F \text{ is ramified}, \\ 1+\varpi_{\E}^{\IntP{n/2}} \vO & \text{if } \E/\F \text{ is unramified}. \end{matrix} \right. $$
\end{proof}
\begin{lemma}
	For any embedding $\iota$,
	$$ \alpha(\pi'; \Omega, \iota) = \Vol(\gp{Z} \backslash \gp{T}). $$
\label{LocalEst0}
\end{lemma}
\begin{proof}
	This follows for example by (\ref{NSPFRAverage}).
\end{proof}

	\subsection{$\B$-Split, Archimedean Place}
	
\begin{lemma}
\noindent (1) If $\E$ is split, then we can choose a unitary $v_0 \in \pi^{\infty}$ corresponding to a bump function on $\F^{\times}$ in the Kirillov model, such that for any $\epsilon > 0$, there is $r \in \F$ with $\Cond(\Omega)^{1-\epsilon} \ll |r|_v \leq \Cond(\Omega)^{1+\epsilon}$, $\iota_r$ defined in (\ref{ArchCarEmbS}) and
	$$ \alpha(v_0; \Omega, \iota_r) \gg_{\epsilon, v_0} \Cond(\Omega)^{-1/2-\epsilon}. $$

\noindent (2) If $\F=\R, \E=\C$, then there is $r \in \ag{R}$ with $\norm[r] \asymp \Cond(\Omega)^{1/2}$ such that, with $\iota_r$ defined in (\ref{ArchCarEmbNS}), $v_0$ a unitary lowest weight vector in $\pi$ and $\theta$ any constant towards the Ramanujan-Petersson conjecture
	$$ \alpha(v_0; \Omega, \iota_r) \gg \Cond(\Omega)^{-(1+2\theta)/2}. $$
%	Moreover, if $\pi$ is a discrete series representation then the right hand side of the above inequality can be replaced by $1$.
\label{LocalEst1}
\end{lemma}
\begin{proof}
	The first case is treated in \cite[\S 4.1]{Wu14}. Here it is just a re-formulation. To prove the assertion for the second case, we first notice that we only need to treat one $\pi$ from each class of unitary twists $[\pi] := \{ \pi \otimes \chi: \chi \text{ unitary characters of } \R^{\times} \}$. Precisely, we are reduced to three cases:
\begin{itemize}	
	\item[(\rmnum{1})] $\pi = \pi(\norm^s, \norm^{-s}), s \in i\R \cup (-\theta,\theta)$. Hence $k_0=0$, $2 \mid k$ and $\lambda = s^2-1/4$.
	\item[(\rmnum{2})] $\pi = \pi(\norm^s \sgn, \norm^{-s}), s \in i\R \cup (-\theta,\theta)$. Hence $k_0=1$, $2 \nmid k$ and $\lambda = s^2-1/4$.
	\item[(\rmnum{3})] $\pi = \sigma(\norm^{p/2}, \norm^{-p/2})$ resp. $\sigma(\norm^{p/2}\sgn, \norm^{-p/2})$ for some $p \in \Z_{>0}$ with $2 \nmid p$ resp. $2 \mid p$. Hence $k_0=p+1$, $2 \mid k-k_0$ and $\lambda = k_0(k_0-2)/4$.
\end{itemize}
	Let $e_k \in \pi$ be a unitary weight $k$ vector in $\pi$ which corresponds to $\hat{v}_0$ in Definition \ref{conventionHat_r} such that $\Cond(\Omega) \asymp \norm[k]^2$. We shall study an \emph{effective} asymptotic behavior as $\norm[y] \to \infty$ of
	$$ f(y) := \Pairing{\pi(a(y)).v_0}{\hat{v}_0}. $$
	Because by (\ref{NSPFRAverage}), $\alpha(v_0; \Omega, \iota_r) = \Vol(\SO_2(\R)/\{ \pm 1 \}) \norm[f(r)]^2$. If $v_0$ is of weight $k_0$, then the Casimir operator acts as multiplication by $\lambda=\lambda_{\pi}$ with (c.f. \cite[Example 3.7]{CM82} \footnote{Unfortunately, none of the formulae in \cite[Example 2.7 \& Example 3.7]{CM82} has the right signs!})
	$$ \Delta = H^2 - H - Y \cdot \theta Y, \quad H = \frac{1}{2} \begin{bmatrix} 1 & \\ & -1 \end{bmatrix}, Y = \begin{bmatrix} 0 & 1 \\ 0 & 0 \end{bmatrix}, \theta Y = \begin{bmatrix} 0 & 0 \\ -1 & 0 \end{bmatrix}; $$
	$$ \Delta = \delta^2 - \frac{1+y^2}{1-y^2} \delta - (k^2 + k_0^2) \frac{y^2}{(1-y^2)^2} + kk_0 \frac{y(1+y^2)}{(1-y^2)^2}, \quad \delta := y \frac{d}{dy}. $$
	For $y \gg 1$, making the substitution
	$$ f(y) = \sqrt{\frac{y}{y^2-1}} g(y), $$
	we see that $g(y)$ satisfies the differential equation
	$$ 0 = (\delta^2 - q(y)-\lambda)g(y), \quad q(y) = \frac{1}{4} + (k^2+k_0^2-1) \frac{y^2}{(1-y^2)^2} - kk_0y \frac{1+y^2}{(1-y^2)^2}. $$
	This equation is regular at $\infty$. Hence we can apply the method in \cite[\S 4.3]{Co61}. We expand $q(y)$ at $\infty$ as
	$$ q(y) = \frac{1}{4} + (k^2+k_0^2-1) \sideset{}{_{l = 1}^{\infty}} \sum \frac{l}{y^{2l}} - kk_0 \sideset{}{_{l = 1}^{\infty}} \sum \frac{2l-1}{y^{2l-1}}. $$
	
\noindent \textbf{Case (\rmnum{1}):} Expanding the relevant function at $\infty$ as
	$$ g(y) = y^{\pm s} \sideset{}{_{l = 0}^{\infty}} \sum c_l y^{-l} $$
	we deduce that for $l \in \Z_{\geq 0}$
	$$ c_{2l+1} = 0, \quad ((\pm s + 2l)^2 -s^2) c_{2l} = (k^2-1) \sideset{}{_{j=1}^l} \sum j c_{2(l-j)} \quad \Rightarrow \quad c_{2l} = (k^2-1)^l O_{\pi}(1) c_0. $$
	It follows that $f(y)$ has an expansion at $\infty$ as a linear combination of
	$$ y^{\pm s} \sqrt{\frac{y}{y^2-1}} \cdot \left( 1 + O_{\pi}(1) \sideset{}{_{l = 1}^{\infty}} \sum \frac{(k^2-1)^l}{y^{2l}} \right), $$
	the absolute value of which is $\gg \norm[k]^{-1/2-\theta}$ for $y = C \norm[k]$ with sufficiently large constant $C > 0$.
	
\noindent \textbf{Case (\rmnum{2}):} The proof is quite similar to the case (\rmnum{1}). We omit it.

\noindent \textbf{Case (\rmnum{3}):} First notice that by \cite[\S 2.5 (80), (81) \& (82)]{Go70} we only need to treat the case $k > 0$. Therefore $f(y) \neq 0$ only if $y > 0$. Lemma \ref{DisMC} gives a precise formula for $f$ with the change of parameters $k' := (k-p-1)/2$
	$$ f(e^{2x}) = \sqrt{\binom{p+k'}{p}} \tanh(x)^{k'} \cosh(x)^{-(p+1)}, $$
	which reaches its maximal value of modulus at
	$$ \cosh(x)^{-2} = \frac{p+1}{p+1+k'} \quad \Leftrightarrow \quad e^x = \sqrt{1+\frac{k'}{p+1}} + \sqrt{\frac{k'}{p+1}} \quad \Rightarrow \quad y=e^{2x} \asymp \frac{4k'}{p+1}. $$
	And the maximal value is
	$$ \sqrt{\frac{(p+k')!}{p! k'!}} \left( \frac{k'}{p+1+k'} \right)^{k'/2} \left( \frac{p+1}{p+1+k'} \right)^{(p+1)/2} \asymp_p k^{-1/2}, \quad k' \to \infty $$
	by Stirling's formula. We conclude the proof in this case.
\end{proof}

\begin{remark}
	We believe that with a finer analysis in the cases (\rmnum{1}) and (\rmnum{2}) given in the proof above, $-(1+2\theta)/2$ can be replaced by $-1/2$ in the estimation in (2). But there are some difficulties. On the one hand, although the asymptotic behavior of matrix coefficients has been a lot studied (c.f. \cite{CHH88} and its references), only the upper bounds have been well understood. On the other hand, although the relevant matrix coefficients can be expressed in terms of the classical Legendre functions \cite[\S 4.6.1]{MOS66} in the case (\rmnum{1}), the relevant estimations seem to be ``difficult to find'' (c.f. discussion after \cite[Proposition 5.4]{CU05}).
\label{LackLegEst}
\end{remark}

\begin{lemma}
	Let $\pi=\sigma_p, p \in \Z_{>0}$ be the discrete series representation defined in the case (\rmnum{3}) of the proof of Lemma \ref{LocalEst1} (2). For any $k \in \Z_{\geq 0}$, let $e_{p+1+2k}$ be a unitary vector of weight $p+1+2k$. Then
	$$ f_k(x) := \Pairing{\pi(a(e^{2x})).e_{p+1}}{e_{p+1+2k}}, \quad x \in \R $$
is equal to
	$$ f_k(x) = \sqrt{\binom{p+k}{p}} \tanh(x)^k \cosh(x)^{-(p+1)}. $$
\label{DisMC}
\end{lemma}
\begin{proof}
	The discussion in \cite[\S 2.7.1]{Wu14} implies that there exist $\tilde{e}_{p+1+2k}$ proportional to $e_{p+1+2k}$ satisfying
	$$ V_-.\tilde{e}_{p+1} = 0, \quad V_+.\tilde{e}_{p+1+2k} = \tilde{e}_{p+1+2(k+1)}; \quad V_{\pm} = \begin{bmatrix} 1 & \pm i \\ \pm i & -1 \end{bmatrix}. $$
	Expressing the relevant differential operators in the $(p+1,p+1+2k)$-spherical coordinates, we get
	$$ \left( - \frac{d}{dx} - (p+1) \tanh(x) \right) \tilde{f}_{0}(x) = 0, $$
	$$ \tilde{f}_{k+1}(x) = \left( - \frac{d}{dx} + (p+1+2k) \tanh(x) + \frac{2k}{\sinh(2x)} \right) \tilde{f}_k(x), $$
	for $\tilde{f}_k(x) := \Pairing{\pi(a(e^{2x})).\tilde{e}_{p+1}}{\tilde{e}_{p+1+2k}}$. Up to scalar, a solution is given by
	$$ \tilde{f}_k(x) \sim \tanh(x)^k \cosh(x)^{-(p+1)}. $$
	A Haar measure of $\PGL_2(\R)$ in the spherical coordinates can be given by $\sinh(2x) dx$ (c.f. \cite[(7.22)]{KL72}). We thus obtain
	$$ \int_0^{\infty} \left( \tanh(x)^k \cosh(x)^{-(p+1)} \right)^2 \sinh(2x) dx = \int_0^1 t^{p-1}(1-t)^k dt = \frac{\Gamma(p)\Gamma(k+1)}{\Gamma(p+k+1)} $$
	with the change of variables $t := \cosh(x)^{-2}$ and conclude by the theory of formal degree \cite[(9.4)]{Kn86}.
\end{proof}

\begin{lemma}
	Let $\pi$ be a unitary spherical representation with parameter $\leq \theta$. Let $e_0$ be a unitary spherical vector. Let $\iota_r$ be defined in (\ref{ArchCarEmbNS}) with $r \geq 1$. Then we have for any $\epsilon > 0$
	$$ \alpha(e_0; 1, \iota_r) \ll_{\epsilon} r^{-1+2\theta+\epsilon}. $$
\label{LocalEstSphUpBis}
\end{lemma}
\begin{proof}
	By the decay of matrix coefficients \cite[Theorem 2.31]{Wu14}, we have for any $\epsilon > 0$
	$$ \alpha(e_0; 1, \iota_r) = \extnorm{ \Pairing{a(r).e_0}{e_0} }^2 \ll_{\theta, \epsilon} \left( \Xi_{\R}(a(r))^{1-2\theta + \epsilon} \right)^2, $$
where $\Xi_{\R}$ is the Harisch-Chandra function at a real place. The desired bound follows from the bound of $\Xi_{\R}$ given in \cite[\S 5.2.2]{CU05}.
\end{proof}

	\subsection{$\B$-Split, Finite Place, $\E$-Split}
	
\begin{lemma}
	Let $v_0$ be a new vector of $\pi$, $\sigma_0$ be defined in Definition \ref{InvSpDef}. Write $r = \sideset{}{_{\vP \mid \vp}} \min \cond(\Omega_{\vP})$. We have with absolute implicit constant
	$$ \tilde{\alpha}(\sigma_0(\pi); \Omega, \iota_r) \geq \tilde{\alpha}(v_0; \Omega, \iota_r) \gg \Cond(\Omega)^{-1/2}. $$
If $\pi$ is spherical and $\Omega$ is unramified, i.e. $\cond(\Omega_w) = 0, \forall w \mid v$ or $r = 0$, then we have
	$$ \tilde{\alpha}(\sigma_0(\pi); \Omega, \iota_0) = \tilde{\alpha}(v_0; \Omega, \iota_0) = 1. $$
\label{LocalEst2}
\end{lemma}
\begin{proof}
	(\ref{SPFRIndividual}) implies the positivity of the local factors $\alpha(\cdot)$, hence the first inequality. The rest of the first part is a re-formulation of \cite[Proposition 3.1]{FMP} or  \cite[Lemma 11.7]{Ve10} or \cite[Corollary 4.8]{Wu14}. For the second part, by the definition of new vectors, we have
	$$ \alpha(v_0; \Omega, \iota_0) = \lvert L(1/2, \pi \otimes \chi) \rvert^2 \lVert v_0 \rVert^{-2} $$
where $\chi$ is such that $\Omega$ corresponds to $(\chi, \omega^{-1}\chi^{-1})$ via the/an identification $\E \simeq \F \times \F$ and $\lVert v_0 \rVert$ is calculated in the Whittaker model by
	$$ \lVert v_0 \rVert^2 = \zeta(2)^{-1} L(1, \pi \times \tilde{\pi}) = \zeta(2)^{-1} \zeta(1) L(1,\pi,{\rm Ad}). $$
We conclude by noting that $L(1,\eta)=\zeta(1)$ (because $\E$ is split), and
	$$ L(1/2, \pi_{\E} \otimes \Omega) = L(1/2,\pi \otimes \chi) L(1/2, \pi \otimes \omega^{-1}\chi^{-1}), \quad L(1/2, \pi \otimes \omega^{-1}\chi^{-1}) = \overline{L(1/2, \pi \otimes \chi)}. $$
\end{proof}

	\subsection{$\B$-Split, Finite Place, $\E$-nonsplit, $\pi$ spherical}
	
	There exist unramified quasi-characters $\mu_1,\mu_2$ of $\F^{\times}$ such that $\pi = \pi(\mu_1,\mu_2) = \Ind_{\gp{B}}^{\GL_2} (\mu_1,\mu_2)$. Write
	$$ \alpha_1=\mu_1(\varpi), \alpha_2=\mu_2(\varpi). $$
Let $v_0 \in \pi$ be a spherical vector. The function 
	$$ g \mapsto \frac{\langle \pi(g).v_0, v_0 \rangle}{\langle v_0, v_0 \rangle} $$
is $\gp{K}$-bi-invariant. Its value on $\gp{K}a(\varpi^m)\gp{K}, m \in \N$ is given by Macdonald formula (\cite[Theorem 4.6.1]{Bu98})
	$$ \sigma(\pi,m) = \frac{q^{-m/2}}{1+q^{-1}} \left( \alpha_1^m \frac{1-q^{-1}\alpha_2\alpha_1^{-1}}{1-\alpha_2\alpha_1^{-1}} + \alpha_2^m \frac{1-q^{-1}\alpha_1\alpha_2^{-1}}{1-\alpha_1\alpha_2^{-1}} \right). $$
Since we have
	$$ \F^{\times} \backslash \B^{\times} = \vo^{\times} \backslash \bigsqcup_{m \geq 0} \gp{K} a(\varpi^m) \gp{K}, $$
we get for any embedding $\iota: \E \to \B$
	$$ \alpha(v_0; \Omega, \iota) = \int_{\F^{\times} \backslash \E^{\times}} \frac{\langle \pi(\iota(t)).v_0, v_0 \rangle}{\langle v_0, v_0 \rangle} \Omega(t) dt = \sum_{m \geq 0} \sigma(\pi,m) \int_{\vo^{\times} \backslash \iota^{-1}(Ka(\varpi^m)K)} \Omega(t) dt. $$
If the \emph{(logarithmic) conductor} of $\iota$ is $r$ (see Proposition \ref{LogCondEmbd}), then it is easy to see
	$$ \iota^{-1}(\gp{K}a(\varpi^m)\gp{K}) = \vO_r^{(m)} $$
where the right hand side is defined in (\ref{ArithQuadSet}). We have obtained
\begin{equation}
	\alpha(v_0; \Omega, \iota) = \sum_{m \geq 0} \sigma(\pi,m) \int_{\vo^{\times} \backslash \vO_r^{(m)}} \Omega(t) dt.
\label{alphatoArith}
\end{equation}
	Let's write $e(\E/\F) \in \{ 1,2 \}$ for the ramification index of $\E/\F$ and define
	$$ r_0 = \IntCP{\cond(\Omega)/e(\E/\F)}. $$
	By Lemma \ref{LIDecomp} and Corollary \ref{LIQuotient}, for $r\geq 1$ (\ref{alphatoArith}) becomes, if $\E/\F$ is unramified
\begin{align*}
	\frac{\alpha(v_0; \Omega, \iota)}{\Vol(\gp{Z} \backslash \gp{T})} &= \sigma(\pi, 0) \cdot \frac{1_{r \geq r_0}}{q^r(1+q^{-1})} + \frac{\sigma(\pi, 2r)}{(\alpha_1 \alpha_2)^r} \cdot \left( 1_{0 \geq r_0} - \frac{1_{1 \geq r_0}}{q(1+q^{-1})} \right) \\
	&+ \sum_{m=1}^{r-1} \frac{\sigma(\pi,2m)}{(\alpha_1 \alpha_2)^m} \cdot \left( \frac{1_{r-m \geq r_0}}{q^{r-m}(1+q^{-1})} - \frac{1_{r-m+1 \geq r_0}}{q^{r-m+1}(1+q^{-1})} \right);
\end{align*}
while if $\E/\F$ is ramified
\begin{align*}
	\frac{\alpha(v_0; \Omega, \iota)}{\Vol(\gp{Z} \backslash \gp{T})} &= \sigma(\pi, 0) \cdot \frac{1_{r \geq r_0}}{2q^r} + \frac{\sigma(\pi, 2r+1)}{\Omega(\varpi_{\E})(\alpha_1 \alpha_2)^r} \cdot \frac{1_{0 \geq r_0}}{2} \\
	&+ \sum_{m=1}^r \frac{\sigma(\pi,2m)}{(\alpha_1 \alpha_2)^m} \cdot \left( \frac{1_{r-m \geq r_0}}{2q^{r-m}} - \frac{1_{r-m+1 \geq r_0}}{2q^{r-m+1}} \right).
\end{align*}
Thus if $r=r_0>0$, we get
	$$ \frac{\alpha(v_0; \Omega, \iota)}{\Vol(\gp{Z} \backslash \gp{T})} = \frac{1-\sigma(\pi,2)(\alpha_1 \alpha_2)^{-1}}{q^{r_0}} \cdot \left\{ \begin{matrix} (1+q^{-1})^{-1} & \text{if }\E/\F \text{ unramified}; \\ 1/2 & \text{if }\E/\F \text{ ramified}. \end{matrix} \right. $$
\begin{lemma}
	Let $r_0 = \IntCP{\cond(\Omega)/e(\E/\F)}$ and $v_0 \in \pi$ be a spherical vector. If $\cond(\Omega) > 0$ and the conductor of $\iota$ is $r_0$, then we have
	$$ \frac{\alpha(v_0; \Omega, \iota)}{\Vol(\gp{Z} \backslash \gp{T})} = q^{-r_0} \frac{\zeta(2)L(1,\eta)}{e(\E/\F)L(1,\pi, {\rm Ad})}, $$
where $e(\E/\F)$ is the ramification index of $\E/\F$ and $\eta$ is the quadratic character associated with $\E/\F$.
\label{LocalEst4}
\end{lemma}
\begin{proof}
	This is simply a way to re-write the equation before the lemma.
\end{proof}
\noindent If $r_0 = 0$, i.e., $\cond(\Omega)=0$, and $\E/\F$ is unramified, we get for $r\geq 1$
	$$ \frac{\alpha(v_0; \Omega, \iota)}{\Vol(\gp{Z} \backslash \gp{T})} = \frac{1}{q^r(1+q^{-1})} \left( 1+(1-q^{-1}) \sum_{m=1}^r \frac{q^m\sigma(\pi,2m)}{(\alpha_1 \alpha_2)^m} \right); $$
while if $\E/\F$ is ramified, we get for $r \geq 1$
	$$ \frac{\alpha(v_0; \Omega, \iota)}{\Vol(\gp{Z} \backslash \gp{T})} = \frac{1}{2q^r} \left( 1+\frac{q^r\sigma(\pi,2r+1)}{\Omega(\varpi_{\E})^{2r+1}} + (1-q^{-1}) \sum_{m=1}^r \frac{q^m\sigma(\pi,2m)}{(\alpha_1 \alpha_2)^m} \right). $$
If $\cond(\Omega)=r=0$ and $\E/\F$ is unramified, we get
	$$ \frac{\alpha(v_0; \Omega, \iota)}{\Vol(\gp{Z} \backslash \gp{T})} = 1 = \frac{\zeta(2)L(1/2,\pi_{\E}\otimes \Omega)}{L(1,\eta)L(1,\pi,{\rm Ad})} $$
since in this case $L(1,\eta)=\zeta(2)\zeta(1)^{-1}$ and
	$$ L(1/2,\pi_{\E}\otimes \Omega)=(1-\alpha_1 \alpha_2^{-1}q_{\E}^{-1})^{-1/2}(1-\alpha_2 \alpha_1^{-1}q_{\E}^{-1/2})^{-1} =\zeta(1)^{-1}L(1,\pi,{\rm Ad}); $$
while if $\E/\F$ is ramified, we get
\begin{align*}
	\frac{\alpha(v_0; \Omega, \iota)}{\Vol(\gp{Z} \backslash \gp{T})} &= 1+\sigma(\pi,1)\Omega(\varpi_{\E}) = \frac{(1+\alpha_1\Omega(\varpi_{\E})q^{-1/2})(1+\alpha_2\Omega(\varpi_{\E})q^{-1/2})}{1+q^{-1}} \\
	&= \frac{(1-\alpha_1\alpha_2^{-1}q^{-1})(1-\alpha_2\alpha_1^{-1}q^{-1})(1-q^{-1})}{(1-q^{-2})(1-\alpha_1 \Omega(\varpi_{\E})q^{-1/2})(1-\alpha_2 \Omega(\varpi_{\E})q^{-1/2})} = \frac{\zeta(2)L(1/2,\pi_{\E}\otimes \Omega)}{L(1,\eta)L(1,\pi,{\rm Ad})}
\end{align*}
since $\Omega(\varpi_{\E})^2=(\alpha_1\alpha_2)^{-1}$ and $L(1,\eta)=1$ in this case.
\begin{lemma}
	Let $v_0 \in \pi$ be a spherical vector. If $\cond(\Omega)=0$ and the conductor of $\iota$ is $r \geq 1$, then we have the estimation with absolute implicit constant
	$$ \frac{\alpha(v_0; \Omega, \iota)}{\Vol(\gp{Z} \backslash \gp{T})} \ll r q^{-(1-2\theta)r}, $$
where $\theta$ is a constant towards the Ramanujan conjecture. If $\cond(\Omega)$ and the conductor of $\iota$ are both $0$, then we have
	$$ \frac{\alpha(v_0; \Omega, \iota)}{\Vol(\gp{Z} \backslash \gp{T})} = \frac{\zeta(2)L(1/2,\pi_{\E}\otimes \Omega)}{L(1,\eta)L(1,\pi,{\rm Ad})}. $$
\label{LocalEstSphUp}
\end{lemma}
\begin{remark}
	The second part of the previous lemma is a slight generalization of \cite[Proposition 5.9 (a) \& Proposition 5.10]{CU05}.
\end{remark}

	\subsection{$\B$-Split, Finite Place, $\E$-nonsplit, $\pi$ ramified non supercuspidal}
	
	There exist quasi-characters $\mu_1,\mu_2$ of $\F^{\times}$ such that $\pi = \pi(\mu_1,\mu_2)$ is (maybe a subquotient of) $\Ind_{\gp{B}}^{\GL_2} (\mu_1,\mu_2)$. The Waldspurger vector $\hat{v}_r'$ is realized in the induced model by the function $\hat{f}_r'$ (Definition \ref{conventionHat_r}), determined by (c.f. \cite[(4.4)]{FMP})
	$$ \hat{f}_r'\left( \begin{pmatrix} a & * \\ 0 & d \end{pmatrix} \iota_r'(t) \right) = \left\lvert \frac{a}{d} \right\rvert^{1/2} \mu_1(a) \mu_2(d) \Omega^{-1}(t), \quad \forall a,d\in \F^{\times}, t \in \E^{\times}. $$
\begin{lemma}
	Let $c \geq \max (\cond(\mu_1), \cond(\mu_2), 1)$ and $r_0 = \IntCP{\cond(\Omega)/e(\E/\F)}$ as before. Recall the notations in Definition \ref{InvSpDef}.
\begin{itemize}
	\item[(1)] Assume
\begin{itemize}
	\item[-] $\cond(\Omega) \leq c$ if $\E/\F$ is unramified;
	\item[-] $\cond(\Omega) \leq 2c-1$ if $\E/\F$ is ramified.
\end{itemize}
Then we have $\hat{v}_0' = \hat{v}_{0,\Omega}' \in [\pi; c]$. Hence $\Proj_c(\hat{v}_0') = \hat{v}_0'$. Consequently, by (\ref{NSPFRAverage})
	$$ \frac{\alpha([\pi;c]; \Omega, \iota_0')}{\Vol(\gp{Z} \backslash \gp{T})} = 1. $$	
	\item[(2)] Assume
\begin{itemize}
	\item[-] $\cond(\Omega) > c$ if $\E/\F$ is unramified;
	\item[-] $\cond(\Omega) > 2c-1$ if $\E/\F$ is ramified.
\end{itemize}
Take $r=r_0-c$. Then we have
	$$ \lVert \Proj_c(\hat{v}_r') \rVert^2 = q^{-r}\zeta_{\E,\vp}(1) \lVert \hat{v}_r' \rVert^2. $$
Consequently, by (\ref{NSPFRAverage})
	$$ \frac{\alpha([\pi;c]; \Omega, \iota_r')}{\Vol(\gp{Z} \backslash \gp{T})} = q^{-r}\zeta_{\E,\vp}(1) (\asymp_c \Cond(\Omega)^{-1/2}). $$
\end{itemize}
\label{LocalEst5}
\end{lemma}
\begin{proof}[Proof of Lemma \ref{LocalEst5} (1)]
	For any $\kappa \in \gp{K}(\vp^c), b \in \gp{B}, t \in \E^{\times}$, we have
	$$ \pi(\kappa).\hat{f}_0(b\iota_0'(t)) = \hat{f}_0'(b \iota_0'(t) \kappa \iota_0'(t)^{-1} \iota_0'(t)) = \hat{f}_0'(b\iota_0'(t)) \cdot \hat{f}_0'(\iota_0'(t) \kappa \iota_0'(t)^{-1}). $$
Hence it suffices to prove
	$$ \hat{f}_0'(\iota_0'(t) \kappa \iota_0'(t)^{-1}) =1, \quad \forall \kappa \in \gp{K}(\vp^c), t \in \E^{\times}. $$
Note that
	$$ \iota_0'(\E^{\times}) = \left\{ \begin{matrix} \gp{Z} \iota_0'(\vO^{\times}) & \text{if } \E/\F \text{ is unramified}, \\ \gp{Z} \iota_0'(\vO^{\times}) \bigsqcup \iota_0'(\varpi_{\E}) \gp{Z} \iota_0'(\vO^{\times}) & \text{if } \E/\F \text{ is ramified}, \end{matrix}  \right. $$
and $\iota_0'(\vO^{\times}) \subseteq \gp{K}$ which stabilizes $\gp{K}(\vp^c)$. We only need to show that
\begin{equation}
	\hat{f}_0'(\kappa) =1, \quad \forall \kappa \in \gp{K}(\vp^c),
\label{Inv1}
\end{equation}
and if $\E/\F$ is ramified
\begin{equation}
	\hat{f}_0'(\iota_0'(\varpi_{\E}) \kappa \iota_0'(\varpi_{\E})^{-1}) =1, \quad \forall \kappa \in \gp{K}(\vp^c).
\label{Inv2}
\end{equation}
	Recall the notations in \S \ref{ClassEmb}. Note that we have the decomposition
	$$ \kappa = \begin{pmatrix} x_1 & x_2 \\ x_3 & x_4 \end{pmatrix} = \begin{pmatrix} \det \kappa \cdot \Nr(x_4-\fa'^{-1}x_3\beta')^{-1} & * \\ 0 & 1 \end{pmatrix} \iota_0'(x_4-\fa'^{-1}x_3\beta'). $$
If $\kappa \in \gp{K}(\vp^c)$, then $\det \kappa \in 1+\vp^c, x_3 \in \vp^c, x_4 \in 1+\vp^c$, hence $x_4-\fa'^{-1}x_3\beta' \in 1+\varpi^c \vO, \Nr(x_4-\fa'^{-1}x_3\beta') \in 1+\vp^c$. We get (\ref{Inv1}) by noting $c\geq \cond(\mu_1), e(\E/\F)c \geq\cond(\Omega)$ hence
	$$ \hat{f}_0'(\kappa) \in \mu_1(1+\vp^c) \Omega^{-1}(1+\varpi^c \vO) = \{1\}. $$
If $\E/\F$ is ramified, we can take $\varpi_{\E} = \beta' - u_0$ with $u_0$ given by (\ref{u0}). Hence
\begin{align*}
	\iota_0'(\varpi_{\E}) \gp{K}(\vp^c) \iota_0'(\varpi_{\E})^{-1} &\subseteq \begin{pmatrix} \fb'-u_0-\fa' u_0^{-1} & 1 \\ 0 & -u_0 \end{pmatrix} \begin{pmatrix} 1 & 0 \\ \fa' u_0^{-1} & 1 \end{pmatrix} \gp{K}(\vp^c) \iota_0'(\varpi_{\E})^{-1} \\
	&= \begin{pmatrix} \fb'-u_0-\fa' u_0^{-1} & 1 \\ 0 & -u_0 \end{pmatrix} \begin{pmatrix} 1+\vp^c & \vp^c \\ 0 & 1+\vp^c \end{pmatrix} \begin{pmatrix} 1 & \\ \vp^c & 1 \end{pmatrix} \begin{pmatrix} 1 & 0 \\ \fa' u_0^{-1} & 1 \end{pmatrix} \iota_0'(\varpi_{\E})^{-1} \\
	&\subseteq \begin{pmatrix} (\fb'-u_0-\fa' u_0^{-1})(1+\vp^c) & 1+\vp^c \\ 0 & -u_0(1+\vp^c) \end{pmatrix} \begin{pmatrix} 1 & 0 \\ \fa' u_0^{-1}+\vp^c & 1 \end{pmatrix} \iota_0'(\varpi_{\E})^{-1}.
\end{align*}
	Note that
	$$ \begin{pmatrix} 1 & 0 \\ \fa' u_0^{-1}+\vp^c & 1 \end{pmatrix} \subseteq \begin{pmatrix} \Nr(1-u_0^{-1}\beta' + \vp^c \beta')^{-1} & \F \\ 0 & 1 \end{pmatrix} \iota_0'(1-u_0^{-1}\beta'+\vp^c \beta'). $$
	But $1-u_0^{-1}\beta' + \vp^c \beta' \subseteq (1-u_0^{-1}\beta')(1 + \varpi_{\E}^{2c-1} \vO)$, and $\Omega$ is trivial on $1 + \varpi_{\E}^{2c-1} \vO$ by assumption. Thus
	$$ \hat{f}_0'(\iota_0'(\varpi_{\E}) \gp{K}(\vp^c) \iota_0'(\varpi_{\E})^{-1}) = \left\lvert \frac{\fb'-u_0-\fa' u_0^{-1}}{\Nr(1-u_0^{-1}\beta')} \right\rvert^{1/2} \mu_1\left(\frac{\fb'-u_0-\fa' u_0^{-1}}{\Nr(1-u_0^{-1}\beta')}\right)\mu_2(-u_0) \Omega\left( \frac{-u_0+\beta'}{1-u_0^{-1}\beta'} \right), $$
which is $1$ since $\Omega(-u_0)=\omega(-u_0)^{-1}=\mu_1^{-1}\mu_2^{-1}(-u_0)$, proving (\ref{Inv2}).

%If $\E/\F$ is ramified, we apply the same decomposition to
%\begin{align*}
%	\begin{pmatrix} x_1' & x_2' \\ x_3' & x_4' \end{pmatrix} &:= \iota_0(\beta) \begin{pmatrix} x_1 & x_2 \\ x_3 & x_4 \end{pmatrix} \iota_0(\beta)^{-1} = \begin{pmatrix} x_4 - \fb x_3 & - \fa x_3 \\ \fa^{-1}(\fb (x_1-x_4) + \fb^2 x_3 - x_2) & x_1 + \fb x_3 \end{pmatrix} \\
%	&= \begin{pmatrix} \det \kappa & * \\ 0 & \Nr(x_4'-x_3' \beta) \end{pmatrix} \iota_0(x_4' - x_3' \beta)^{-1}.
%\end{align*}
%	If $\kappa \in \gp{K}(\vp^c)$, then $\det \kappa \in 1+\vp^c, x_3' \in \vp^{c-1}, x_4' \in 1+\vp^c$, hence $x_4' - x_3' \beta \in 1+\varpi_{\E}^{2c-1}\vO \Rightarrow \Nr(x_4'-x_3' \beta) \in (1+\vP^{2c-1}) \cap \vo = 1+\vp^c$. We conclude (\ref{Inv2}) the same way.

\end{proof}

	The proof of Lemma \ref{LocalEst5} (2) is based on a property of the branching law for $[\pi; c]$ restricted to $\iota_r'(\vO_r^{\times})$. A double coset decomposition with $\kappa_1=1$
	$$ \gp{K} = \sideset{}{_i} \bigsqcup \gp{K}_0(\vp^c) \kappa_i \iota_r'(\vO_r^{\times}) $$
yields an orthogonal decomposition
	$$ \Ind_{\gp{K}_0(\vp^c)}^{\gp{K}} (\mu_1, \mu_2) = \sideset{}{_i} \bigoplus \mathcal{S}_i(c), $$
where $\mathcal{S}_i(c) = \mathcal{S}_{\kappa_i}(c)$ is the subspace of functions supported in $\gp{K}_0(\vp^c) \kappa_i \iota_r'(\vO_r^{\times})$.
\begin{remark}
	Note that $\gp{K}_0(\vp^c) \iota_r'(\vO_r^{\times})$ is ``big'' by (\ref{SppS1}) and (\ref{DecompK}), hence intuitively should ``hold'' more irreducible representations of $\E^{\times}$ than the other cosets.
\end{remark}
\begin{lemma}
\noindent (1) The restriction of $\Ind_{\gp{K}_0(\vp^c)}^{\gp{K}} (\mu_1, \mu_2)$ to $\iota_r'(\vO_r^{\times})$ only contains characters $\Omega'$ of $\vO_r^{\times}$ with $r_0' \leq r+c$ where $r_0' := \IntCP{\cond(\Omega') / e(\E/\F)}$.

\noindent (2) If $r_0'=r+c$, $\Omega'$-isotypic space lies in $\mathcal{S}_1(c)$ with support $\gp{K}_0(\vp^c) \iota_r'(\vO_r^{\times})$ (including the case $r=0$) and is of dimension $1$ except in the case $\E/\F$ ramified, $r=0, c=r_0=r_0'$ and $\cond(\Omega')=2c-1$. 

%\noindent (3) In the exceptional case of (2), an extra $\Omega'$-isotypic space lies in $\mathcal{S}_{n_-(u_0)}$ with $u_0$ given by (\ref{u0}).
\label{BLNS}
\end{lemma}
\noindent Temporarily admitting Lemma \ref{BLNS}, we can give:
\begin{proof}[Proof of Lemma \ref{LocalEst5} (2)]
	For any $t_0 \in \vO_r^{\times}$, we have $\iota_r'(t_0) \in \gp{K}$. Hence $\Proj_c(\hat{v}_r')$, calculated by
	$$ \Proj_c(\hat{v}_r') = \int_{\gp{K}(\vp^c)} \pi(\kappa).\hat{v}_r' \tilde{d}\kappa $$
with normalized Haar measure ${\rm Vol}(\gp{K}(\vp^c), \tilde{d}\kappa)=1$, satisfies
\begin{align*}
	\pi(\iota_r'(t_0)).\Proj_c(\hat{v}_r') &= \int_{\gp{K}(\vp^c)} \pi(\iota_r'(t_0) \kappa \iota_r'(t_0)^{-1}).\pi(\iota_r'(t_0)).\hat{v}_r' \tilde{d}\kappa \\
	&= \Omega^{-1}(t_0) \int_{\gp{K}(\vp^c)} \pi(\kappa).\hat{v}_r' \tilde{d}\kappa = \Omega^{-1}(t_0) \Proj_c(\hat{v}_r'),
\end{align*}
Hence $\Proj_c(\hat{v}_r')$ is a vector in $[\pi;c]$ which lies in the $\Omega^{-1}$-isotypic subspace under the action of $\iota_r'(\vO_r^{\times})$. But in our case we have an isomorphism of $\gp{K}$-representations
	$$ [\pi; c] \simeq \Ind_{\gp{K}_0(\vp^c)}^{\gp{K}} (\mu_1,\mu_2). $$
Applying Lemma \ref{BLNS}, we see that $\Proj_c(\hat{f}_r')$ lies in $\mathcal{S}_1(c)$ with
	$$ \Proj_c(\hat{f}_r')(\kappa \iota_r'(t_0)) = (\mu_1,\mu_2)(\kappa) \Omega^{-1}(t_0) \Proj_c(\hat{f}_r')(1), \quad \forall \kappa \in \gp{K}_0(\vp^c), t_0 \in \vO_r^{\times}. $$
Note for any $\kappa = \begin{pmatrix} x_1 & x_2 \\ x_3  & x_4 \end{pmatrix} \in \gp{K}(\vp^c)$, we have
	$$  \kappa = \begin{pmatrix} x_1 & x_2 \\ x_3 & x_4 \end{pmatrix} \in \begin{pmatrix} \det \kappa \cdot \Nr(x_4-\fa'^{-1}\varpi^r x_3\beta')^{-1} & \vo \\ 0 & 1 \end{pmatrix} \iota_r'(x_4-\fa'^{-1}\varpi^r x_3\beta'), $$
and $\det \kappa \in 1+\vp^c, x_4-\fa'^{-1}\varpi^r x_3\beta' \in 1+\vp^c+\vp^{r+c}\beta' \subseteq (1+\vp^c)(1+\varpi^{c+r}\vO)$ hence $\Nr(x_4-\fa'^{-1}\varpi^r x_3\beta') \in 1+\vp^c$. We deduce that $\hat{f}_r'(\kappa)=1, \forall \kappa \in \gp{K}(\vp^c)$. Therefore
	$$ \Proj_c(\hat{f}_r')(1)=1. $$
	
	First suppose $\pi$ is not special. If $\pi$ is in the principal unitary series, then the norm of a function $f\in \pi$ is calculated by
	$$ \int_{\gp{B} \backslash \GL_2} |f(g)|^2 dg. $$
Otherwise, the norm is calculated via the intertwining operator $\Intw$
	$$ \Intw f(g) = \int_{\F} f(wn(x)g) dx, $$
as some constant multiple (depending only on $\mu_1,\mu_2$) of
	$$ \int_{\gp{B} \backslash \GL_2} f(g) \overline{\Intw f(g)} dg, $$
In the non principal series case, if $\tilde{\hat{f}}_r' \in \tilde{\pi} \simeq \pi(\mu_2,\mu_1)$ is the function defined by
	$$ \tilde{\hat{f}}_r'\left( \begin{pmatrix} a & * \\ 0 & d \end{pmatrix} \iota_r'(t) \right) = \left\lvert \frac{a}{d} \right\rvert^{1/2} \mu_2(a) \mu_1(d) \Omega^{-1}(t), \quad \forall a,d\in \F^{\times}, t \in \E^{\times}, $$
which is obvious a continuous section in $\mu_1,\mu_2$, then by the uniqueness of $\hat{v}$, we have
	$$ \Intw (\hat{f}_r') = \Intw (\hat{f}_r')(1) \tilde{\hat{f}}_r'. $$
Since $\Proj_c$ commutes with $\Intw $, we get
	$$ \Intw (\Proj_c(\hat{f}_r')) = \Intw (\hat{f}_r')(1) \Proj_c(\tilde{\hat{f}}_r'), $$
where $\Proj_c(\tilde{\hat{f}}_r')$ is supported in $\gp{K}_0(\vp^c)\iota_r'(\vO_r^{\times})$ and
	$$ \Proj_c(\tilde{\hat{f}}_r')(\kappa \iota_r'(t_0)) = (\mu_2,\mu_1)(\kappa) \Omega^{-1}(t_0), \quad \forall \kappa \in \gp{K}_0(\vp^c), t_0 \in \vO_r^{\times} $$
by the same argument as above. We deduce that, in any case
	$$ \frac{\lVert \Proj_c(\hat{f}_r') \rVert^2}{\lVert \hat{f}_r' \rVert^2} = \frac{{\rm Vol}(\gp{K}_0(\vp^c)\iota_r'(\vO_r^{\times}), d\kappa)}{C {\rm Vol}(\gp{Z} \backslash \gp{T}_r, dt)}, $$
where $dt$ is a the standard Tamagawa measure on $\gp{Z} \backslash \gp{T}_r \simeq \F^{\times} \backslash \E^{\times}$, $\gp{T}_r$ being the image of $\iota_r'$. By Proposition \ref{MeasCompar}, the constant $C > 0$ is such that
	$$ C db dt = \zeta_{\vp}(2) dg = db d\kappa, $$
where $dg$ is the Tamagawa measure on $\GL_2(\F)$, $db$ is the Tamagawa left invariant measure on $\gp{B}$. Like in the proof of Lemma \ref{IwaTorHaar}, from
	$$ \begin{pmatrix} y & x \\ & 1 \end{pmatrix} \begin{pmatrix} s+\fb' u & \varpi^r u \\ - \fa' \varpi^{-r} u & s \end{pmatrix} = \begin{pmatrix} ys+(\fb'y - \fa' \varpi^{-r} u) & \varpi^r u y + s x \\ -\fa' \varpi^{-r} u & s \end{pmatrix} $$
	and the computation of the relevant Jacobian, we deduce that
	$$ C/\zeta_{\vp}(2) = q^r / (\zeta_{\vp}(1) \zeta_{\E,\vp}(1)), $$
which concludes the proof in this case.
	
	Finally suppose $\pi = \mathrm{St}_{\mu}$. We want the above argument for complementary series to ``pass to limit'' using Lemma \ref{InnPSt}. It suffices to check that the continuous section $\hat{f}_r'(\sigma) \in \pi_{\sigma} := \mu \otimes \pi(\norm^{\sigma}, \norm^{-\sigma})$
	$$ \hat{f}_r'\left(\sigma, \begin{pmatrix} a & * \\ 0 & d \end{pmatrix} \iota_r'(t) \right) = \left\lvert \frac{a}{d} \right\rvert^{1/2+\sigma} \mu(ad) \Omega^{-1}(t), \quad \forall a,d\in \F^{\times}, t \in \E^{\times} $$
is admissible in the sense of Definiton \ref{AdmS}. Since $c \geq \cond(\mu)$, $\mu \circ \det$ is trivial on $\gp{K}(\vp^c)$. Hence it suffices to check that $\Proj_c(\hat{f}_r'(\sigma))$ is admissible. Decomposing the support of $\Proj_c(\hat{f}_r'(\sigma))$ into $\iota_r'(\vO_r^{\times})$ right-cosets, it suffices to check
	$$ \int_{\vO_r^{\times}} \Omega(t) \mu(\Nr_{\F}^{\E} t)^{-1} dt = 0. $$
	Since $\vO_r^{\times} = \vo^{\times}(1+\varpi^r \vO) \cap \vO^{\times}$, and if $\E/\F$ is unramified then $r=\cond(\Omega)-c < \cond(\Omega)$, while if $\E/\F$ is ramified then $2r \leq 2(r_0-1) \leq \cond(\Omega) - 1$, it suffices to show that $\cond(\Omega \cdot (\mu \circ \Nr)^{-1}) = \cond(\Omega)$, or $\cond(\mu \circ \Nr) \leq \cond(\Omega)-1$. But if $\E/\F$ is unramified, then $r = \cond(\Omega)-c \geq 1$ and for $k \geq 1$, $\Nr(1+\varpi^k \vO) \subset 1+\varpi^k \vo$, implying $\cond(\mu \circ \Nr) \leq \cond(\mu) \leq c < \cond(\Omega)$; while if $\E/\F$ is ramified, then $\Nr(1+\varpi^{k-1}\varpi_{\E}\vO) \subset 1+\varpi^k \vo$, implying $\cond(\mu \circ \Nr) \leq \max(2\cond(\mu)-1,0) \leq 2c-1 < \cond(\Omega)$, we are done.
\end{proof}

	It remains to prove Lemma \ref{BLNS}. To this end, we need to use two relevant characters. Recall (\ref{vOr*}) and (\ref{vO0*}). One character is
\begin{equation}
	\tilde{\Omega}_c: \vO_c^{\times} \to \C^{\times}, xy \mapsto \omega(x), \quad \forall x\in \vo^{\times}, y \in 1+\varpi^c \vO.
\label{BaseChar}
\end{equation}
	It is well-defined because $\omega=\mu_1\mu_2$ is trivial on $1+\vp^c = \vo^{\times} \cap (1+\varpi^c\vO)$. The other one, which makes sense only in the case $\E/\F$ is ramified, is defined by
\begin{equation}
	\tilde{\Omega}_c': \vO_{c-1}^{\times} \to \C^{\times}, xy \mapsto \omega(x), \quad \forall x\in \vo^{\times}, y \in 1+ \varpi_{\E}^{2c-1} \vO.
\label{BaseCharRam}
\end{equation}
	It is well-defined because $1+\vp^c = \vo^{\times} \cap (1+\varpi_{\E}^{2c-1}\vO)$ and for $x,y\in \vo$ we have
	$$ 1+\varpi^{c-1}(x+y\beta') = 1+\varpi^{c-1}(x+yu_0) + \varpi^{c-1}(\beta'-u_0)y \in 1+\vp^{c-1} + \varpi^{c-1}\varpi_{\E} \vO, $$
	$$ \Rightarrow 1+\varpi^{c-1}\vO = (1+\varpi^{c-1} \varpi_{\E}\vO)(1+\varpi^{c-1}\vo). $$
\begin{proof}[Proof of Lemma \ref{BLNS}]
\noindent \textbf{Case 1:} First assume $r>0$. Recall (\ref{SppS1}).
\begin{itemize}
	\item[-] $\gp{K}_0(\vp^c) \cap \iota_r'(\vO_r^{\times}) = \iota_r'(\vO_{r+c}^{\times})$, and the restriction of the character $(\mu_1,\mu_2)$ of $\gp{K}_0(\vp^c)$ to $\iota_r'(\vO_{r+c}^{\times})$ corresponds to the character $\tilde{\Omega}_{r+c}$ of $\vO_{r+c}^{\times}$. We deduce that $\mathcal{S}_1(c)$ is isomorphic to $ \Ind_{\vO_{r+c}^{\times}}^{\vO_r^{\times}} \tilde{\Omega}_{r+c} $, hence contains each character $\Omega'$ of $\vO_r^{\times}$ restricting to $\omega$ on $\vo^{\times}$ with $r_0' \leq r+c$ once.
	\item[-] If $\kappa = \begin{pmatrix} c_1 & c_2 \\ c_3 & c_4 \end{pmatrix} \in \gp{K} - \gp{K}_0(\vp^c)\iota_r'(\vO_r^{\times}) $, then $c_4 \in \vp$. For any $x\in \vo^{\times}$ and $y \in \varpi^{r+c-1} \vo$, $\kappa \iota_r'(x+\beta' y) \kappa^{-1} $ is equal to
	$$ \begin{pmatrix} x+\frac{c_1c_4\fb'-c_2c_4\fa' \varpi^{-r} -c_1c_3\varpi^r}{\det \kappa}y & * \\ \frac{-c_4^2 \fa' \varpi^{-r} + c_3 c_4 \fb' - c_3^2 \varpi^r}{\det \kappa} y & x+\frac{-c_2c_3\fb'+c_2c_4\fa' \varpi^{-r} +c_1c_3\varpi^r}{\det \kappa}y \end{pmatrix} $$
which lies in $\gp{K}_0(\vp^c)$, and on which $(\mu_1,\mu_2)$ acts as $\tilde{\Omega}_{r+c-1}$. Hence $\mathcal{S}_{\kappa}(c)$ only contains characters $\Omega'$ of $\vO_r^{\times}$ restricting to $\omega$ on $\vo^{\times}$ with $r_0' \leq r+c-1$.
\end{itemize}

\noindent \textbf{Case 2:} Assume $r=0$. Recall (\ref{DecompK}).
\begin{itemize}
	\item[-] $\gp{K}_0(\vp^c) \cap \iota_0'(\vO^{\times}) = \iota_0'(\vO_c^{\times})$, and the restriction of the character $(\mu_1,\mu_2)$ of $\gp{K}_0(\vp^c)$ to $\iota_0'(\vO_c^{\times})$ corresponds to the character $\tilde{\Omega}_c$ of $\vO_c^{\times} $. We deduce that $\mathcal{S}_1(c)$ is isomorphic to $ \Ind_{\vO_c^{\times}}^{\vO^{\times}} \tilde{\Omega}_c. $, hence contains each character $\Omega'$ of $\vO^{\times}$ restricting to $\omega$ on $\vo^{\times}$ with $r_0' \leq c$ once.
	\item[-] Assume $\E/\F$ is ramified. Let $u \equiv u_0 \pmod{\vp}$. For any $x+y\beta' \in \vO_{c-1}^{\times}$ we have
\begin{equation}
	n_-(u) \iota_0(x+y\beta') n_-(-u) = \begin{pmatrix} x+y(\fb' -u) & y \\ -(u^2-\fb' u+\fa')y & x+uy \end{pmatrix} \in \gp{K}_0(\vp^c),
\label{Conjn_u_0}
\end{equation}
on which the character $(\mu_1,\mu_2)$ acts as $\tilde{\Omega}_c'$. In fact, $x+y\beta' \in 1+\varpi_{\E}^{2c-1}\vO$ if and only if we can write $x=1+\varpi^{c-1}x', y=\varpi^{c-1}y' $ for some $x',y' \in \vo, x'+u_0 y' \in \vp$. Since $\fb'-2u_0 = \Tr(\beta'-u_0) \in \vp$, we get
	$$ x'+y'(\fb' -u) \equiv x'+uy' \equiv x'+u_0y' \equiv 0 \pmod{\vp}. $$
	Consequently, we have
\begin{align*}
	&\quad \mu_1(x+y(\fb' -u)) \cdot \mu_2(x+u y) \\
	&= \mu_1(1+\varpi^{c-1}(x'+ y' (\fb'-u))) \cdot \mu_2(1+\varpi^{c-1}(x'+uy')) = 1.
\end{align*}
	Hence $\mathcal{S}_{n_-(u)}(c)$ contain only characters $\Omega'$ of $\vO^{\times}$ restricting to $\omega$ on $\vo^{\times}$ and trivial on  $1+ \varpi_{\E}^{2c-1} \vO$.
\end{itemize}
	We have thus completed the verification of all assertions.
\end{proof}

	\subsection{$\B$-Split, Finite Place, $\E$-nonsplit, $\pi$ supercuspidal}
	
\begin{lemma}
	Let $e=e(\E/\F)$ be the ramification index and $r_0 = \IntCP{\cond(\Omega)/e(\E/\F)}$ as before. Take any $c \geq \max(\IntP{\cond(\pi)/2} + e - 1, \IntCP{\cond(\pi)/2} )$. (In particular, $c = \cond(\pi)$ is admissible since $\cond(\pi) \geq 2$ for any supercuspidal $\pi$.)
\begin{itemize}
	\item[(1)] Assume $r_0 \leq c-e+1$. Then we have $\hat{v}_0 = \hat{v}_{0,\Omega} \in [\pi; c]$, i.e., $\Proj_c(\hat{v}_0) = \hat{v}_0$. Consequently, by (\ref{NSPFRAverage})
	$$ \frac{\alpha([\pi;c]; \Omega, \iota_0)}{\Vol(\gp{Z} \backslash \gp{T})} = 1. $$	
	\item[(2)] Assume $r_0 > c-e+1$ (in particular, $r_0 \geq c$). We have
	$$ \lVert \Proj_c(\hat{v}_r) \rVert^2 \leq \lVert \hat{v}_r \rVert^2 q^{c-r_0} \cdot \left\{ \begin{matrix} (1+q^{-1})^{-1} & \text{if } \E/\F \text{ unramified;} \\ 1/2 & \text{if } \E/\F \text{ ramified.} \end{matrix} \right. $$
Equality holds if and only if $r=r_0-c+e-1$. Consequently, by (\ref{NSPFRAverage}),
	$$ \frac{\alpha([\pi;c]; \Omega, \iota_r)}{\Vol(\gp{Z} \backslash \gp{T})} \leq q^{c-r_0} \cdot \left\{ \begin{matrix} (1+q^{-1})^{-1} & \text{if }\E/\F \text{ unramified;} \\ 1/2 & \text{if } \E/\F \text{ ramified.} \end{matrix} \right. $$
Equality holds if and only if $r=r_0-c+e-1$.
\end{itemize}
\label{LocalEst6}
\end{lemma}

	It will be convenient to write for each character $\Omega$ of $\E^{\times}$
\begin{equation}
	\tilde{\cond}(\Omega) = \IntCP{\cond(\Omega)/e(\E/\F)}.
\label{CondVar}
\end{equation}
	
	First note that if $\pi$ is not minimal supercuspidal, then there is a minimal supercuspidal $\vartheta$ and a character $\chi$ of $\F^{\times}$, such that
	$$ \pi \simeq \vartheta \otimes (\chi \circ \det ), \cond(\pi)=2\cond(\chi) > \cond(\vartheta). $$
For $c > \cond(\pi)/2=\cond(\chi)>\cond(\vartheta)/2$, $\chi$ is trivial on $\gp{K}(\vp^c)$. We must have
	$$ [\pi;c] = [\vartheta;c] \otimes (\chi \circ \det ). $$
It is also easy to see
	$$ \hat{v}_{r,\Omega} = \hat{v}_{r, \Omega \chi} \otimes 1, $$
where $\hat{v}_{r, \Omega \chi}$ lies in $\vartheta$ and we have identified $\chi$ with $\chi \circ \Nr_{\F}^{\E}$. Since $c > \cond(\chi) \geq \tilde{\cond}(\chi \circ \Nr_{\F}^{\E})$, we have
	$$ \tilde{\cond}(\Omega) \geq c \Leftrightarrow \tilde{\cond}(\Omega \chi) \geq c; \quad \tilde{\cond}(\Omega) > c \Leftrightarrow \tilde{\cond}(\Omega \chi) > c. $$
Thus, Lemma \ref{LocalEst6} for $\pi$ is a consequence of it for $\vartheta$.

	For various types of minimal supercuspidals, the proofs of Lemma \ref{LocalEst6} are similar. The case of Type 0 is much simpler than the other cases. We treat the other cases in detail and will comment on Type 0 afterwards.
%We treat the case of Type 1 minimal supercuspidals in detail, and will give brief comments on the other cases afterwards. 
At this point, we recommend the reader to review our notations for the construction of such representations in \S \ref{T1MS}-\ref{T3MS}, such as $m, \bL, \gp{J}, \gp{J}^0, \lambda, \rho$, etc.

\begin{lemma}
	Let $e'=e(\bL/\F)$ be the ramification index. If there is some non zero function $f \in [\pi;c]$, then we must have $c \geq \IntCP{\cond(\pi)/2}$ and the support of $f$ is contained in
	$$ \B^{(\leq c-\IntP{\cond(\pi)/2})} := \bigsqcup_{l=e'-1}^{c-\IntP{\cond(\pi)/2}} \gp{J} \begin{pmatrix} 1 & \\ & \varpi^l \end{pmatrix} \gp{K}. $$
Reciprocally, any function supported in the above set is $\gp{K}(\vp^c)$-invariant. Consequently, the projector $\Proj_c$ is given by
	$$ f \mapsto f 1_{\B^{(\leq c-\IntP{\cond(\pi)/2})}}. $$
The projection from $\pi$ to $\sigma_0=\sigma_0(\pi)$ is given by
	$$ f \mapsto f 1_{\B^{(\IntCP{\cond(\pi)/2})}}, \quad \B^{(m)} := \B^{(\leq m)} - \B^{(\leq m-1)}. $$
\label{InvSpace1}
\end{lemma}
\begin{proof}
	If we have some $\gp{K}(\vp^c)$-invariant $f \neq 0$, there must be some $\gp{J} a_-(\varpi^l) \gp{K}$ with $l \geq e'-1$ on which $f \neq 0$, in view of Proposition \ref{VarCartan} (1) \& (2). We may assume $ f(a(\varpi^l) \kappa_0) \neq 0$ for some $\kappa_0 \in \gp{K}$. Take
	$$ y \in \vp^{c-l} \cap \left\{ \begin{matrix} \vp^{m+1} & \text{if } \pi \text{ is of Type 1 with level } 2m+1 \\ \vp^m & \text{if } \pi \text{ is of Type 2 with level } 2m \\ \vp^{\IntCP{m/2}} & \text{if } \pi \text{ is of Type 3 with level } m+1/2. \end{matrix} \right. $$
	Then we have
\begin{align*}
	f(a_-(\varpi^l) \kappa_0) &= f(a_-(\varpi^l) \kappa_0 \kappa_0^{-1} n(\varpi^l y) \kappa_0) \quad (\text{ since } \kappa_0^{-1} n(\varpi^l y) \kappa_0 \in \gp{K}(\vp^c)) \\
	&= f(n(y) a_-(\varpi^l) \kappa_0) = \begin{matrix} \lambda(n(y)) \\ \rho(n(y)) \end{matrix} f(a_-(\varpi^l) \kappa_0) \quad (\text{ since } n(y) \in \gp{J}),
\end{align*}
	i.e., $0 \neq f(a_-(\varpi^l) \kappa_0)$ is an invariant vector in $V_{\lambda}$ (Type 1 and 3) or $V_{\rho}$ (Type 2) for $n(y)$. Such a vector does not exist if $y$ can run over $\vp^{2m+1}$ (Type 1) resp. $\vp^{2m}$ (Type 2) resp. $\vp^m$ (Type 3) according to our recall in \S \ref{T1MS}, \S \ref{T2MS} and \S \ref{T3MS} on the branching law of $\lambda$ restricted to $n(\vp^{\IntP{\cond(\pi)/2}-1})$. Hence if $c-l \leq \IntP{\cond(\pi)/2}-1$, we would get a contradiction. Thus we necessarily have $c-l \geq \IntP{\cond(\pi)/2}$, i.e., $c \geq \IntP{\cond(\pi)/2}+e'-1 = \IntCP{\cond(\pi)/2}$ and $l \leq c-\IntP{\cond(\pi)/2}$.

\noindent Assume $c-l \geq \IntP{\cond(\pi)/2}$. The functions with support contained in $\B^{(\leq c-\IntP{\cond(\pi)/2})}$ are $\gp{K}(\vp^c)$-invariant because for any $\kappa \in \gp{K}$
	$$ a_-(\varpi^l) \kappa \gp{K}(\vp^c) \kappa^{-1} a_-(\varpi^{-l}) \subset \gp{K}(\vp^{c-l}) \subset \gp{K}(\vp^{\IntP{\cond(\pi)/2}}), $$
	on which $\lambda$ (Type 1 and 3) or $\rho$ (Type 2) is trivial. (In the case of Type 2, $\rho \mid_{\gp{H}^1}$ is a repeated copy of $\lambda$ trivial on the above group.)

\noindent The assertion concerning $\sigma_0$ follows from $\pi^{\gp{K}(\vp^{\cond(\pi)-1})} \oplus \sigma_0 = \pi^{\gp{K}(\vp^{\cond(\pi)})}$, c.f. \cite[Theorem 5.2]{Wu14}.
\end{proof}

\begin{lemma}
	Recall (\ref{CondVar}) and $e_u$ defined in \S \ref{T2MS}. Let $s=\max( \tilde{\cond}(\Omega)-\IntP{\cond(\pi)/2},0)-1+e(\E/\F)$ and $\hat{v}_0$ correspond to a function $\hat{f}_0$. Up to a constant multiple, we have:
\begin{itemize}
	\item[-] If $s \geq e(\E/\F)$, then there is a unique $z\in \vo^{\times} / (1+\vp^{\IntP{(\cond(\pi)+1)/4}})$ such that $\hat{f}_0$ is supported in $\gp{J} a_-(z\varpi^s) \gp{T}_0$ and for some unique $u \in 1+\vp^{\IntP{(\cond(\pi)+1)/4}} / 1+\vp^{\IntP{(\cond(\pi)+1)/4}+1}$ in the Type 2 case
	$$ \hat{f}_0(xa_-(z\varpi^s)\iota_0(t)) = \left\{ \begin{matrix} \lambda(x)\Omega^{-1}(t) & \text{Type 1 \& 3} \\ \Omega^{-1}(t)\rho(x).e_u & \text{Type 2} \end{matrix} \right., \quad \forall x\in \gp{J}, t\in \E^{\times}. $$
	\item[-] If $s=e(\E/\F)-1$, then there is a unique double coset $\gp{J}^0 \kappa_0 \iota_s(\vO^{\times}) \in \gp{J}^0 \backslash \gp{K} /\iota_s(\vO^{\times})$ (Type 1 \& 2) or $\in \gp{J}^0 \backslash \gp{K}_0(\vp) /\iota_s(\vO^{\times})$ (Type 3) such that $\hat{f}_0$ is supported in $\gp{J} \kappa_0 a_-(\varpi^s) \gp{T}_0$ and for some unique $u \in 1+\vp^{\IntP{(\cond(\pi)+1)/4}} / 1+\vp^{\IntP{(\cond(\pi)+1)/4}+1}$ in the Type 2 case
	$$ \hat{f}_0(x\kappa_0a_-(\varpi^s)\iota_0(t)) = \left\{ \begin{matrix} \lambda(x)\Omega^{-1}(t) & \text{Type 1 \& 3} \\ \Omega^{-1}(t)\rho(x).e_u & \text{Type 2} \end{matrix} \right., \quad \forall x\in \gp{J}, t\in \E^{\times}. $$
\end{itemize}
\label{IdOmega1}
\end{lemma}
\begin{proof}
	We prove case by case. First consider the Type 1 case. Proposition \ref{VarCartan} and Lemma \ref{KdecompJ^0} yield the following decomposition
	$$ \GL_2 = \left\{ \begin{matrix} \sideset{}{_{\kappa_0}} \bigsqcup \gp{J} \kappa_0 \gp{T}_0 \bigsqcup \sideset{}{_{l \geq 1}} \bigsqcup \sideset{}{_z} \bigsqcup \gp{J} a_-(z\varpi^l) \gp{T}_0 & \text{if } \E/\F \text{ is unramified,} \\ \sideset{}{_{\kappa_0}} \bigsqcup \gp{J} \kappa_0'a_-(\varpi) \gp{T}_0 \bigsqcup \sideset{}{_{l \geq 2}} \bigsqcup \sideset{}{_z} \bigsqcup \gp{J} a_-(z\varpi^l) \gp{T}_0 & \text{if } \E/\F \text{ is ramified,} \end{matrix} \right. $$
where $\kappa_0$ resp. $\kappa_0'$ runs over a set of representatives of $\gp{J}^0 \backslash \gp{K} / \iota_0(\vO^{\times})$ resp. $\gp{J}^0 \backslash \gp{K} / \iota_1(\vO^{\times})$ and $z$ runs over a set of representatives of $\vo^{\times} / (1+\vp^{m+1})$. Let $\mathcal{S}(l,z)$ be the space of functions supported in $\gp{J} a_-(z\varpi^l) \gp{T}_0$ if $l \geq e(\E/\F)$; let $\mathcal{S}(\kappa_0)$ resp. $\mathcal{S}(\kappa_0')$ be the space of functions supported in $\gp{J} \kappa_0 \gp{T}_0$ resp. $\gp{J} \kappa_0' a_-(\varpi) \gp{T}_0$. Then we have a decomposition as $\gp{T}_0$-representations
	$$ \Res_{\gp{T}_0}^{\GL_2} \pi = \left\{ \begin{matrix} \sideset{}{_{\kappa_0}} \bigoplus \mathcal{S}(\kappa_0) \bigoplus \sideset{}{_{l \geq 1,z}} \bigoplus \mathcal{S}(l,z) & \text{if } \E/\F \text{ is unramified,} \\ \sideset{}{_{\kappa_0'}} \bigoplus \mathcal{S}(\kappa_0') \bigoplus \sideset{}{_{l \geq 2,z}} \bigoplus \mathcal{S}(l,z) & \text{if } \E/\F \text{ is ramified.} \end{matrix} \right. $$
	
\noindent Consider the case $l \geq e(\E/\F)$ and write $\tilde{l} := l + 1 - e(\E/\F) = l - v(\fa)$. The proof of Lemma \ref{KdecompJ^0} shows (with $z_1=z_2=z$) that
\begin{align*}
	&\quad a_-(z\varpi^l) \gp{T}_0 a_-(z\varpi^l)^{-1} \cap \gp{J} = \gp{Z} \left( a_-(z) \iota_l(\vO^{\times}) a_-(z)^{-1} \cap \gp{K} \cap \gp{J}^0 \right) \\
	&= \gp{Z} \left( a_-(z) \iota_l(\vO_{\tilde{l}}^{\times}) a_-(z)^{-1} \cap \gp{J}^0 \right) = \gp{Z} \left( a_-(z) \iota_l(\vO_{\tilde{l}+m+1}^{\times}) a_-(z)^{-1} \right),
\end{align*}
on which $\lambda$ acts as a character $\Omega_z$ of $\F^{\times}\vO_{\tilde{l}+m+1}^{\times} = \F^{\times}(1+\vp^{\tilde{l}+m+1}\beta)$ defined by
	$$ \Omega_z(x(1+y\beta)) = \omega(x) \tilde{\psi}(\varpi^{-l-2m-1}y(-\fa_0 \fa z^{-1} + \varpi^{2l} z + \fa_1 \fb \varpi^l)), \quad x\in \F^{\times}, y\in \vp^{\tilde{l}+m+1}. $$
The function
	$$ \vo^{\times} / (1+\vp^{m+1}) \to \vo^{\times} / (1+\vp^{m+1}), \quad z \mapsto \varpi^{-v(\fa)} (-\fa_0 \fa z^{-1} + \varpi^{2l} z + \fa_1 \fb \varpi^l)) $$
is injective, since $ - \fa_0 \fa z_1^{-1} + \varpi^{2l} z_1 = - \fa_0 \fa z_2^{-1} + \varpi^{2l} z_2 \pmod{\vp^{m+1+v(\fa)}} $ if and only if $(z_1 - z_2)(\fa_0 \fa z_1^{-1}z_2^{-1}+ \varpi^{2l}) \in \vp^{m+1+v(\fa)}$ if and only if $z_1 = z_2 \pmod{\vp^{m+1}}$. It is thus bijective since $\vo^{\times} / (1+\vp^{m+1})$ is finite. As $z$ runs over $\vo^{\times} / (1+\vp^{m+1})$, the function
	$$ \vp^{\tilde{l}+m+1}/\vp^{\tilde{l}+2m+2} \to \C^{\times}, \quad y \mapsto \tilde{\psi}(\varpi^{-l-2m-1}y(-\fa_0 \fa z^{-1} + \varpi^{2l} z + \fa_1 \fb \varpi^l)) $$
runs over all the characters of $\vp^{\tilde{l}+m+1}/\vp^{\tilde{l}+2m+2}$ not trivial on $\vp^{\tilde{l}+2m+1}/\vp^{\tilde{l}+2m+2}$. Hence $\Omega_z$ runs over all the characters of $\F^{\times}\vO_{\tilde{l}+m+1}^{\times}$ restricting to $\omega$ on $\F^{\times}$ with $\tilde{\cond}(\Omega_z)=\tilde{l}+2m+2$. As $\mathcal{S}(l,z) \simeq \Ind_{\F^{\times}\vO_{\tilde{l}+m+1}^{\times}}^{\E^{\times}} \Omega_z$ which is the direct sum of all characters $\Omega'$ of $\E^{\times}$ restricting to $\Omega_z$ on $\F^{\times}\vO_{\tilde{l}+m+1}^{\times}$, we see that
	$$ \sideset{}{_z} \bigoplus \mathcal{S}(l,z) $$
contains all characters $\Omega'$ of $\E^{\times}$ with $\tilde{\cond}(\Omega') = \tilde{l}+2m+2$ with multiplicity one. We conclude by the multiplicity one result of Waldspurger \cite{Wa85}.

	Now we turn to Type 2 case. After obtaining
	$$ a_-(z\varpi^l) \gp{T}_0 a_-(z\varpi^l)^{-1} \cap \gp{J} = \gp{Z} \left( a(z) \iota_l(\vO_{\tilde{l}+m}^{\times}) a(z)^{-1} \right), $$
	we must argue more carefully as follows. Taking $y \in \vp^{\tilde{l}+m}$, we get
	$$ a_-(z) \iota_l(1+y\beta) a_-(z)^{-1} = \begin{pmatrix} 1 & -\fa z^{-1} y \varpi^{-l} \\ zy\varpi^l & 1+y\fb \end{pmatrix} = \begin{pmatrix} 1 & 0 \\ zy\varpi^l & 1+y\fb + \fa y^2 \end{pmatrix} n(-\fa z^{-1} y \varpi^{-l}). $$
	The first matrix lies in $\gp{K}(\vp^{m+1})$, on which $\rho = \eta = \lambda^{\oplus q}$ acts as multiplication by 
	$$ \tilde{\psi}(\varpi^{-2m} (\fa_1 \fb + \varpi^l z )y). $$
	(\ref{BLeta}) shows that the action of $\rho$ through the second matrix decomposes as the direct sum of
	$$ \tilde{\psi}(-\varpi^{-l-2m} u z^{-1} \fa_0 \fa y), \quad u \in 1+\vp^m / 1+\vp^{m+1}. $$
	It follows that
	$$ \mathcal{S}(s,z) = \sideset{}{_{u \in 1+\vp^m / 1+\vp^{m+1}}} \bigoplus \Ind_{\F^{\times} \vO_{\tilde{l}+m}^{\times}}^{\E^{\times}} \Omega_{z,u}, $$
	where $\Omega_{z,u}$ is a character of $\F^{\times} \vO_{s+m}^{\times}$ defined by
	$$ \Omega_{z,u}(x(1+y\beta)) = \omega(x) \tilde{\psi}(\varpi^{-l-2m}y(- \fa_0 \fa u z^{-1} + \varpi^{2l}z + \fa_1 \fb \varpi^l)), \quad x \in \F^{\times}, y \in \vp^{\tilde{l}+m}. $$
	We observe that replacing $\varpi^{2l} z$ with $\varpi^{2l} u^{-1}z$ does not change the above equation, and that as $z,u$ vary, $z' := u^{-1}z$ runs over $\vo^{\times}/(1+\vp^{m+1})$. Hence
	$$ \sideset{}{_z} \bigoplus \mathcal{S}(l,z) = \sideset{}{_{z' \in \vo^{\times}/(1+\vp^{m+1})}} \bigoplus \Ind_{\F^{\times} \vO_{\tilde{l}+m}^{\times}}^{\E^{\times}} \Omega_{z'}, $$
	$$ \Omega_{z'}(x(1+y\beta)) = \omega(x) \tilde{\psi}(\varpi^{-l-2m}y(- \fa_0 \fa z'^{-1} + \varpi^{2l}z' + \fa_1 \fb \varpi^l)), \quad x \in \F^{\times}, y \in \vp^{\tilde{l}+m}. $$
	This is the same type of decomposition as in Type 1 case, hence the rest of the argument goes through in the same way.
	
	Finally we treat Type 3 case. By considering the valuation of the determinant, we get
\begin{align*}
	&\quad a_-(z\varpi^l) \gp{T}_0 a_-(z\varpi^l)^{-1} \cap \gp{J} \\
	&= \left\{ \begin{matrix} \gp{Z} \left( a_-(z) \iota_l(\vO^{\times}) a_-(z)^{-1} \cap \gp{J}^0 \right) & \E/\F \text{ unramified,} \\ \gp{Z} \left( a_-(z) \iota_l(\vO^{\times}) a_-(z)^{-1} \cap \gp{J}^0 \right) \bigsqcup \gp{Z} \left( a_-(z) \iota_l(\beta \vO^{\times}) a_-(z)^{-1} \cap \alpha \gp{J}^0 \right) & \E/\F \text{ ramified.} \end{matrix} \right.
\end{align*}
	If the second component in the $\E/\F$ ramified case is non empty, i.e., if
\begin{equation}
	\exists \gamma \in a_-(z) \iota_l(\beta \vO^{\times}) a_-(z)^{-1} \cap \alpha \gp{J}^0,
\label{RRNExt}
\end{equation}
then we deduce that
	$$ a_-(z) \iota_l(\beta \vO^{\times}) a_-(z)^{-1} \cap \alpha \gp{J}^0 = \gamma \left( a_-(z) \iota_l(\vO^{\times}) a_-(z)^{-1} \cap \gp{J}^0 \right). $$
	Such $\gamma$ will give conditions on $\Omega(\beta)$ in the branching law and has no effect on the conductor. Thus in both cases, we only need to consider (by Lemma \ref{KdecompJ^0})
	$$ \gp{Z} \left( a_-(z) \iota_l(\vO^{\times}) a_-(z)^{-1} \cap \gp{J}^0 \right) = \gp{Z} a_-(z) \iota_l(\vO_{\tilde{l}+\IntCP{m/2}}^{\times}) a_-(z)^{-1}, $$
on which $\lambda$ acts as a character $\Omega_z$ of $\F^{\times}\vO_{\tilde{l}+\IntCP{m/2}}^{\times} = \F^{\times}(1+\vp^{\tilde{l}+\IntCP{m/2}}\beta)$ defined by
	$$ \Omega_z(x(1+y\beta)) = \omega(x) \tilde{\psi}(\varpi^{-l-m-1}y(-\fa_0 \fa z^{-1} + \varpi^{2l} z + \fa_1 \fb \varpi^l)), \quad x\in \F^{\times}, y\in \vp^{\tilde{l}+\IntCP{m/2}}. $$
	The rest of the argument is the same as in Type 1 case.
\end{proof}
\begin{remark}
	Moreover, (\ref{RRNExt}) is possible only if $l=1=e(\E/\F)-1$. In fact, we have
	$$ a_-(z) \iota_l(\beta \vO^{\times}) a_-(z)^{-1} = \left\{ \begin{pmatrix} x & - z^{-1} \fa \varpi^{-l} y \\ z \varpi^l y & x+by \end{pmatrix} \ |\ x \in \vp, y \in \vo^{\times} \right\}, $$
	$$ \alpha \gp{J}^0 \subset \oJ. $$
\label{RRRNExt}
\end{remark}

\begin{proof}[Proof of Lemma \ref{LocalEst6}]
	For the first part, by Lemma \ref{InvSpace1} and \ref{IdOmega1}, it suffices to show that
	$$ \begin{matrix} \text{if } \E/\F \text{ unramified} & \gp{J} a_-(\varpi^s) \gp{K}_{\oM} \\ \text{if } \E/\F \text{ ramified} & \gp{J} a_-(\varpi^s) \gp{K}_{\oJ} \end{matrix} \subset \B^{(\leq c-\IntP{\cond(\pi)/2})} = \sideset{}{_{l \leq c-\IntP{\cond(\pi)/2}}} \bigsqcup \gp{J} a_-(\varpi^l) \gp{K} $$
for $s = \max( r_0-\IntP{\cond(\pi)/2},0)-1+e$, since the support of $\hat{f}_0$ is contained in the LHS. This is the case by the ``moreover'' part of Proposition \ref{VarCartan} (1) and
	$$ s \leq \max(c-e+1-\IntP{\cond(\pi)/2},0) - 1+e = c-\IntP{\cond(\pi)/2}. $$

	For the second part, let $s=\tilde{\cond}(\Omega)-\IntP{\cond(\pi)/2}+e-1 \geq e$, $\tilde{s} = s +1 -e = \tilde{\cond}(\Omega)-\IntP{\cond(\pi)/2}$. Note that $\hat{f}_0$ is of modulus $1$ on its support, hence the same is true for $\hat{f}_r = \pi(a_-(\varpi^r)).\hat{f}_0$ for any $r\in \N$. Since, by Lemma \ref{IdOmega1}, the support of $\hat{f}_0$ is $\gp{J} a(z\varpi^s) \gp{T}_0$ for some $z \in \vo^{\times}$, we get, by Corollary \ref{InvSpace1},
	$$ \lVert \hat{f}_r \rVert^2 = {\rm Vol}(\gp{J} \backslash \gp{J} a_-(z\varpi^s) \gp{T}_0 a_-(\varpi^{-r})) = {\rm Vol}(\gp{J} \backslash \gp{J} a_-(z\varpi^s) \gp{T}_0); $$
	$$ \lVert \Proj_c(\hat{f}_r) \rVert^2 = {\rm Vol} \left( \gp{J} \backslash \left( \gp{J} a_-(z\varpi^s) \gp{T}_0 a_-(\varpi^{-r}) \cap \B^{(\leq c-\IntP{\cond(\pi)/2})} \right) \right). $$
	As $\gp{J}$ is compact open modulo center, the relevant measure for the above volume function can be taken as the counting measure. But by the proof of Lemma \ref{IdOmega1} (and Remark \ref{RRRNExt}), the orbits of $\gp{J}$ on
	$$ \gp{J} a_-(z\varpi^s) \gp{T}_0 \quad \text{resp.} \quad \gp{J} a_-(z\varpi^s) \gp{T}_0 a_-(\varpi^{-r}) \cap \B^{(\leq c-\IntP{\cond(\pi)/2})} $$
	is in bijection with those of $\gp{T}_0 \cap a_-(z\varpi^s)^{-1} \gp{J} a_-(z\varpi^s) = \gp{Z}\iota_0(\vO_{\tilde{s}+\IntCP{(\cond(\pi)-3)/4}}^{\times})$ on
	$$ \gp{T}_0 \quad \text{resp.} \quad \gp{T}_0 \cap a_-(z\varpi^s)^{-1} \B^{(\leq c-\IntP{\cond(\pi)/2})} a_-(\varpi^r). $$
	Consequently, we obtain
	$$ \lVert \hat{f}_r \rVert^2 = \extnorm{ \F^{\times} \vO_{\tilde{s}+\IntCP{(\cond(\pi)-3)/4}}^{\times} \backslash \E^{\times} }. $$
	For $\lVert \Proj_c(\hat{f}_r) \rVert^2$, we need to compute
\begin{align*}
	&\quad \gp{T}_0 \cap a_-(z\varpi^s)^{-1} \B^{(\leq c-\IntP{\cond(\pi)/2})} a_-(\varpi^r) = \sideset{}{_{n \in \Z}} \bigsqcup \varpi^n \cdot \\
	& \left( \gp{T}_0 \cap a_-(z\varpi^s)^{-1} \begin{Bmatrix} \sideset{}{_{k \leq c-\IntP{\cond(\pi)/2}}} \bigsqcup \gp{K} a_-(\varpi^k) \gp{K} & \text{If } \pi \text{ is of Type 1 or 2} \\ \sideset{}{_{k \leq c-\IntP{\cond(\pi)/2}-1}} \bigsqcup \gp{K} a_-(\varpi^k) \gp{K} \bigsqcup \gp{K}(l) & \text{If } \pi \text{ is of Type 3}  \end{Bmatrix} a_-(\varpi^r) \right)
\end{align*}
	where $\gp{K}(l) = \gp{K}_0(\vp) a_-(\varpi^l) \gp{K} = \gp{K} a_-(\varpi^l) \gp{K} \cap \Pi \gp{K} a_-(\varpi^{l-1}) \gp{K}$ by the proof of Proposition \ref{VarCartan} (1) and Remark \ref{DCDes}. Still by Remark \ref{DCDes}, we can deduce that for any $l \geq 0$
\begin{align*}
	&\quad \gp{T}_0 \cap a_-(z\varpi^s)^{-1} \left( \sideset{}{_{k \leq l}} \bigsqcup \gp{K} a_-(\varpi^k) \gp{K} \right) a_-(\varpi^r) \\
	&= \left\{ \begin{matrix} \iota_0(\varpi^{r-s} \vO_{s-v(\fa)}^{(\leq l-r+s)}) & \text{if } r \geq s \\ \iota_0(\vO_{r-v(\fa)}^{(\leq l-s+r)}) & \text{if } r < s \end{matrix} \right. = \iota_0(\varpi^{\max(r-s,0)} \vO_{\min(r-v(\fa),s-v(\fa))}^{(\leq l - \norm[r-s])}),
\end{align*}
\begin{align*}
	 \gp{T}_0 \cap a_-(z\varpi^s)^{-1} \gp{K}(l) a_-(\varpi^r) &= \iota_0 \left( \varpi^{\max(r+1-s,0)} \vO_{\min(r-v(\fa),s-v(\fa)-1)}^{(l-1-\norm[r+1-s])} \cap \varpi^{\max(r-s,0)} \vO_{\min(r-v(\fa),s-v(\fa))}^{(l - \norm[r-s])} \right) \\
	 &= \left\{ \begin{matrix} \varpi^{r+1-s} \vO_{(\tilde{s}-1)}^{(l-2-r+s)} & \text{if } r \geq s \\ \vO_{r-v(\fa)}^{(l+r-s)} & \text{if } r < s, \end{matrix} \right.
\end{align*}
	where we have used the notations in (\ref{ArithQuadSet}). Thus we get in the cases of Type 1 and Type 2
	$$ \frac{\lVert \Proj_c(\hat{f}_r) \rVert^2}{\lVert \hat{f}_r \rVert^2} = \frac{\extnorm{ \vO_{\tilde{s}+\IntCP{(\cond(\pi)-3)/4}}^{\times} \backslash \vO_{\min(r-v(\fa),s-v(\fa))}^{(\leq c-\IntP{\cond(\pi)/2}-|r-s|)} }}{\extnorm{\vO_{\tilde{s}+\IntCP{(\cond(\pi)-3)/4}}^{\times} \backslash (\vO - \varpi \vO)}} = \frac{\extnorm{\vO_{\min(r-v(\fa),s-v(\fa))}^{\times} \backslash \vO_{\min(r-v(\fa),s-v(\fa))}^{(\leq c-\IntP{\cond(\pi)/2}-|r-s|)}}}{\extnorm{\vO_{\min(r-v(\fa),s-v(\fa))}^{\times} \backslash (\vO - \varpi \vO)}}; $$
	while in the case of Type 3
\begin{equation}
	\frac{\lVert \Proj_c(\hat{f}_r) \rVert^2}{\lVert \hat{f}_r \rVert^2} = \left\{ \begin{matrix} \frac{\extnorm{\vO_{s-v(\fa)}^{\times} \backslash \left( \vO_{s-v(\fa)}^{(\leq c-\IntP{\cond(\pi)/2}-1-|r-s|)} \sqcup \vO_{(s-v(\fa)-1)}^{(c-\IntP{\cond(\pi)/2}-2-|r-s|)} \right)} }{\extnorm{\vO_{s-v(\fa)}^{\times} \backslash (\vO - \varpi \vO)}} & \text{if } r \geq s \\ \frac{\extnorm{\vO_{r-v(\fa)}^{\times} \backslash \vO_{r-v(\fa)}^{(\leq c-\IntP{\cond(\pi)/2}-|r-s|)}}}{\extnorm{\vO_{r-v(\fa)}^{\times} \backslash (\vO - \varpi \vO)}} & \text{if } r < s \end{matrix}. \right.
\label{ArithNaturProj}
\end{equation}
	The formulas given in Corollary \ref{LIQuotient} show that the RHS in both cases is decreasing in $r$ once non-vanishing, hence attains its maximum when $c-\IntP{\cond(\pi)/2}-(s-r) = 0, s \geq r$, concluding the proof.

%The case for $\sigma_0$ is obtained by replacing $\B^{(\leq c-2m-2)}$ by $\B^{(2m+2)}$, using Corollary \ref{InvSpace1}.
\end{proof}

\section{Proof of Main Theorem}

	\subsection{Reduction to Period Bound}
	
	We begin with recall on the set-up. We fix a number field $\F$ with ring of integers $\vo$, ring of adeles $\A=\A_{\F}$ and a cuspidal representation $\pi$ of $\GL_2(\A)$ with central character $\omega=\omega_{\pi}$. There is a finite set $Q(\pi)$ of quaternion algebras $\B$ defined over $\F$, such that the Jacquet-Langlands lift of $\pi$ exists on $\gp{G}_{\B}$, the $\F$-group of invertible elements of $\B$. Precisely, these are the quaternion algebras whose set of ramification places is included in the set of places $v$ such that $\pi_v$ is square integrable. For each $\B \in Q(\pi)$, we fix an $\vo$-maximal order $\VO = \VO_{\B}$. For each place $v$ at which $\B$ is not ramified, we fix an isomorphism of $\F_v$-algebras $\delta_v: \B(\F_v) \simeq \Mat_2(\F_v)$ such that if $v=\vp < \infty$ then $\delta_{\vp}(\VO_{\vp}) = \Mat_2(\vo_{\vp})$. Any two sets of choices of $\VO$ and $\delta_{\vp}$ are conjugate to each other by elements in $\GL_2(\vo_{\vp})$ at all but finitely many finite places $\vp$.
	
	Let $(\E,\Omega)$ range over pairs consisting of a quadratic field extension of $\E/\F$ and a Hecke character $\Omega$ of $\E$ such that
	$$ \Omega \mid_{\A^{\times}} = \omega^{-1}, \quad \varepsilon(1/2, \pi_{\E} \otimes \Omega) = 1. $$
	Each pair $(\E,\Omega)$ determines a unique quaternion algebra $\B \in Q(\pi)$ over $\F$, belonging to the triple $(\pi, \E, \Omega)$ (see Definition \ref{QABelong} and the discussion thereafter). We write for $\gp{G}=\gp{G}_{\B}$ the $\F$-group of invertible elements of $\B$ and denote by $\pi' = \otimes_v' \pi_v' = \JL(\pi)$ the Jacquet-Langlands lifting of $\pi$ on $\gp{G}(\A)$. $\B$ contains $\E$ via some $\F$-embedding $\iota: \E \to \B$. $\iota(\E^{\times})$ defines an $\F$-sub-torus $\gp{T}$ of $\gp{G}$. 
	
	We specify a finite dimensional subspace $\sigma_v \subset \pi_v'$ and a $T_v \in \gp{G}(\F_v)$ at each place $v$, according to
\begin{equation}
	 S_{\F} = S_0 \sqcup S_{\infty,s} \sqcup S_{\infty,i} \sqcup S_s \sqcup S_{n,n} \sqcup S_{n,s} \sqcup S_{ur}
\label{PlacesPartition}
\end{equation}	
a partition of the set $S_{\F}$ of places of $\F$ defined as follows:
\begin{itemize}	
	\item[(0)] $S_0$ consists of all places $v$ at which $\B$ is ramified. For $v \in S_0$, we take $\sigma_v = \pi_v'$ equal to the whole space and $T_v=1$.
	\item[(1)] $S_{\infty,s}$ consists of all archimedean places $v$ at which both $\B$ and $\E$ are split. For $v \in S_{\infty,s}$, we take $\sigma_v$ to be the one-dimensional subspace spanned by $v_0$ specified in Lemma \ref{LocalEst1} (1). From $r$ with $\Cond(\Omega_v)^{1-\epsilon} \ll \norm[r]_v \ll \Cond(\Omega_v)^{1+\epsilon}$ specified also in Lemma \ref{LocalEst1} (1) and $\iota_r$ defined in (\ref{ArchCarEmbS}), we take $T_v$ so that
\begin{equation}
	\delta_v \circ (T_v^{-1} \iota_v T_v) = \iota_r.
\label{TransSpec}
\end{equation}
	\item[(2)] $S_{\infty,i}$ consists of all archimedean places $v$ at which $\B$ is split and $\E$ is inert. Necessarily $v$ is real. For $v \in S_{\infty,i}$, we take $\sigma_v$ to be the one-dimensional subspace spanned by the new vector. From $r$ with $\norm[r]_v \asymp \Cond(\Omega_v)^{1/2}$ specified in Lemma \ref{LocalEst1} (2) and $\iota_r$ defined in (\ref{ArchCarEmbNS}), we take $T_v$ so that (\ref{TransSpec}) holds. We also sub-divide
	$$ S_{\infty,i} = S_{\infty,i,d} \sqcup S_{\infty,i,c}, $$
	where $S_{\infty,i,d}$ resp. $S_{\infty,i,c}$ is the set of places $v \in S_{\infty,i}$ such that $\pi_v$ is resp. is not a discrete series representation.
	\item[(3)] $S_s$ consists of all finite places $v=\vp$ at which $\E$ is split and $\pi_{\vp}$ is ramified. Necessarily $\B$ is also split. For $\vp \in S_s$, we take $\sigma_{\vp}$ to be the subspace of the $\gp{K}_{\vp}$-representation generated by the new vector. From $r$ with $r \asymp \cond(\Omega_{\vp})/2$ specified in Lemma \ref{LocalEst2} and $\iota_r$ defined in (\ref{CarEmbS}), we take $T_{\vp}$ so that (\ref{TransSpec}) holds.
	\item[(4)] $S_{n,n}$ consists of all finite places $v=\vp$ at which $\B$ is split, $\E$ is non-split, and $\pi_{\vp}$ is ramified non-supercuspidal. Necessarily there exist characters $\mu_1,\mu_2$ of $\F_{\vp}^{\times}$ such that $\pi_{\vp}$ is induced from them. For $\vp \in S_{n,n}$ we take $\sigma_{\vp} = [\pi_{\vp}';c] \simeq [\pi_{\vp};c]$ (c.f. Definition \ref{InvSpDef}) with $c = \max(\cond(\mu_1),\cond(\mu_2),1)) (\leq \cond(\pi_{\vp}))$. From $r = \max(\IntCP{\cond(\Omega_{\vp})/e(\E_{\vp}/\F_{\vp})} - c, 0)$ specified in Lemma \ref{LocalEst5} and $\iota_r'$ defined in Remark \ref{AltEmb}, we take $T_{\vp}$ so that
\begin{equation}
	\delta_{\vp} \circ (T_{\vp}^{-1} \iota_{\vp} T_{\vp}) = \iota_r'.
\label{TransSpecBis}
\end{equation}
	\item[(5)] $S_{n,s}$ consists of all finite places $v=\vp$ at which $\B$ is split, $\E$ is non-split, and $\pi_{\vp}$ is supercuspidal. For $\vp \in S_{n,s}$, we take $\sigma_{\vp} = [\pi_{\vp}';c]$ with $c = \max(\IntP{\cond(\pi_{\vp})/2}+e(\E_{\vp}/\F_{\vp})-1, \IntCP{\cond(\pi_{\vp})/2}) (\leq \cond(\pi_{\vp}))$. From $r=\max(\IntCP{\cond(\Omega_{\vp})/e(\E_{\vp}/\F_{\vp})}-c+e(\E_{\vp}/\F_{\vp})-1,0)$ specified in Lemma \ref{LocalEst6} and $\iota_r$ defined in (\ref{CarEmbNS}), we take $T_{\vp}$ so that (\ref{TransSpec}) holds.
	\item[(6)] $S_{ur}$ consists of all finite places $v=\vp$ at which $\pi_{\vp}$ is unramified/spherical. Necessarily $\B$ is split. For $\vp \in S_{ur}$ we take $\sigma_{\vp}$ to be the one-dimensional subspace of spherical vectors. According as $\E_{\vp}$ is split resp. non-split, we have $r \asymp \cond(\Omega_{\vp})/2$ resp. $r = \IntCP{\cond(\Omega_{\vp})/e(\E_{\vp}/\F_{\vp})}$ specified in Lemma \ref{LocalEst2} resp. \ref{LocalEst4}, $\iota_r$ defined in (\ref{CarEmbS}) resp. (\ref{CarEmbNS}). We take $T_{\vp}$ so that (\ref{TransSpec}) holds. Note that only $S_{ur}$ contains infinitely many places and for all but finitely many such $\vp$ we have $T_{\vp} \in \gp{K}_{\vp}$.
\end{itemize}
\begin{definition}
	Write $e_{\vp} := e(\E_{\vp}/\F_{\vp})$ if $\E$ is non-split at the place $\vp < \infty$. Define
	$$ \Cond^{\sharp}(\Omega, \pi) := \Cond(\Omega) \cdot \prod_{v \in S_{\infty,i,c}} \Cond(\Omega_v)^{2\theta}; $$
	$$ \Cond^{\flat}(\Omega, \pi) := \sideset{}{_{\substack{v \mid \infty \\ v \notin S_0}}} \prod \Cond(\Omega_v) \cdot \sideset{}{_{\substack{\vp < \infty \\ \E_{\vp} \text{ split}}}} \prod \Cond(\Omega_{\vp}) \cdot \sideset{}{_{\substack{\vp < \infty \\ \E_{\vp} \text{ non-split}}}} \prod \Nr(\vp)^{2 \IntCP{\cond(\Omega_{\vp})/e_{\vp}}}. $$
\label{CondWaF}
\end{definition}
\begin{remark}
	By Lemma \ref{CondComp} and Proposition \ref{QuatCondBdInf} \& \ref{QuatCondBdFin}, we have
	$$ \Cond^{\flat}(\Omega,\pi) \asymp_{\pi} \Cond(\Omega). $$
\end{remark}
\begin{lemma}
	Recall the subspace version of the Waldspurger formula (\ref{WaFSV}). For the choice of $\sigma := \otimes_v' \sigma_v$ and $T=(T_v)_v \in \gp{G}(\A)$ given above, we have for any $\epsilon > 0$
	$$ \extnorm{ \sideset{}{_{v \mid \infty}} \prod \alpha_v(T_v.\sigma_v; \Omega_v, \iota_v) \cdot \sideset{}{_{\vp < \infty}} \prod \tilde{\alpha}_{\vp}(T_{\vp}.\sigma_{\vp}; \Omega_{\vp}, \iota_{\vp}) } \gg_{\F,\pi} \Dis(\E)^{-\frac{1}{2}-\epsilon} \Cond^{\sharp}(\Omega, \pi)^{-\frac{1}{2}-\epsilon}. $$
\label{LocalMain}
\end{lemma}
\begin{proof}
	Taking our measure normalization on $\gp{Z} \backslash \gp{T}$ and Lemma \ref{FixEmb} into account, this is just a re-statement of Lemma \ref{LocalEst0}, \ref{LocalEst1}, \ref{LocalEst2}, \ref{LocalEst4}, \ref{LocalEst5} and \ref{LocalEst6}.
\end{proof}

	\subsection{Amplification}
	
	The task is thus reduced to bounding the global period $\alpha(T.\sigma; \Omega; \iota)$. To this end, we set up the method of \emph{amplification}. Let $E$ be a positive real number to be optimized later and define
	$$ I_E := \{ \vp \in S_{ur} \ |\ E \leq q_{\vp}:= \Nr(\vp) \leq 2E, \E_{\vp} \text{ is split}, \cond(\Omega_{\vp}) = 0 \}, \quad M_E := \norm[I_E]. $$
	Our choice of $T_{\vp}$ implies that for $\vp \in I_E$ we have with $\iota_0$ defined in (\ref{CarEmbS})
	$$ \delta_{\vp} \circ (T_{\vp}^{-1} \iota_{\vp} T_{\vp}) = \iota_0. $$
	Let $[\varpi_{\vp}] \in \gp{T}(\F_{\vp})$ and $\chi_{\vp}$ the character of $\F_{\vp}^{\times}$ be such that
	$$ \delta_{\vp}([\varpi_{\vp}]^T) = a(\varpi_{\vp}), \quad \Omega_{\vp}([\varpi_{\vp}]) = \chi_{\vp}(\varpi_{\vp}) \quad \text{where} \quad [\varpi_{\vp}]^T:=T_{\vp}^{-1} [\varpi_{\vp}] T_{\vp}. $$
	By the invariance of the measure, we have the observation that for any $\phi \in \sigma, \vp \in I_E$
	$$ \ell(T.\phi; \Omega, \iota) = \int_{[\gp{T}]} T.\phi(t [\varpi_{\vp}]) \Omega(t [\varpi_{\vp}]) dt = \chi_{\vp}(\varpi_{\vp}) \ell(T[\varpi_{\vp}]. \phi; \Omega, \iota). $$
	Consequently, if we denote by $\Bas(\sigma)$ an/any orthogonal basis of $\sigma$, we get by Cauchy-Schwarz inequality
\begin{align*}
	\alpha(T.\sigma; \Omega, \iota) &= \sum_{\varphi \in \Bas(\sigma)} \frac{\extnorm{\ell(T.\varphi; \Omega, \iota)}^2}{\Norm[\varphi]_{[\gp{G}]}^2} = \sum_{\varphi \in \Bas(\sigma)} \frac{1}{\Norm[\varphi]_{[\gp{G}]}^2} \extnorm{ \frac{1}{M_E} \sum_{\vp \in I_E} \chi(\varpi_{\vp}) \ell(T [\varpi_{\vp}]^T.\varphi; \Omega, \iota) }^2 \\
	&\leq \sum_{\varphi \in \Bas(\sigma)} \frac{1}{\Norm[\varphi]_{[\gp{G}]}^2} \ell \left( T. \extnorm{ \frac{1}{M_E} \sum_{\vp \in I_E} \chi(\varpi_{\vp}) [\varpi_{\vp}]^T. \varphi }^2; 1, \iota \right) \cdot \Vol([\gp{T}], dt) \\
	&= \frac{2 \Lambda(1,\eta_{\E/\F})}{M_E^2} \sum_{\vec{\vp} \in I_E^2} \chi_{\vec{\vp}} \cdot \ell(T.\Phi(\vec{\vp}); 1,\iota),
\end{align*}
	where $\chi_{\vec{\vp}} := \chi_{\vp_1}(\varpi_{\vp_1}) \overline{\chi_{\vp_2}(\varpi_{\vp_2})}, \vec{\vp}:=(\vp_1,\vp_2)$; $1$ is the trivial character of $\gp{T}(\A)$; we have used the Tamagawa number computation (\ref{TamNumT}) and put
	$$ \Phi(\vp_1,\vp_2) := \sum_{\varphi \in \Bas(\sigma)} \Norm[\varphi]_{[\gp{G}]}^{-2} \cdot ([\varpi_{\vp_1}]^T.\varphi) \cdot (\overline{[\varpi_{\vp_2}]^T.\varphi}), $$
	which is a smooth function on $[\gp{G}]$ satisfying the following properties of invariance:
\begin{itemize}
	\item[(0)] It is invariant by $\gp{Z}(\A)$ the center of $\gp{G}(\A)$ and by $\gp{G}(\F_v)$ for any $v \in S_0$;
	\item[(1)] It is invariant by $\gp{K}_v$ for any $v \in S_{\infty,i}$;
	\item[(2)] It is invariant by $\gp{K}_{\vp}$ for any $\infty > \vp \notin \{ \vp_1,\vp_2 \}$;
	\item[(3)] At $\vp = \vp_1 \neq \vp_2$ or $\vp = \vp_2 \neq \vp_1$, it is invariant by $\gp{K}_{\vp} \cap [\varpi_{\vp}]^T \gp{K}_{\vp} ([\varpi_{\vp}]^T)^{-1}$;
	\item[(4)] At $\vp = \vp_1 = \vp_2$, it is invariant by $[\varpi_{\vp}]^T \gp{K}_{\vp} ([\varpi_{\vp}]^T)^{-1}$.
\end{itemize}
	Moreover, if $\B = \Mat_2(\F)$ it is of rapid decay, since $\varphi$ are cusp forms; otherwise $[\gp{G}]$ is compact. In any case, the Fourier inversion of $\Phi(\vp_1,\vp_2)$ with Fourier coefficients denoted by $C_{\vec{\vp}}(\cdot)$ ($\vec{\vp} := (\vp_1,\vp_2)$ for simplicity)
\begin{align*}
	\Phi(\vec{\vp}) &:= \Phi(\vp_1,\vp_2) = \sum_{\vartheta} C_{\vec{\vp}}(\sigma,\vartheta) \frac{\vartheta \circ \Nr_{\B}}{\Vol([\gp{G}])^{1/2}} + \sum_{\varrho \text{ cuspidal}} \sum_{\phi \in \Bas(\varrho)} C_{\vec{\vp}}(\sigma, \varrho) \phi \\
	&\quad + 1_{\B = \Mat_2} \sum_{\xi \in \widehat{\F^{\times} \backslash \A^{(1)}}} \sum_{f \in \Bas(\xi)} \int_{-\infty}^{\infty} C_{\vec{\vp}}(\sigma, \xi, i\tau) \eis(i\tau, f) \frac{d\tau}{4\pi}.
\end{align*}	
normally converges (c.f. \cite[Theorem 2.16]{Wu14}). Here
\begin{itemize}
	\item[(1)] $\Nr_{\B}$ is the (adelic) reduced norm of $\B$. $\vartheta$ runs over quadratic characters of the class group of $\F$, i.e., quadratic Hecke characters unramified at every $\vp < \infty$, which are trivial on $\Nr_{\B}(\B_v^{\times})$ at every $v \in S_0$.
	\item[(2)] $\Bas(\varrho)$ is an orthonormal basis of $\varrho$, which runs over cuspidal representations such that
	\begin{itemize}
		\item $\varrho_v$ is the trivial representation at every $v \in S_0$;
		\item $\varrho_v$ is spherical at every $v \in S_{\infty,i}$ resp. every $\infty > v=\vp \notin \{ \vp_1,\vp_2 \}$ resp. $v=\vp=\vp_1=\vp_2$, and $\phi_v$ is spherical with respect to $\delta_v^{-1}(\SO_2(\R))$ resp. $\gp{K}_{\vp}$ resp. $[\varpi_{\vp}]^T \gp{K}_{\vp} ([\varpi_{\vp}]^T)^{-1}$;
		\item $\cond(\varrho_{\vp}) \leq 1$ at $\vp = \vp_1 \neq \vp_2$ or $\vp = \vp_2 \neq \vp_1$ and $\phi_{\vp}$ is invariant by $\gp{K}_{\vp} \cap [\varpi_{\vp}]^T \gp{K}_{\vp} ([\varpi_{\vp}]^T)^{-1}$.
	\end{itemize}
	\item[(3)] $\xi$ is unramified at every $\vp < \infty$, extended to $\A^{\times}$ by triviality on $s_{\F}(\R_+)$ via a fixed section map $s_{\F}$ of the adelic norm map
	$$ \norm_{\A}: \A^{\times} \to \R_+, $$
	and $\Bas(\xi)$ is an orthonormal basis of the induced representation $\Ind_{\gp{B}(\A)}^{\GL_2(\A)}(\xi,\xi^{-1})$ such that
	\begin{itemize}
		\item $f_v$ is spherical at every $v \in S_{\infty,i}$ resp. every $\infty > v=\vp \notin \{ \vp_1,\vp_2 \}$ resp. $v=\vp=\vp_1=\vp_2$ with respect to $\delta_v^{-1}(\SO_2(\R))$ resp. $\gp{K}_{\vp}$ resp. $[\varpi_{\vp}]^T \gp{K}_{\vp} ([\varpi_{\vp}]^T)^{-1}$;
		\item $f_{\vp}$ is invariant by $\gp{K}_{\vp} \cap [\varpi_{\vp}]^T \gp{K}_{\vp} ([\varpi_{\vp}]^T)^{-1}$ at $\vp = \vp_1 \neq \vp_2$ or $\vp = \vp_2 \neq \vp_1$.
	\end{itemize}
\end{itemize}
	Since $[\gp{T}]$ is always compact, we can insert the Fourier inversion into $\ell(\cdot)$ and get
\begin{align}
	\ell(T.\Phi(\vec{\vp}); 1,\iota) &= \sum_{\vartheta} C_{\vec{\vp}}(\sigma,\vartheta) \ell \left( \frac{T. (\vartheta \circ \Nr_{\B})}{\Vol([\gp{G}])^{1/2}}; 1, \iota \right) + \sum_{\varrho \text{ cuspidal}} \sum_{\phi \in \Bas(\varrho)} C_{\vec{\vp}}(\sigma, \varrho, \phi) \ell(T.\phi; 1, \iota) \nonumber \\
	&\quad + 1_{\B = \Mat_2} \sum_{\xi \in \widehat{\F^{\times} \backslash \A^{(1)}}} \sum_{f \in \Bas(\xi)} \int_{-\infty}^{\infty} C_{\vec{\vp}}(\sigma, \xi, i\tau, f) \ell(T.\eis(i\tau, f); 1, \iota) \frac{d\tau}{4\pi} \nonumber \\
	&=: S_1(\vec{\vp}) + S_{\cusp}(\vec{\vp}) + 1_{\B = \Mat_2} \cdot S_{\Eis}(\vec{\vp}); \label{RegroupS}
\end{align}
\begin{equation}
	\alpha(T.\sigma; \Omega, \iota) \leq \Sigma_1 + \Sigma_{\cusp} + 1_{\B = \Mat_2} \cdot \Sigma_{\Eis}
\label{RegroupSigma}
\end{equation}
	where $C_{\vec{\vp}}(\cdot)$ are the Fourier coefficents of $\Phi(\vec{\vp})$ and we have put
	$$ \Sigma_* := \frac{2 \Lambda(1,\eta_{\E/\F})}{M_E^2} \sideset{}{_{\vec{\vp} \in I_E^2}} \sum \chi_{\vec{\vp}} S_*(\vec{\vp}), \quad * \in \{ 1, \cusp, \Eis \}. $$
\begin{remark}
	The procedure used above is similar to that given in \cite[\S 3]{Wu14}, but simpler thanks to the compactness of the domain of integration. Hence no re-arrange of the Fourier inversion is needed. But consequently, the treatment of $\Sigma_1$ and $\Sigma_{\Eis}$ need to be adapted from \cite{Wu14}.
\end{remark}

	\subsection{Bound of One Dimensional Part}
	
	We first bound $\Sigma_1$.
\begin{lemma}
\noindent (1) The contribution of $\vartheta$ in $S_1(\vec{\vp})$ hence in $\Sigma_1$ is non vanishing only if $\vartheta \in \{ 1, \eta_{\E/\F} \}$. Moreover, it is non-vanishing for $\vartheta = \eta_{\E/\F}$ only if
\begin{itemize}
	\item $S_0 \subset S_{\infty} := \{ v \in S_{\F} \ |\ v \mid \infty \}$, i.e., $\B$ is ramified only at infinite places;
	\item $\E$ is contained in the Hilbert class field of $\F$;
	\item $\E$ is non-split at every $v \in S_0$. 
\end{itemize}
\noindent (2) We have an upper bound for any $\epsilon > 0$
	$$ \norm[\Sigma_1] \ll_{\F,\pi,\epsilon} \Dis(\E)^{\epsilon}(M_E^{-1} + E^{-1+\epsilon}). $$
\label{GlobEst1}
\end{lemma}
\begin{proof}
\noindent (1) For any quadratic Hecke character $\vartheta$, we note that
	$$ \ell(\vartheta \circ \Nr_{\B}; 1, \iota) = \int_{\E^{\times} \A^{\times} \backslash \A_{\E}^{\times}} \vartheta(\Nr_{\E}(t)) dt = \Vol([\gp{T}]) \cdot 1_{\vartheta \mid_{\Nr_{\E}(\A_{\E}^{\times})} = 1}. $$
	By the class field theory, $\F^{\times} \Nr_{\E}(\A_{\E}^{\times}) \backslash \A^{\times} \simeq \mathrm{Gal}(\E/\F)$. Hence the condition $\vartheta \mid_{\Nr_{\E}(\A_{\E}^{\times})} = 1$ is equivalent to $\vartheta \in \{ 1, \eta_{\E/\F} \}$. Moreover, if $\vartheta = \eta_{\E/\F}$ does appear in the spectral decomposition of $\Phi(\vec{\vp})$, then it must be unramified at every $\vp < \infty$ hence a character of the class group of $\F$, and must be trivial on $\Nr_{\B}(\B_v^{\times})$ at every $v \in S_0$. Since $\E$ is contained in $\B$, it is non-split at every $v \in S_0$. But If $\infty > \vp \in S_0$ exists, then $\eta_{\vp} = 1$ since $\Nr_{\B}(\B_{\vp}^{\times}) = \F_{\vp}^{\times}$, hence $\E$ must both be split at $\vp$ and contained in $\B_{\vp}$ the division quaternion algebra over $\F_{\vp}$, which is a contradiction. Thus $S_0 \subset S_{\infty}$.
	
\noindent (2) We deduce from (1) that
\begin{align*}
	\Sigma_1 &= \frac{2 \Lambda(1,\eta_{\E/\F})}{M_E^2} \sideset{}{_{\vec{\vp} \in I_E^2}} \sum \chi_{\vec{\vp}} \sideset{}{_{\vartheta \in \{ 1, \eta \}}} \sum \frac{\Vol([\gp{T}])}{\Vol([\gp{G}])} \vartheta(\Nr_{\B}(T)) \cdot \Pairing{\Phi(\vec{\vp})}{\vartheta \circ \Nr_{\B}}_{[\gp{G}]} \\
	&= \frac{\Lambda(1,\eta_{\E/\F})^2}{M_E^2} \sideset{}{_{\vartheta \in \{ 1, \eta \}}} \sum \vartheta(\Nr_{\B}(T)) \sideset{}{_{\varphi \in \Bas(\sigma)}} \sum \Norm[\varphi]_{[\gp{G}]}^{-2} \extPairing{ \extnorm{ \sideset{}{_{\vp \in I_E}} \sum \chi_{\vp}(\varpi_{\vp}) [\varpi_{\vp}]^T.\varphi }^2 }{\vartheta \circ \Nr_{\B}}_{[\gp{G}]}.
\end{align*}
	Since the first function in the inner product is non-negative and the second one is a character, the contribution of $\vartheta = \eta$ is dominated by the contribution of $\theta = 1$. Thus
\begin{align*}
	\norm[\Sigma_1] &\leq \frac{2\norm[\Lambda(1,\eta_{\E/\F})]^2}{M_E^2} \sideset{}{_{\varphi \in \Bas(\sigma)}} \sum \Norm[\varphi]_{[\gp{G}]}^{-2} \extPairing{ \extnorm{ \sideset{}{_{\vp \in I_E}} \sum \chi_{\vp}(\varpi_{\vp}) [\varpi_{\vp}]^T.\varphi }^2 }{1}_{[\gp{G}]} \\
	&= \frac{2\norm[\Lambda(1,\eta_{\E/\F})]^2}{M_E^2} \sideset{}{_{\varphi \in \Bas(\sigma)}} \sum \sideset{}{_{\vec{\vp} \in I_E^2}} \sum \chi_{\vec{\vp}} \frac{\Pairing{[\varpi_{\vp_1}]^T.\varphi}{[\varpi_{\vp_2}]^T. \varphi}_{[\gp{G}]}}{\Norm[\varphi]_{[\gp{G}]}^2} \\
	&\leq \frac{2\norm[\Lambda(1,\eta_{\E/\F})]^2}{M_E} \norm[\Bas(\sigma)] + \frac{2\norm[\Lambda(1,\eta_{\E/\F})]^2}{M_E^2} \sideset{}{_{\varphi \in \Bas(\sigma)}} \sum \sideset{}{_{\substack{\vec{\vp} \in I_E^2 \\ \vp_1 \neq \vp_2 }}} \sum \frac{\extnorm{ \Pairing{[\varpi_{\vp_1}]^T.\varphi_{\vp_1}}{\varphi_{\vp_1}} }}{\Norm[\varphi_{\vp_1}]_{\vp_1}^2} \frac{\extnorm{ \Pairing{[\varpi_{\vp_2}]^T.\varphi_{\vp_2}}{\varphi_{\vp_2}} }}{\Norm[\varphi_{\vp_2}]_{\vp_2}^2},
\end{align*}
	where we recall the choice of local norms specified in Theorem \ref{WaF}. Since $\varphi_{\vp_1}$ and $\varphi_{\vp_2}$ are spherical by our choice, \cite[Theorem 4.6.5]{Bu98} implies
	$$ \frac{\extnorm{ \Pairing{[\varpi_{\vp}]^T.\varphi_{\vp}}{\varphi_{\vp}} }}{\Norm[\varphi_{\vp}]_{\vp}^2} = \frac{\Nr(\vp)^{-1/2}}{1+\Nr(\vp)^{-1}} \norm[\mathrm{tr}_{\vp}], \quad \vp \in \{ \vp_1, \vp_2 \} $$
	where $\mathrm{tr}_{\vp} = \alpha_{\vp} + \alpha_{\vp}^{-1}$ is the sum of Satake parameters of $\pi_{\vp}$. Iwaniec's trick \cite[Lemma 6.1]{Wu14} or \cite[(19.16)]{DFI02} implies that for any $\epsilon > 0$
	$$ \sideset{}{_{\substack{\vec{\vp} \in I_E^2 \\ \vp_1 \neq \vp_2 }}} \sum \frac{\extnorm{ \Pairing{[\varpi_{\vp_1}]^T.\varphi_{\vp_1}}{\varphi_{\vp_1}} }}{\Norm[\varphi_{\vp_1}]_{\vp_1}^2} \frac{\extnorm{ \Pairing{[\varpi_{\vp_2}]^T.\varphi_{\vp_2}}{\varphi_{\vp_2}} }}{\Norm[\varphi_{\vp_2}]_{\vp_2}^2} \ll E^{-1} \left( \sideset{}{_{\vp \in I_E}} \sum \norm[\mathrm{tr}_{\vp}] \right)^2 \ll_{\F,\pi,\epsilon} \frac{M_E^2}{E} E^{\epsilon}. $$
	We conclude by noting $\norm[\Bas(\sigma)] \ll_{\pi} 1$ and Siegel's upper bound for $\Lambda(1,\eta)$.
\end{proof}

	\subsection{Bound of Cuspidal Part}
	
	We then turn to $\Sigma_{\cusp}$. We shall bound it under an assumption that we state as a theorem.
\begin{theorem}
	Fix $\B$ a quaternion algebra defined over $\F$ with the set of ramification places $\Ram(\B)$. Let $S$ be a finite set of finite places of $\F$. Let $\E$ vary over quadratic field extensions of $\F$ contained in $\B$, with discriminant $\Dis(\E)$. We write $\eta_{\E/\F}$ for the quadratic Hecke character associated by the class field theory. We sum $\varrho$ over the cuspidal representations of $\gp{G}_{\B}(\A)$ under the following restrictions:
\begin{itemize}
	\item the central character of $\varrho$ is trivial;
	\item $\varrho_v = 1$ is the trivial representation at every $v \in \Ram(\B)$;
	\item $\varrho_{\vp}$ is either spherical or Steinberg at every $\vp \in S$, and spherical at every finite place $\vp \notin S \cup \Ram(\B)$.
\end{itemize}
	We denote by
\begin{itemize}
	\item $\JL(\varrho)$ the Jacquet-Langlands lifting of $\varrho$ on $[\GL_2]$;
	\item $\lambda_{\rho,\infty}$ the sum over $v \mid \infty, v \notin \Ram(\B)$ of eigenvalues belonging to the lowest weight vector of $\varrho_v$ of the operator $\Delta_v := - \Casimir_{\SL_2(\F_v)} - 2 \Casimir_{\gp{K}_v}$, where $\Casimir_{*}$ denotes the Casimir element of the relevant Lie group $*$ and we have identified $\B_v$ with $\Mat_2(\F_v)$ via $\delta_v$;
	\item $E(S) := \sideset{}{_{\vp \in S}} \prod \Nr(\vp)$.
\end{itemize}
	Then there exist constants $B > 0, 0 < \delta < 1/2$ such that for any $\epsilon > 0$ and sufficiently large $A > 0$
	$$ \sum_{\varrho} \frac{L(1/2,\JL(\varrho)) L(1/2, \JL(\varrho) \otimes \eta_{\E/\F})}{L(1, \JL(\varrho), \mathrm{Ad})} \lambda_{\varrho, \infty}^{-A} \ll_{\F,\B,\epsilon} (E(S)\Dis(\E))^{\epsilon} E(S)^{1+B} \Dis(\E)^{\frac{1}{2}-\delta}. $$
\label{AvgBound}
\end{theorem}
\begin{remark}
	The above theorem can be deduced from individual subconvex bounds for $L(1/2,\pi \otimes \eta_{\E/\F})$, \emph{effective} on the polynomial dependence on the analytic conductor of the fixed cuspidal representation $\pi$ on $[\PGL_2]$. Such a version is established in \cite[Theorem 2.3]{Wu3}. However,
\begin{itemize}
	\item[(1)] the exponent $B$ provided by the current version of \cite[Theorem 2.3]{Wu3} still needs to be optimized;
	\item[(2)] \cite[Theorem 2.3]{Wu3} deals with twists by general Hecke characters instead of quadratic ones $\eta_{\E/\F}$, and we feel that for the purpose of Theorem \ref{AvgBound} a relative trace formula approach would give a better $B$ (even possibly better $\delta$).
\end{itemize}
\end{remark}
\begin{lemma}
\noindent (1) For $\phi \in \Bas(\varrho)$ appearing in (\ref{RegroupS}), we have for any $\epsilon > 0$
	$$ \extnorm{\ell(T.\phi; 1, \iota)} \ll_{\pi, \theta, \epsilon} \sqrt{\frac{L(1/2, \JL(\varrho)_{\E})}{L(1, \JL(\varrho), \mathrm{Ad})}} \Dis(\E)^{-1/4} \cdot \left( \sideset{}{_{v \in S_{\infty,s}}} \prod \dim(\gp{K}_v.\phi_v)^{1/2} \right) \cdot \Cond^{\flat}(\Omega,\pi)^{-(1-2\theta)/4+\epsilon}, $$
where $\JL(\varrho)$ is the Jacquet-Langlands lifting of $\varrho$ on $[\GL_2]$.

\noindent (2) Consider the case $\vec{\vp}=(\vp_1,\vp_2)$ with $\vp_1 \neq \vp_2$. Assuming Theorem \ref{AvgBound}, we have for any $\epsilon > 0$
	$$ \norm[S_{\cusp}(\vec{\vp})] \ll_{\F, \pi, \epsilon} (E\Dis(\E)\Cond^{\flat}(\Omega,\pi))^{\epsilon} E^{1+B} \Dis(\E)^{-\frac{\delta}{2}} \Cond^{\flat}(\Omega,\pi)^{-\frac{1-2\theta}{4}}. $$
	Consequently, we get
	$$ \norm[\Sigma_{\cusp}] \ll_{\F, \pi, \epsilon} (E\Dis(\E)\Cond^{\flat}(\Omega,\pi))^{\epsilon} E^{1+B} \Dis(\E)^{-\frac{\delta}{2}} \Cond^{\flat}(\Omega,\pi)^{-\frac{1-2\theta}{4}}. $$
\label{GlobEst2}
\end{lemma}
\begin{proof}
\noindent (1) Since $\phi$ is $\intL^2$-normalized, (\ref{WaFSV}) implies
\begin{align*}
	\extnorm{ \ell(T.\phi; 1, \iota) }^2 &= \alpha(T.\phi; 1, \iota) = \frac{L(1/2, \JL(\varrho)_{\E})}{2 L(1, \JL(\varrho), \mathrm{Ad})} \cdot \\
	&\quad \sideset{}{_{v \mid \infty}} \prod \frac{L_v(1,\eta_v)}{\zeta_v(2)} \cdot \sideset{}{_{v \mid \infty}} \prod \alpha_v(T_v.\phi_v; 1, \iota_v) \cdot \sideset{}{_{\vp < \infty}} \prod \tilde{\alpha}_{\vp}(\phi_{\vp}; 1, \iota_{\vp}).
\end{align*}
	The desired bound then follows from the relevant local bounds place by place and our normalization of the measures on $\gp{Z}_v \backslash \gp{T}_v$:

\noindent (\rmnum{1}) At $v \in S_0$, $\varrho_v$ is the trivial representation, hence $\phi_v$ is a $\gp{G}_v$-invariant function. Thus
	$$ \alpha_v(T_v.\phi_v; 1, \iota_v) = \int_{\gp{Z}_v \backslash \gp{T}_v} dt = \Vol(\gp{Z}_v \backslash \gp{T}_v). $$
	
\noindent (\rmnum{2}) At $v \in S_{\infty,i}$ resp. $\infty > v=\vp \notin \{ \vp_1, \vp_2 \}$ resp. $v=\vp=\vp_1=\vp_2$, $\varrho_v$ is spherical and $\phi_v$ is spherical resp. spherical resp. $[\varpi_{\vp}]^T$-translate of a spherical vector. In the first two cases, we have by Lemma \ref{FixEmb}, according as $v \notin S_{n,n}$ or $v \in S_{n,n}$
	$$ \alpha_v(T_v.\phi_v; 1, \iota_v) = \alpha_v(\phi_v; 1, \delta_v^{-1} \circ \iota_r) \quad \text{or} \quad \alpha_{\vp}(T_{\vp}.\phi_{\vp}; 1, \iota_{\vp}) = \alpha_{\vp}(\phi_{\vp}; 1, \delta_{\vp}^{-1} \circ \iota_r'). $$
	In the last case, we have
	$$ \alpha_{\vp}(T_{\vp}.\phi_{\vp}; 1, \iota_{\vp}) = \alpha_{\vp}(\phi_{\vp}; 1, \delta_{\vp}^{-1} \circ \iota_0) = \alpha_{\vp}(([\varpi_{\vp}]^T)^{-1}.\phi_{\vp}; 1, \delta_{\vp}^{-1} \circ \iota_0) $$
	since $\delta_{\vp}([\varpi_{\vp}]^T) = a(\varpi_{\vp})$ lies in the stabilizer of $\iota_0$. We can thus apply either Lemma \ref{LocalEstSphUpBis} or Lemma \ref{LocalEstSphUp}, since the relevant vectors are spherical, and get
	$$ \frac{\alpha_v(T_v.\phi_v; 1, \iota)}{\Vol(\gp{Z}_v \backslash \gp{T}_v)} \left\{ \begin{matrix} \ll_{\epsilon} \Cond(\Omega_v)^{-(1-2\theta)/2+\epsilon} & \text{if } v \in S_{\infty,i} \\ \ll_{\pi_{\vp},\epsilon} \Cond(\Omega_{\vp})^{-(1-2\theta)/2+\epsilon} & \text{if } v=\vp<\infty, \E_{\vp} \text{ not ramified}, \cond(\Omega_{\vp}) > 0 \\ \ll_{\pi_{\vp}, \epsilon} \Nr(\vp)^{-(1-2\theta)\IntCP{\cond(\Omega_{\vp})/2}+\epsilon} & \text{if } v=\vp<\infty, \E_{\vp} \text{ ramified}, \cond(\Omega_{\vp}) > 0 \\ = \frac{\zeta_{\vp}(2) L_{\vp}(1/2,\JL(\varrho)_{\E_{\vp}})}{L_{\vp}(1,\eta_{\vp})L_{\vp}(1,\JL(\varrho_{\vp}), \mathrm{Ad})} & \text{if } v=\vp<\infty, \cond(\Omega_{\vp}) = 0 \end{matrix} \right. . $$
	
\noindent (\rmnum{3}) At $v \in S_{\infty,s}$, $\phi_v$ runs over an orthonormal basis of $\gp{K}_v$-isotypic vectors of $\varrho_v$. By (\ref{SPFRIndividual}), \cite[Lemma 6.8 \& Corollary 6.9]{Wu14} are applicable and give
	$$ \alpha_v(T_v.\phi_v; 1, \iota_v) \ll_{\theta,\epsilon} \dim(\gp{K}_v.\phi_v) \Cond(\Omega_v)^{-(1-2\theta)/2+\epsilon}. $$
	
\noindent (\rmnum{4}) At $v=\vp=\vp_1 \neq \vp_2$ or $\vp=\vp_2 \neq \vp_1$, $\E_{\vp}$ is split. Hence $\varrho_{\vp}$ is either spherical or Steinberg. By (\ref{SPFRIndividual}), \cite[Lemma 6.12]{Wu14} is applicable and gives
	$$ \tilde{\alpha}_{\vp}(T_{\vp}.\phi_{\vp}; 1, \iota_{\vp}) \ll_{\theta,\epsilon} 1. $$
	
\noindent (2) Denote by $\Delta_{\B,\infty} = \sideset{}{_v} \sum \Delta_v$ the sum of local Laplacian operators at $v \mid \infty, v \notin \Ram(\B)$ defined in Theorem \ref{AvgBound}. Write $\lambda_{\phi,\infty}$ to be the eigenvalue belonging to $\phi$ of $\Delta_{\B,\infty}$. Inserting the bound in (1) and applying the Cauchy-Schwarz inequality we get for $l \gg 1$ (say odd integer)
\begin{align*}
	\norm[S_{\cusp}(\vec{\vp})] &\ll_{\pi,\epsilon} \Dis(\E)^{-\frac{1}{4}} \Cond^{\flat}(\Omega,\pi)^{-\frac{1-2\theta}{4}} \sideset{}{_{\varrho}} \sum \sideset{}{_{\phi}} \sum \lambda_{\phi,\infty}^{\frac{l+1}{2}} C_{\vec{\vp}}(\sigma,\varrho,\phi) \cdot \sqrt{\frac{L(1/2, \JL(\varrho)_{\E})}{L(1, \JL(\varrho), \mathrm{Ad})}} \lambda_{\phi,\infty}^{-\frac{l}{2}} \\
	&\leq \Dis(\E)^{-\frac{1}{4}} \Cond^{\flat}(\Omega,\pi)^{-\frac{1-2\theta}{4}} \left( \sideset{}{_{\varrho}} \sum \sideset{}{_{\phi}} \sum \lambda_{\phi,\infty}^{l+1} C_{\vec{\vp}}(\sigma,\varrho,\phi)^2 \right)^{\frac{1}{2}} \left( \sideset{}{_{\varrho}} \sum \sideset{}{_{\phi}} \sum \frac{L(1/2, \JL(\varrho)_{\E})}{L(1, \JL(\varrho), \mathrm{Ad})} \lambda_{\phi,\infty}^{-l} \right)^{\frac{1}{2}} \\
	&\ll \Dis(\E)^{-\frac{1}{4}} \Cond^{\flat}(\Omega,\pi)^{-\frac{1-2\theta}{4}} \extNorm{ \Delta_{\B,\infty}^{\frac{l+1}{2}} \Phi(\vec{\vp}) }_2 \cdot \left( \sideset{}{_{\varrho}} \sum \frac{L(1/2, \JL(\varrho)_{\E})}{L(1, \JL(\varrho), \mathrm{Ad})} \lambda_{\varrho,\infty}^{-\frac{l}{2}} \right)^{\frac{1}{2}}.
\end{align*}
	Note that $\Delta_{\B,\infty}^{\frac{l+1}{2}} \Phi(\vec{\vp})$ is a finite sum of
	$$ \sideset{}{_{\varphi \in \Bas(\sigma)}} \sum \Norm[\varphi]_{[\gp{G}]}^{-2} D_1.[\varpi_{\vp_1}]^T.\varphi \cdot \overline{ D_2.[\varpi_{\vp_2}]^T.\varphi } $$
	for certain pairs of differential operators $(D_1,D_2)$, and that
	$$ \extNorm{ D_1.[\varpi_{\vp_1}]^T.\varphi \cdot \overline{ D_2.[\varpi_{\vp_2}]^T.\varphi  } }_2 \leq \extNorm{ D_1.[\varpi_{\vp_1}]^T.\varphi }_4 \extNorm{ D_2.[\varpi_{\vp_2}]^T.\varphi }_4 = \extNorm{ D_1.\varphi }_4 \extNorm{ D_2.\varphi }_4, $$
	we deduce $\Norm[ \Delta_{\B,\infty}^{\frac{l+1}{2}} \Phi(\vec{\vp}) ]_2 \ll_{\pi} 1$. We thus conclude the first bound by Theorem \ref{AvgBound}. For the (diagonal) case $\vp_1 = \vp_2$, we proceed similarly to get a bound without the appearance of $E$ by applying Theorem \ref{AvgBound} with $S = \emptyset$. We then deduce the second bound.
\end{proof}

	\subsection{Bound of Eisensetein Part}
	
	We finally deal with $\Sigma_{\Eis}$. As in the cuspidal case, we need an assumption which we state as a theorem.
\begin{theorem}
	We fix a section $s_{\F}$ of the adelic norm map $\norm_{\A}: \A^{\times} \to \R_{>0}$. Let $\xi$ run over the Hecke characters of $\F^{\times} \backslash \A^{\times}$ trivial on the image of $s_{\F}$ and unramified at every finite place $v=\vp < \infty$. Let $\E$ vary over quadratic field extensions of $\F$ with discriminant $\Dis(\E)$. We write $\eta_{\E/F}$ for the quadratic Hecke character associated by the class field theory. We denote by $\lambda_{\xi,s,\infty}$ the eigenvalue of $- \sideset{}{_{v \mid \infty}} \sum \Casimir_{\SL_2(\F_v)}$ on 
	$$ \sideset{}{_{v \mid \infty}} \prod \pi( \xi_v \norm_v^s, \xi_v^{-1} \norm_v^{-s} ), \quad \pi( \xi_v \norm_v^s, \xi_v^{-1} \norm_v^{-s} ) = \Ind_{\gp{B}(\F_v)}^{\GL_2(\F_v)} ( \xi_v \norm_v^s, \xi_v^{-1} \norm_v^{-s} ). $$
	Then there exists a constant $0 < \delta' < 1/2$ such that for any $\epsilon$ and sufficiently large $A>0$ we have
	$$ \sideset{}{_{\xi}} \sum \int_{-\infty}^{\infty} \extnorm{\frac{L(1/2+i\tau,\xi)L(1/2+i\tau,\xi \eta)}{L(1+2i\tau,\xi)}}^2 \lambda_{\xi,s,\infty}^{-A} \frac{d\tau}{4\pi} \ll_{\F,\epsilon} \Dis(\E)^{\frac{1}{2}-\delta'+\epsilon}. $$
\label{AvgBoundEis}
\end{theorem}
\begin{remark}
	The above theorem can be deduced from individual subconvex bounds for $L(1/2,\xi \eta)$ provided by \cite{Wu2} with $\delta'=(1-2\theta)/8$. As in the cuspidal case, we feel that a relative trace formula approach would give a better $\delta'$. In general, one can expect that $\delta' \geq \delta$ provided by Theorem \ref{AvgBound}. Hence the major obstruction comes from the cuspidal spectrum other than the continuous one.
\end{remark}
\begin{lemma}
\noindent (1) For $f \in \Bas(\xi)$ appearing in (\ref{RegroupS}), we have for any $s \in i\R$ and $\epsilon > 0$
	$$ \extnorm{ \ell(T.\eis(s,f); 1, \iota) } \ll_{\F,\pi,\epsilon} \extnorm{ \frac{L(1/2+s,\xi_{\E})}{L(1+2s,\xi^2)} } \Dis(\E)^{-1/4} \cdot \left( \sideset{}{_{v \mid \infty}} \prod \dim(\gp{K}_v.f_v)^{1/2} \right) \cdot \Cond^{\flat}(\Omega,\pi)^{-1/4+\epsilon}, $$
	where $L(1/2+s,\xi_{\E}) = L(1/2+s,\xi)L(1/2+s,\xi\eta)$.
	
\noindent (2) Consider the case $\vec{\vp}=(\vp_1,\vp_2)$ with $\vp_1 \neq \vp_2$. Assuming Theorem \ref{AvgBoundEis}, we have for any $\epsilon > 0$
	$$ \norm[ S_{\Eis}(\vec{\vp}) ] \ll_{\F,\pi,\epsilon} (\Dis(\E) \Cond^{\flat}(\Omega,\pi))^{\epsilon} \Dis(\E)^{-\frac{\delta'}{2}} \Cond^{\flat}(\Omega,\pi)^{-\frac{1}{4}}. $$
	Consequently, we get
	$$ \norm[ \Sigma_{\Eis} ] \ll_{\F,\pi,\epsilon} (\Dis(\E) \Cond^{\flat}(\Omega,\pi))^{\epsilon} \Dis(\E)^{-\frac{\delta'}{2}} \Cond^{\flat}(\Omega,\pi)^{-\frac{1}{4}}. $$
\label{GlobEst3}
\end{lemma}
\begin{proof}
\noindent (1) Since $f$ is normalized, we have $\norm[ \ell(T.\eis(s,f); 1, \iota) ]^2 = \norm[ \ell(T.\eis(s,f); 1, \iota) ]^2 / \extNorm{ \eis(s,f) }_{\Eis}^2$, to which we can apply Proposition \ref{WieF}. We obtain the desired bound by the same argument as in the proof of Lemma \ref{GlobEst2} (1), with $\theta=0$.

\noindent (2) Denote by $\Delta_{\infty} = \Delta_{\Mat_2,\infty}$ defined in the proof of Lemma \ref{GlobEst2} (2). Write $\lambda_{f,s,\infty}$ to be the eigenvalue belonging to $f_s$ of $\Delta_{\infty}$. Inserting the bound in (1) and applying the Cauchy-Schwarz inequality we get for $l \gg 1$
\begin{align*}
	\norm[ S_{\Eis}(\vec{\vp}) ] &\ll_{\pi,\epsilon} \Dis(\E)^{-\frac{1}{4}} \Cond^{\flat}(\Omega,\pi)^{-\frac{1}{4}} \sideset{}{_{\xi}} \sum \sideset{}{_{f \in \Bas(\xi)}} \sum \int_{-i\infty}^{i\infty} \lambda_{f,s,\infty}^{\frac{l+1}{2}} C_{\vec{\vp}}(\sigma,\xi,s,f) \cdot \extnorm{\frac{L(1/2+s,\xi_{\E})}{L(1+2s,\xi^2)}} \lambda_{f,s,\infty}^{-\frac{l}{2}} \frac{ds}{4\pi i} \\
	&\leq \Dis(\E)^{-\frac{1}{4}} \Cond^{\flat}(\Omega,\pi)^{-\frac{1}{4}} \cdot \left( \sideset{}{_{\xi}} \sum \sideset{}{_{f \in \Bas(\xi)}} \sum \int_{-i\infty}^{i\infty} \lambda_{f,s,\infty}^{l+1} C_{\vec{\vp}}(\sigma,\xi,s,f)^2 \frac{ds}{4\pi i} \right)^{\frac{1}{2}} \\
	&\quad \cdot \left( \sideset{}{_{\xi}} \sum \sideset{}{_{f \in \Bas(\xi)}} \sum \int_{-i\infty}^{i\infty} \extnorm{ \frac{L(1/2+s,\xi_{\E})}{L(1+2s,\xi^2)} }^2 \lambda_{f,s,\infty}^{-l} \frac{ds}{4\pi i} \right)^{\frac{1}{2}} \\
	&\ll \Dis(\E)^{-\frac{1}{4}} \Cond^{\flat}(\Omega,\pi)^{-\frac{1}{4}} \cdot \extNorm{ \Delta_{\infty}^{\frac{l+1}{2}} \Phi(\vec{\vp}) }_2 \cdot \left( \sideset{}{_{\xi}} \sum \int_{-i\infty}^{i\infty} \extnorm{ \frac{L(1/2+s,\xi_{\E})}{L(1+2s,\xi^2)} }^2 \lambda_{\xi,s,\infty}^{-\frac{l}{2}} \frac{ds}{4\pi i} \right)^{\frac{1}{2}}.
\end{align*}
	We conclude by arguing as in the proof of Lemma \ref{GlobEst2} (2) with Theorem \ref{AvgBound} replaced by Theorem \ref{AvgBoundEis}.
\end{proof}

	\subsection{Finalizing the Proof}
	
	Inserting Lemma \ref{GlobEst1}, \ref{GlobEst2} and \ref{GlobEst3} into (\ref{RegroupSigma}) and taking Lemma \ref{LocalMain} into account, we get
\begin{align*}
	&\quad \extnorm{ \frac{L(1/2,\pi_{\E} \otimes \Omega)}{L(1,\pi,\mathrm{Ad})} } \Dis(\E)^{-\frac{1}{2}} \Cond^{\sharp}(\Omega,\pi)^{-\frac{1}{2}} \ll_{\F,\pi} \norm[\Sigma_1] + \norm[\Sigma_{\cusp}] + 1_{\B = \Mat_2} \cdot \norm[\Sigma_{\Eis}] \\
	&\ll_{\F,\pi,\epsilon} (E\Dis(\E)\Cond^{\flat}(\Omega,\pi))^{\epsilon} \cdot \left\{ (M_E^{-1} + E^{-1}) + E^{1+B} \Dis(\E)^{-\frac{\delta}{2}} \Cond^{\flat}(\Omega,\pi)^{-\frac{1-2\theta}{4}} \right. \\
	&\quad \left. + 1_{\B = \Mat_2} \cdot \Dis(\E)^{-\frac{\delta'}{2}} \Cond^{\flat}(\Omega,\pi)^{-\frac{1}{4}} \right\}.
\end{align*}
	It is expected that $\delta' \geq \delta$, hence we can ignore the contribution from the last line. Under the assumption
	$$ M_E \gg E/\log E $$
	for the choice of $E$ that will be determined in a moment, we get the desired bound
	$$ \norm[L(1/2,\pi_{\E} \times \Omega)] \ll_{\F,\pi,\epsilon} (\Dis(\E)\Cond^{\flat}(\Omega,\pi))^{\epsilon} \cdot \Dis(\E)^{\frac{1}{2} - \frac{\delta}{2(2+B)}} \Cond^{\sharp}(\Omega,\pi)^{\frac{1}{2}} \Cond^{\flat}(\Omega,\pi)^{-\frac{1-2\theta}{4(2+B)}} $$
	for the optimal choice
	$$ E = \Dis(\E)^{\frac{\delta}{2(2+B)}} \Cond^{\flat}(\Omega,\pi)^{\frac{1-2\theta}{4(2+B)}}. $$

\section{Abundance of Split Places}

	\subsection{Statement of a Probabilistic Result}

	For any field $\F$, denote by $\mathrm{Et}_2(\F)$ the set of quadratic separable extensions of $\F$. If $\F$ is a number field with ring of integers $\vo_{\F}$ and degree $d_{\F} := [\F:\Q]$, then its characteristic is $0$. Hence we have an identification
	$$ \F^{\times}/(\F^{\times})^2 \simeq \mathrm{Et}_2(\F), \quad [u] \mapsto [t].\F[\sqrt{u}] := \F[\sqrt{u}], $$
	which equip $\mathrm{Et}_2(\F)$ with the structure of a $2$-group. Take $U < \F^{\times}/(\F^{\times})^2 $ any finitely generated subgroup. Then $U$ is finite and acts on $\mathrm{Et}_2(\F)$ by
	$$ U \times \mathrm{Et}_2(\F) \to \mathrm{Et}_2(\F), \quad ([t], \F[\sqrt{u}]) \mapsto \F[\sqrt{tu}]. $$
	We can choose $\tilde{U} \subset \F^{\times}$ a system of representatives of $U$ and denote by
\begin{equation}
	S_U := \left\{ \vp \nmid 2 \text{ prime ideal of } \F \quad \middle| \quad \exists u \in \tilde{U}, 2 \nmid \mathrm{ord}_{\vp}(u \vo_{\F}) \right\}, \quad \Dis(U) := \sideset{}{_{\vp \in S_U}} \prod \Nr(\vp),
\label{UInv}
\end{equation}
	where $\mathrm{ord}_{\vp}(I)$ is the exponent of $\vp$ in the decomposition of the (fractional) ideal $I$.
\begin{lemma}
	For any $\E \in \mathrm{Et}_2(\F)$ and $[u] \in U$, we have
	$$ 2^{-3d_{\F}} \Dis(U)^{-1} \leq \Dis([u].\E) / \Dis(\E) \leq 2^{3d_{\F}} \Dis(U). $$
	Here $\Dis(\E)$ denotes the absolute discriminant of $\E$ and $\Dis(\E_0) = \Dis(\F)^2$ by convention for the split extension $\E_0 := \F \times \F$.
\label{DisComp}
\end{lemma}
\begin{proof}
	For any prime ideal $\vp$, the natural inclusion $\F \hookrightarrow \F_{\vp}$ induces a map $\mathrm{Et}_2(\F) \to \mathrm{Et}_2(\F_{\vp})$ as well as a group homomorphism $U \to U_{\vp}$ compatible with the group actions, i.e., a commutative diagram
	$$ \begin{matrix} U & \times & {\rm Et}_2(\F) & \to & {\rm Et}_2(\F) \\ \downarrow & & \downarrow & & \downarrow \\ U_{\vp} & \times & {\rm Et}_2(\F_{\vp}) & \to & {\rm Et}_2(\F_{\vp}) \end{matrix}. $$
	The desired bound follows from the local ones provided by the following two cases.
	
\noindent (1) $\vp \nmid 2$. Let $\varpi_{\vp}$ be a uniformizer of $\F_{\vp}$ and take a $\varepsilon_{\vp} \in \vo_{\vp}^{\times} - (\vo_{\vp}^{\times})^2$. By \cite[\S 2.1.6]{GGP69}, we have
\begin{equation}
	\mathrm{Et}_2(\F_{\vp}) = \{ \F_{\vp} \times \F_{\vp}, \F_{\vp}[\sqrt{\varepsilon_{\vp}}] \} \sqcup \{ \F_{\vp}[\sqrt{\varpi_{\vp}}], \F_{\vp}[\sqrt{\varepsilon_{\vp} \varpi_{\vp}}] \}.
\label{OddEt2}
\end{equation}
	If $\vp \notin S_U$, then $U_{\vp}$ is contained in $\{ (\F_{\vp}^{\times})^2, \varepsilon_{\vp} (\F_{\vp}^{\times})^2 \}$. Consequently, the relative discriminants
	$$ \Dis([u].\E \otimes_{\F} \F_{\vp} / \F_{\vp}) = \Dis(\E \times_{\F} \F_{\vp} / \F_{\vp}) \in \{ \vo_{\vp}, \vp \vo_{\vp} \} $$
	are equal with each other. If $\vp \in S_U$, then we have
	$$ \Dis([u].\E \otimes_{\F} \F_{\vp} / \F_{\vp}) / \Dis(\E \times_{\F} \F_{\vp} / \F_{\vp}) \in \{ \vp^{-1}\vo_{\vp}, \vo_{\vp}, \vp\vo_{\vp} \}. $$
	
\noindent (2) $\vp \mid 2$. For any $\F_{\vp} \times \F_{\vp} \neq \E_{\vp} \in \mathrm{Et}_2(\F_{\vp})$, we can always find $x \in \F_{\vp}$ with $x \in \vo_{\vp} - \vp^2 \vo_{\vp}$ such that $\E_{\vp} = \F_{\vp}[\sqrt{x}]$. If $\vO_{\vp}$ denotes the ring of integers of $\E_{\vp}$, then we always have
	$$ \vo_{\vp}[\sqrt{x}] \subset \vO_{\vp} \subset \frac{1}{2}\vo_{\vp}[\sqrt{x}]. $$
	If $x \notin \vp \vo_{\vp}$, then taking the dual modules (with respect to the trace) we get
	$$ \vo_{\vp}[\sqrt{x}] \subset \left( \frac{1}{2} \vo_{\vp}[\sqrt{x}] \right)^* \subset \vO_{\vp}^* \subset \vo_{\vp}[\sqrt{x}]^* \subset \frac{1}{2} \vo_{\vp}[\sqrt{x}] \subset \frac{1}{2} \vO_{\vp} \quad \Rightarrow \quad \Dis(\E_{\vp} / \F_{\vp}) \mid 4 \vo_{\vp}. $$
	Similarly if $x \in \vp \vo_{\vp} - \vp^2 \vo_{\vp}$, then taking the dual modules we get
	$$ \frac{\vo_{\vp}[\sqrt{x}]}{\sqrt{x}} \subset \left( \frac{1}{2} \vo_{\vp}[\sqrt{x}] \right)^* \subset \vO_{\vp}^* \subset \vo_{\vp}[\sqrt{x}]^* \subset \frac{\vo_{\vp}[\sqrt{x}]}{2\sqrt{x}} \subset \frac{\vO_{\vp}}{2\sqrt{x}} \quad \Rightarrow \quad \Dis(\E_{\vp}/\F_{\vp}) \mid 4\vp\vo_{\vp}. $$
	Taking the worse case of the two, we deduce that for every $\E \in \mathrm{Et}_2(\F)$
	$$ \Dis(\E \otimes_{\Q} \Q_2 / \Q_2) \mid 4^{d_{\F}} 2^{\sideset{}{_{\vp \mid 2}} \sum f_{\vp}} \Z_2 \mid 2^{3d_{\F}} \Z_2, $$
	where $f_{\vp}$ is the degree of the residual field $f_{\vp} = [\vo_{\vp}/\vp : \Z/2\Z]$.
\end{proof}
	
	For any $E > 0$ and $\E \in \mathrm{Et}_2(\F)$, write
	$$ \pi_{\F}(E) = \left\{ \vp \text{ (integral) prime ideal of } \F \quad \middle| \quad \Nr(\vp) \leq E \right\}, $$
	$$ \mathrm{Sp}_{\F}(E,\E) = \left\{ \vp \text{ prime ideal of } \F \quad \middle| \quad \Nr(\vp) \leq E, \vp \text{ split in } \E \right\}. $$
\begin{theorem}
	Fix $U$ with $\norm[U] \geq 2$. For any orbit $U.\E$ of $U$ and any constants $\epsilon, \delta > 0$ there exist two constants $c_j=c_j(\epsilon, \delta, U)$ so that for any $E > \max(c_1,c_2\Dis(\E)^{\delta})$, we have
	$$ \frac{1}{\norm[U]} \sideset{}{_{[u] \in U}} \sum \extnorm{ \mathrm{Sp}_{\F}(E, [u].\E) } > \left( \frac{1}{2} - \frac{1}{2\norm[U]} - \epsilon \right) \norm[ \pi_{\F}(E) ]. $$
\label{SpPLBd}
\end{theorem}
\noindent Before entering into the proof of the main estimation above, we deduce a consequence which is of main interest for our purpose.
\begin{corollary}
	Under the same conditions as in Theorem \ref{SpPLBd} (with possibly two different constants $c_j'$), for any $E > \max(c_1,c_2 \Dis(\E)^{\delta})$, there exists at most one extension $\E'$ in the orbit $U.\E$ such that
	$$ \norm[\mathrm{Sp}_{\F}(E,\E')] \leq \left( \frac{1}{8} - \epsilon \right) \norm[\pi_{\F}(E)]. $$
	In other words, the probability that $\E' \in U.\E$ does not satisfy the above inequality is at least $1-1/\norm[U]$, which can be chosen as large as we want.
\label{SpPAbd}
\end{corollary}
\begin{proof}
	As $U$ is a $2$-group, for any $1 \neq [u] \in U$, $\{ 1, [u] \}$ is a subgroup, to which we can apply Theorem \ref{SpPLBd} with the corresponding constants denoted by $c_j([u])$. Taking
	$$ c_1 = \sideset{}{_{[u] \in U}} \sup c_1([u]), \quad c_2 = \left( 2^{3d_{\F}} \sideset{}{_{1 \neq [u] \in U}} \max \Dis(\{ 1, [u] \}) \right)^{\delta} \sideset{}{_{[u] \in U}} \sup c_2([u]), $$
	Lemma \ref{DisComp} implies that we can apply Theorem \ref{SpPLBd} to $\{ 1, [u] \}$ for any $1 \neq [u] \in U$ and to any $\E' \in U.\E$. Hence at most one of $\E', [u].\E'$ satisfies the inequality. If one $\E'$ does satisfy the inequality, then all other $[u].\E'$ with $[u] \neq 1$ fails the inequality.
\end{proof}

	\subsection{Algebraic Lemmas}
	
	We shall need some algebraic facts to reduce our main estimation to the Chebotarev Theorem for some appropriate number field. Any finitely generated $U < \F^{\times}/(\F^{\times})^2$ is automatically a finite dimensional vector space over $\mathbb{F}_2 = \Z/2\Z$. Let $\{ u_1, \cdots, u_r \} \subset \F^{\times}$ whose image in $\F^{\times} / (\F^{\times})^2$ form a basis of $U$. Define for any integer $1 \leq k \leq r$
	$$ \F_k := \F[\sqrt{u_k}], \quad \E_k := \F[\sqrt{u_1}, \cdots, \sqrt{u_k}]. $$
\begin{lemma}
\noindent (1) $[\E_k : \F] = 2^k$.

\noindent (2) The obvious homomorphism between Galois groups
	$$ \mathrm{Gal}(\E_k / \F) \to \sideset{}{_{j=1}^k} \prod \mathrm{Gal}(\F_j / \F) \simeq (\Z/2\Z)^k $$
	is an isomorphism.
	
\noindent (3) All quadratic sub-extension of $\F$ in $\E_k$ are of the form $\F[\sqrt{u}]$, where $u$ is the product of the elements in a subset of $\{ u_1, \cdots, u_k \}$.
\label{GalId}
\end{lemma}
\begin{proof}
	One (simple but fastidious) proof works by induction on $k$. We leave the details to the reader.
\end{proof}

\begin{lemma}
	Let $\vp \nmid 2$ be a prime ideal of $\vo_{\F}$. Let $u \in \vo - \vp$ which is not a square in $\F^{\times}$.

\noindent (1) If $u$ is a square modulo $\vp$, then $\vp$ is split in $\F[\sqrt{u}$. The Artin symbol $(\vp, \F[\sqrt{u}]/\F)$ is the identity element in $\mathrm{Gal}(\F[\sqrt{u}]/\F)$.

\noindent (2) If $u$ is not a square modulo $\vp$, then $\vp$ is inert in $\F[\sqrt{u}]$. The Artin symbol $(\vp, \F[\sqrt{u}]/\F)$ is the non-trivial element in $\mathrm{Gal}(\F[\sqrt{u}]/\F)$.
\end{lemma}
\begin{proof}
	This is how the Artin symbol generalizes the Legendre symbol. We omit the proof.
\end{proof}

	\subsection{Proof of the Main Estimation}

\begin{definition}
\noindent (1) Let $S_{\fin}$ be the set of prime ideals of $\vo_{\F}$. The \emph{Legendre symbol} for $\F$ is defined as
	$$ \mathrm{Et}_2(\F) \times S_{\fin} \to \{ 0, \pm 1 \}, \quad (\E,\vp) \mapsto \Legendre{\E}{\vp} := \left\{ \begin{matrix} 0 & \text{if } \E \text{ is ramified in } \E, \\ 1 & \text{if } \E \text{ is split in } \E, \\ -1 & \text{if } \E \text{ is inert in } \E. \end{matrix} \right. $$

\noindent (2) When $\vp$ is fixed, the above map factors through the natural (surjective) homomorphism $\mathrm{Et}_2(\F) \to \mathrm{Et}_2(\F_{\vp})$, denoted by
	$$ \Legendre{\cdot}{\vp}: \mathrm{Et}_2(\F_{\vp}) \to \{ 0, \pm 1 \}, \quad \E_{\vp} \mapsto \Legendre{\E_{\vp}}{\vp}. $$
	
\noindent (3) Write $\E_0 = \F \times \F$ to be the split quadratic extension of $\F$. For any $[u] \in \F^{\times}/(\F^{\times})^2$ or $u \in \F_{\vp}^{\times} / (\F_{\vp}^{\times})^2$ we write
	$$ \Legendre{[u]}{\vp} := \Legendre{[u].\E_0}{\vp}. $$
\label{LegendreNF}
\end{definition}
\begin{lemma}
\noindent (1) We have an explicit description
	$$ \Legendre{[u]}{\vp} = \left\{ \begin{matrix} 0 & \text{if there exists a representative } u \in \vp\vo_{\vp}-\vp^2\vo_{\vp} \\ 1 & \text{if there exists a representative } u \in (\vo_{\vp}^{\times})^2 \\ -1 & \text{if there exists a representative } u \in \vo_{\vp}^{\times} - (\vo_{\vp}^{\times})^2 \end{matrix} \right. . $$
	
\noindent (2) Let $U < \F^{\times}/(\F^{\times})^2$ be finite(ly generated) and recall $S_U$ in (\ref{UInv}). Then for any $\vp \nmid 2, \vp \notin S_U$, we have
	$$ \Legendre{[u].\E_{\vp}}{\vp} = \Legendre{[u]}{\vp} \cdot \Legendre{\E_{\vp}}{\vp}, \quad \forall [u] \in U, \E_{\vp} \in \mathrm{Et}_2(\F_{\vp}). $$
\label{MultLSymbol}
\end{lemma}
\begin{proof}
\noindent (1) This follows from the fact that a representative $u \in \vo_{\vp} - \vp^2 \vo_{\vp}$ always exists, and the definition.

\noindent (2) Under the assumption we have
	$$ \Legendre{[u]}{\vp} \neq 0. $$
	We verify the equality by the explicit description of $\mathrm{Et}_2(\F_{\vp})$ in (\ref{OddEt2}).
\end{proof}

\begin{corollary}
	Let $\vp \nmid 2, \vp \notin S_U$ and $\E \in \mathrm{Et}_2(\F)$. If $\vp$ is ramified in $\E$, then it is ramified in $[u].\E$ for every $[u] \in U$, hence
	$$ \sideset{}{_{[u] \in U}} \sum \Legendre{[u].\E}{\vp} = 0 = \Legendre{\E}{\vp}. $$
	Otherwise, $\vp$ is not ramified in any $[u].\E$ and we have
	$$ \sideset{}{_{[u] \in U}} \sum \Legendre{[u].\E}{\vp} = \left\{ \begin{matrix} \norm[U] \cdot \Legendre{\E}{\vp} & \text{if } \vp \text{ is split in } [u].\E_0 \text{ for every } [u] \in U \\ 0 & \text{otherwise} \end{matrix} \right. . $$
\label{OrthoV}
\end{corollary}
\begin{proof}
	Both assertions follow easily from the previous lemma and
	$$ \sideset{}{_{[u] \in U}} \sum \Legendre{[u]}{\vp} = \left\{ \begin{matrix} \norm[U] & \text{if } \Legendre{\cdot}{\vp} \text{ is the trivial character of } U \\ 0 & \text{otherwise} \end{matrix} \right. . $$
	In fact, $\Legendre{\cdot}{\vp}$ coincides with the trivial character of $U$ if and only if the image $U_{\vp}$ of $U$ under the natural homomorphism $\F^{\times} / (\F^{\times})^2 \to \F_{\vp}^{\times} / (\F_{\vp}^{\times})^2$ is trivial, which is equivalent to that $\vp$ is split in every $[u].\E_0$.
\end{proof}

\begin{lemma}
	For any $\epsilon > 0$, there is a constant $c_1>0$ depending only on $U$ and $\epsilon$ such that
	$$ \extnorm{ \sideset{}{_{[u] \in U}} \bigcap \mathrm{Sp}_{\F}(E, [u].\E_0) } < \left( \frac{1}{\norm[U]} + \epsilon \right) \norm[\pi_{\F}(E)], \quad \forall E \geq c_1. $$
\label{JtSp}
\end{lemma}
\begin{proof}
	For any $\vp \notin S_U, \vp \nmid 2$, $\vp$ is unramified in $\F_j$ for every $1 \leq j \leq r$. Hence $\vp$ is unramified in $\E_r$. By definition, the image of $(\vp, \E_r / \F)$ under the isomorphism of Galois groups in Lemma \ref{GalId} (2) is the product of $(\vp, \F_j / \F)$. From Lemma \ref{MultLSymbol} (2), we deduce that
	$$ \vp \in \sideset{}{_{[u] \in U}} \bigcap \mathrm{Sp}_{\F}(E, [u].\E_0) \quad \Leftrightarrow \quad (\vp, \F_j/\F) = 1, \forall 1 \leq j \leq r \quad \Leftrightarrow \quad (\vp, \E_r/\F) = 1. $$
	Applying any effective version of the Chebotarev Theorem \cite{LMO79} to $\E_r / \F$, we conclude.
\end{proof}

\begin{proof}[Proof of Theorem \ref{SpPLBd}]
	First notice that for any $\E \in \mathrm{Et}_2(\F)$ we can write
	$$ \norm[\mathrm{Sp}_{\F}(E,\E)] = \frac{1}{2} \norm[\pi_{\F}(E)] + \frac{1}{2} N(E,\E) + O(\log \Dis(\E)), \quad N(E,\E) := \sideset{}{_{\Nr(\vp) \leq E}} \sum \Legendre{\E}{\vp}. $$
	By Corollary \ref{OrthoV} and Lemma \ref{JtSp}, we have
\begin{align*}
	\sideset{}{_{[u] \in U}} \sum N(E,[u].\E) &= \frac{1}{[U]} \sideset{}{_{\substack{ \vp \mid 2 \text{ or } \vp \in S_U \\ \Nr(\vp) < E}}} \sum \sideset{}{_{[u] \in U}} \sum \Legendre{[u].\E}{\vp} + \frac{1}{[U]} \sideset{}{_{\substack{ \vp \nmid 2 \text{ and } \vp \notin S_U \\ \Nr(\vp) < E}}} \sum \sideset{}{_{[u] \in U}} \sum \Legendre{[u].\E}{\vp} \\
	&\geq - \norm[S_U]-d_{\F} - \extnorm{ \left\{ \vp \middle| \vp \nmid 2, \vp \notin S_U, \vp \in \mathrm{Sp}_{\F}(E,[u].\E) \text{ for all } [u] \in U \right\} } \\
	&\geq -O(\norm[S_U] + d_{\F}) - \left( \frac{1}{\norm[U]} + \epsilon \right) \norm[\pi_{\F}(E)],
\end{align*}
	which clearly implies the desired bound.
\end{proof}

\bibliographystyle{acm}
	
\bibliography{mathbib}

\address{\quad \\ Han WU \\ MA C3 604 \\ EPFL SB MATHGEOM TAN \\ CH-1015, Lausanne \\ Switzerland \\ wuhan1121@yahoo.com}
	
\end{document}